\documentclass[12pt,a4paper]{article}
\usepackage[left=2.5cm,right=2.5cm,top=2.5cm,bottom=2.5cm]{geometry}
\usepackage{eqnarray}
\usepackage{xcolor}

\usepackage{amsmath,amssymb,amsthm}
\usepackage{youngtab}
\usepackage[all]{xy}
\xyoption{2cell}
\usepackage{etoolbox}
\usepackage{tikz}
\usetikzlibrary{positioning,
    calc,
    arrows,
    decorations.markings,
    decorations.pathreplacing,
    backgrounds,
    patterns}
\UseAllTwocells

\pgfdeclarelayer{background}
\pgfdeclarelayer{foreground}
\pgfsetlayers{background,main,foreground}

\tikzset{comod/.style={rectangle, minimum width=25pt, minimum height=10pt, draw, inner sep=1pt}}

\makeatletter
\newcommand*{\extralabel}[1]{
  \refstepcounter{equation}
  \ltx@label{#1}
}%
\makeatother

\newcommand\vc[1]{\begin{tabular}{@{}l}#1\end{tabular}}
\newcommand\ket[1]{\ensuremath{| #1 \rangle}}
\newcommand\bra[1]{\ensuremath{\langle #1 |}}
\newcommand{\sdag}{{\ensuremath{\scriptscriptstyle{\dag}}}}
\newcommand\ignore[1]{\ignorespacesafterend}
\newtheorem{theorem}{Theorem}
\numberwithin{theorem}{section}

\newtheorem{lemma}[theorem]{Lemma}
\theoremstyle{definition}

\newtheorem{varremark}[theorem]{Remark}

\def\rBell{\ensuremath{\mathrm{Bell}}}
\def\W{\ensuremath{\mathrm{W}}}
\newcommand*\sxto[1]{\xrightarrow{\smash{#1}}}
\newcommand\pstar{{\vphantom{*}}}

\renewcommand\matrix[1]
{\hspace{-1.5pt}\left( \hspace{-2pt} \raisebox{-0pt}{$\begin{array}{@{\extracolsep{-9pt}}cccc}#1\end{array}$}  \hspace{0pt} \right)\hspace{-1.5pt}}

\newcommand\tinymatrix[1]
{\left( \hspace{-2pt} \renewcommand\thickspace{\kern2pt} \scriptstyle\begin{smallmatrix} #1 \end{smallmatrix} \hspace{-2pt} \right)}

\makeatletter
\def\calign@preamble{%
   &\hfil\strut@
    \setboxz@h{\@lign$\m@th\displaystyle{##}$}%
    \ifmeasuring@\savefieldlength@\fi
    \set@field
    \hfil
    \tabskip\alignsep@
}
\let\cmeasure@\measure@
\patchcmd\cmeasure@{\divide\@tempcntb\tw@}{}{}{}
\patchcmd\cmeasure@{\divide\@tempcntb\tw@}{}{}{}
\patchcmd\cmeasure@{\ifodd\maxfields@
  \global\advance\maxfields@\@ne
  \fi}{}{}{}    
\newenvironment{calign}
{%
  \let\align@preamble\calign@preamble
  \let\measure@\cmeasure@
  \align
}
{%
  \endalign
}  
\makeatother

\newcommand{\C}{\mathbb{C}}

\newcommand\id{\mathrm{id}}

\newcommand{\opname}[1]{\operatorname{#1}}
\newcommand{\catname}[1]{\ensuremath{\boldsymbol{\opname{#1}}}}
\newcommand\cat[1]{{\catname{#1}}}

\renewcommand{\_}[0]{\nobreakdash--\hspace{0pt}}

\renewcommand\matrix[1]
{\hspace{-1.5pt}\left( \hspace{-2pt} \raisebox{-0pt}{$\begin{array}{@{\extracolsep{-5.5pt}\,}cccc}#1\end{array}$}  \hspace{0pt} \right)\hspace{-1.5pt}}

\renewcommand\tinymatrix[1]
{\left( \hspace{-2pt} \renewcommand\thickspace{\kern2pt} \scriptstyle\begin{smallmatrix} #1 \end{smallmatrix} \hspace{-2pt} \right)}

\newenvironment{tz}
{\,\,\begin{aligned}\begin{tikzpicture}}
{\end{tikzpicture}\end{aligned}\,\,}

\def\fillA{blue!20}
\def\fillB{green!20}
\def\fillC{red!20}
\def\fillD{yellow!30}
\def\fillcomp{\fillA}

\newenvironment{tikzequation}{\begin{equation}\begin{aligned}\begin{tikzpicture}}{\end{tikzpicture}\end{aligned}\end{equation}\ignorespacesafterend}

\makeatletter
\tikzset{nomorepostaction/.code={\let\tikz@postactions\pgfutil@empty}}
\makeatother
\ignore{\tikzset{double arrow scope/.style={every node/.style={font=\scriptsize}, every path/.style={
        double, shorten >=1pt,
        postaction={nomorepostaction,decorate,
                    decoration={markings,mark=at position 1 with {\arrow[semithick]{>}}}
                   }
        }}}}
\tikzset{double arrow scope/.style={every node/.style={font=\scriptsize}, every path/.style={
        double, -new double arrowhead}}}

\newcommand\doubleto[0]{\mathbin{\ensuremath{\begin{tikzpicture}
   \node (A) at (0,0) [inner xsep=1.2pt] {};
   \node (B) at (0.5,0) [inner xsep=1pt] {};
   \draw [double,
    decoration={markings,mark=at position 1 with {\arrow[semithick]{>}}},
    postaction={decorate}, shorten >=1pt]
        (A) to (B);
\end{tikzpicture}}}}
\renewcommand\to[0]{\mathbin{\ensuremath{\begin{tikzpicture}
   \node (A) at (0,0) [inner xsep=1.2pt] {};
   \node (B) at (0.5,0) [inner xsep=1pt] {};
   \draw [decoration={markings,mark=at position 1 with {\arrow[semithick]{>}}},
    postaction={decorate}, shorten >=0.5pt
    ] (A) to (B);
\end{tikzpicture}}}}
\newcommand\xdoubleto[1]{\mathbin{\begin{tikzpicture}[baseline={([yshift=-3pt]
current bounding box.south)}]
    \node (A) at (0,0) [inner xsep=0pt, inner ysep=-1pt, minimum width=0.2cm] {\ensuremath{\scriptstyle #1 \strut}};
    \draw [double,->]
        ([xshift=-2.5pt] A.south west)
        to ([xshift=3pt] A.south east);
\end{tikzpicture}}}

\newcommand\xto[1]{\mathbin{\begin{tikzpicture}[baseline={([yshift=-3pt]
current bounding box.south)}]
    \node (A) at (0,0) [inner xsep=0pt, inner ysep=2pt, minimum width=0.2cm] {$\scriptstyle #1$};
    \draw [decoration={markings,mark=at position 1 with {\arrow[semithick]{>}}},
    postaction={decorate}, shorten >=0pt]
        ([xshift=-2.5pt] A.south west)
        to ([xshift=3pt] A.south east);
\end{tikzpicture}}}

\tikzset{keyvertexcolour/.initial=black}
\tikzset{vertex colour/.style={keyvertexcolour={#1}}}
\newlength\vertexradius
\newlength\innerradius
\def\halfanglesep{24}
\def\tripleanglesep{24}
\def\sideangle{30}
\def\tripleanglesep{40}
\makeatletter
\pgfdeclareshape{Vertex}
{
    \savedanchor\centerpoint
    {
        \pgf@x=0pt
        \pgf@y=0pt
    }
    \anchor{n1}
    {
        \pgfextractx{\pgf@x}{\pgfpointpolar{90+\halfanglesep}{\innerradius}}
        \pgfextracty{\pgf@y}{\pgfpointpolar{90+\halfanglesep}{\innerradius}}
    }
    \anchor{n2}
    {
        \pgfextractx{\pgf@x}{\pgfpointpolar{90-\halfanglesep}{\innerradius}}
        \pgfextracty{\pgf@y}{\pgfpointpolar{90-\halfanglesep}{\innerradius}}
    }
    \anchor{n}
    {
        \pgfextractx{\pgf@x}{\pgfpointpolar{90}{\innerradius}}
        \pgfextracty{\pgf@y}{\pgfpointpolar{90}{\innerradius}}
    }
    \anchor{ne1}
    {
        \pgfextractx{\pgf@x}{\pgfpointpolar{\sideangle+\halfanglesep}{\innerradius}}
        \pgfextracty{\pgf@y}{\pgfpointpolar{\sideangle+\halfanglesep}{\innerradius}}
    }
    \anchor{ne2}
    {
        \pgfextractx{\pgf@x}{\pgfpointpolar{\sideangle-\halfanglesep}{\innerradius}}
        \pgfextracty{\pgf@y}{\pgfpointpolar{\sideangle-\halfanglesep}{\innerradius}}
    }
    \anchor{ne}
    {
        \pgfextractx{\pgf@x}{\pgfpointpolar{\sideangle}{\innerradius}}
        \pgfextracty{\pgf@y}{\pgfpointpolar{\sideangle}{\innerradius}}
    }
    \anchor{se1}
    {
        \pgfextractx{\pgf@x}{\pgfpointpolar{-\sideangle+\halfanglesep}{\innerradius}}
        \pgfextracty{\pgf@y}{\pgfpointpolar{-\sideangle+\halfanglesep}{\innerradius}}
    }
    \anchor{se2}
    {
        \pgfextractx{\pgf@x}{\pgfpointpolar{-\sideangle-\halfanglesep}{\innerradius}}
        \pgfextracty{\pgf@y}{\pgfpointpolar{-\sideangle-\halfanglesep}{\innerradius}}
    }
    \anchor{s1}
    {
        \pgfextractx{\pgf@x}{\pgfpointpolar{-90-\halfanglesep}{\innerradius}}
        \pgfextracty{\pgf@y}{\pgfpointpolar{-90-\halfanglesep}{\innerradius}}
    }
    \anchor{s2}
    {
        \pgfextractx{\pgf@x}{\pgfpointpolar{-90+\halfanglesep}{\innerradius}}
        \pgfextracty{\pgf@y}{\pgfpointpolar{-90+\halfanglesep}{\innerradius}}
    }
    \anchor{sA}
    {
        \pgfextractx{\pgf@x}{\pgfpointpolar{-90-\tripleanglesep}{\innerradius}}
        \pgfextracty{\pgf@y}{\pgfpointpolar{-90-\tripleanglesep}{\innerradius}}
    }
    \anchor{sB}
    {
        \pgfextractx{\pgf@x}{\pgfpointpolar{-90}{\innerradius}}
        \pgfextracty{\pgf@y}{\pgfpointpolar{-90}{\innerradius}}
    }
    \anchor{sC}
    {
        \pgfextractx{\pgf@x}{\pgfpointpolar{-90+\tripleanglesep}{\innerradius}}
        \pgfextracty{\pgf@y}{\pgfpointpolar{-90+\tripleanglesep}{\innerradius}}
    }
    \anchor{s}
    {
        \pgfextractx{\pgf@x}{\pgfpointpolar{-90}{\innerradius}}
        \pgfextracty{\pgf@y}{\pgfpointpolar{-90}{\innerradius}}
    }
    \anchor{sw1}
    {
        \pgfextractx{\pgf@x}{\pgfpointpolar{180+\sideangle+\halfanglesep}{\innerradius}}
        \pgfextracty{\pgf@y}{\pgfpointpolar{180+\sideangle+\halfanglesep}{\innerradius}}
    }
    \anchor{sw2}
    {
        \pgfextractx{\pgf@x}{\pgfpointpolar{180+\sideangle-\halfanglesep}{\innerradius}}
        \pgfextracty{\pgf@y}{\pgfpointpolar{180+\sideangle-\halfanglesep}{\innerradius}}
    }
    \anchor{sw}
    {
        \pgfextractx{\pgf@x}{\pgfpointpolar{180+\sideangle}{\innerradius}}
        \pgfextracty{\pgf@y}{\pgfpointpolar{180+\sideangle}{\innerradius}}
    }
    \anchor{nw1}
    {
        \pgfextractx{\pgf@x}{\pgfpointpolar{180-\sideangle+\halfanglesep}{\innerradius}}
        \pgfextracty{\pgf@y}{\pgfpointpolar{180-\sideangle+\halfanglesep}{\innerradius}}
    }
    \anchor{nw2}
    {
        \pgfextractx{\pgf@x}{\pgfpointpolar{180-\sideangle-\halfanglesep}{\innerradius}}
        \pgfextracty{\pgf@y}{\pgfpointpolar{180-\sideangle-\halfanglesep}{\innerradius}}
    }
    \anchor{nw}
    {
        \pgfextractx{\pgf@x}{\pgfpointpolar{180-\sideangle}{\innerradius}}
        \pgfextracty{\pgf@y}{\pgfpointpolar{180-\sideangle}{\innerradius}}
    }
    \anchor{center}{\centerpoint}
    \anchorborder{\centerpoint}
    \backgroundpath
    {
        \pgfkeysgetvalue{/tikz/keyvertexcolour}{\vcol}
        \begin{pgfonlayer}{foreground}
        \pgfsetfillcolor{\vcol}
        \pgfpathcircle{\pgfpoint{0cm}{0cm}}{\vertexradius}
        \pgfusepath{fill}
        \pgfsetstrokecolor{black}
        \pgfsetlinewidth{0.8pt}
        \pgfpathcircle{\pgfpoint{0cm}{0cm}}{\vertexradius}
        \pgfusepath{stroke}
        \end{pgfonlayer}
    }
}
\makeatother

\def\nwangle{180-\sideangle}
\def\neangle{\sideangle}
\def\swangle{180+\sideangle}
\def\seangle{-\sideangle}

\newbool{partpath}
\pgfkeys{/tikz/double partial rendering/.style={
    postaction={
      decorate,
      decoration={
        show path construction,
        moveto code={
          \global\booltrue{partpath}
        },
        lineto code={
          \ifx\tikzinputsegmentfirst\tikzinputsegmentlast
            \ifbool{partpath}{\global\boolfalse{partpath}}{\global\booltrue{partpath}}
          \fi
          \ifbool{partpath}{
            \draw [double] (\tikzinputsegmentfirst) -- (\tikzinputsegmentlast);
          }{}
        },
        curveto code={
          \ifbool{partpath}{
            \draw [double] (\tikzinputsegmentfirst) .. controls (\tikzinputsegmentsupporta) and (\tikzinputsegmentsupportb) .. (\tikzinputsegmentlast);
          }{}
        },
        closepath code={
          \ifbool{partpath}{
            \draw (\tikzinputsegmentfirst) -- (\tikzinputsegmentlast);
          }{}
        }
      }
    }
  }
}
\pgfkeys{/tikz/partial rendering/.style={
    postaction={
      decorate,
      decoration={
        show path construction,
        moveto code={
          \global\booltrue{partpath}
        },
        lineto code={
          \ifx\tikzinputsegmentfirst\tikzinputsegmentlast
            \ifbool{partpath}{\global\boolfalse{partpath}}{\global\booltrue{partpath}}
          \fi
          \ifbool{partpath}{
            \draw (\tikzinputsegmentfirst) -- (\tikzinputsegmentlast);
          }{}
        },
        curveto code={
          \ifbool{partpath}{
            \draw (\tikzinputsegmentfirst) .. controls (\tikzinputsegmentsupporta) and (\tikzinputsegmentsupportb) .. (\tikzinputsegmentlast);
          }{}
        },
        closepath code={
          \ifbool{partpath}{
            \draw (\tikzinputsegmentfirst) -- (\tikzinputsegmentlast);
          }{}
        }
      }
    }
  }
}

\pgfkeys{%
  /tikz/on layer/.code={
    \pgfonlayer{#1}\begingroup
    \aftergroup\endpgfonlayer
    \aftergroup\endgroup
  },
  /tikz/node on layer/.code={
    \pgfonlayer{#1}\begingroup
    \expandafter\def\expandafter\tikz@node@finish\expandafter{\expandafter\endgroup\expandafter\endpgfonlayer\tikz@node@finish}%
  }
}

\def\myfill{red!40}

\newlength\triplesep
\newlength\triplelinewidth
\tikzset{triple/.style={line width=\triplelinewidth,black,
    preaction={
        preaction={draw,line width=2\triplesep+3\triplelinewidth,black},
        draw,line width=2\triplesep+\triplelinewidth,white}
    }
}

\definecolor{green}{rgb}{0.2,1,0.2}

\title{Higher Quantum Theory}
\author{Jamie Vicary
\\
\texttt{jamie.vicary@cs.ox.ac.uk}
\\[10pt]
\begin{tabular}{c}
Centre for Quantum Technologies, University of Singapore
\\
and Department of Computer Science, University of Oxford
\\[0pt]
\end{tabular}}
\date{July 19, 2012}

\usepackage[normalem]{ulem}

\begin{document}

\maketitle
\abstract{We propose a 2-categorical formalism for describing classical information, quantum systems, and their interactions, based on the principle that classical information can be encoded as correlations between quantum systems. Applying this in the 2\-category of 2\_Hilbert spaces recovers ordinary quantum theory. The formalism gives a simple, graphical way to describe the specification and implementation of certain quantum procedures, which we use to investigate  quantum teleportation, dense coding, complementarity and quantum erasure, verifying our results computationally using a software package.}

\section{Introduction}

\subsection{Overview}

\subsubsection*{The formalism}

This paper presents a new formalism for describing flows of quantum and classical information. The key physical insight, which is by no means new, is that classical information can be encoded in correlations between quantum systems. The key mathematical insight is that the theory of symmetric monoidal 2\-categories serves as an elegant and powerful language for describing and manipulating these correlations.

The formalism introduces a family of fundamental operations, each of which carries a different physical interpretation. These operations are defined by diagrams, in which \textit{regions} represent stores of classical information, \textit{lines} represent physical systems, and \textit{vertices} represent dynamics. We list them here, along with their interpretations:
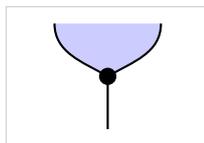
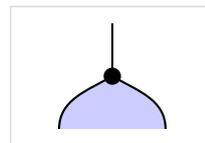
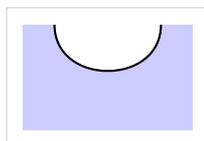
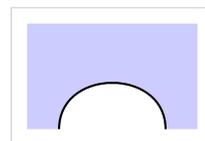
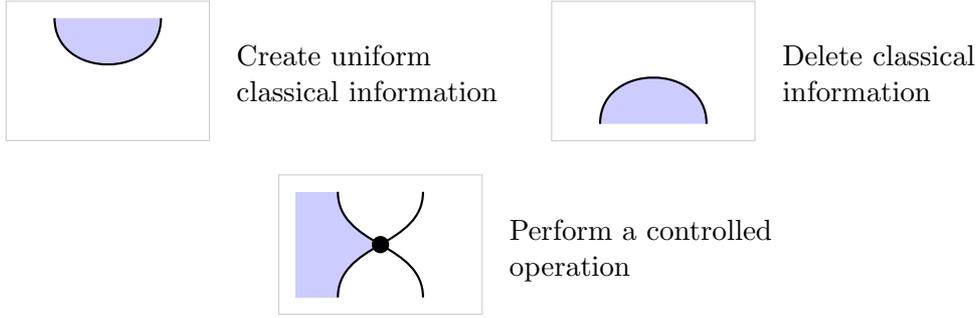
\def\aascale{1.4}
\def\aaspace{\hspace{30pt}}
\setlength\fboxsep{0pt}
\def\sep{5pt}
\def\innersep{3pt}
\def\littlegap{12pt}
\def\boxmargin{0.15cm}
\newcommand{\centerdia}[1]{#1}
\newcommand\newtwocell[2]{\begin{aligned}
\begin{tikzpicture}[scale=\aascale]
    #1
    \draw [black!20]
        ([xshift=-\boxmargin, yshift=-\boxmargin] current bounding box.south west)
        rectangle
        ([xshift=\boxmargin, yshift=\boxmargin] current bounding box.north east);
\end{tikzpicture}
\end{aligned}
\hspace{10pt} \makebox[80pt][l]{\vc{#2}}}
\newcommand\separatetwocells{\\[\sep]}
\allowdisplaybreaks[2]
\begin{calign}
\nonumber
\newtwocell{
    \draw [white] (0.2,1) to (1.8,1);
    \node (v) [Vertex] at (1,1) {};
    \draw [fill=\fillA, thick] (0.5,1.5)
        to [out=down, in=\nwangle] (v.center)
        to [out=\neangle, in=down] (1.5,1.5);
    \draw [thick] (1,0.5) to (1,1);
}
{Perform a\\measurement}
\hspace{\littlegap}&\hspace{\littlegap}
\newtwocell{
    \draw [white] (0.2,-1) to (1.8,-1);
    \node (v) [Vertex] at (1,-1) {};
    \draw [fill=\fillA, thick] (0.5,-1.5)
        to [out=up, in=\swangle] (v.center)
        to [out=\seangle, in=up] (1.5,-1.5);
    \draw [thick] (1,-0.5) to (1,-1);
}
{Encode classical\\information}
\separatetwocells
\nonumber
\newtwocell{
    \draw [fill=\fillA, draw=none] (0.2,-0.5)
        to (0.5,-0.5)
        to [out=down, in=down, looseness=1.5] (1.5, -0.5)
        to (1.8,-0.5)
        to (1.8,-1.5)
        to (0.2,-1.5)
        to (0.2,-0.5);
    \draw [thick] (0.5,-0.5)
        to [out=down, in=down, looseness=1.5] (1.5,-0.5);
}
{Copy classical\\information}
\hspace{\littlegap}&\hspace{\littlegap}
\newtwocell{
    \draw [fill=\fillA, draw=none] (0.2,0.5)
        to (0.5,0.5)
        to [out=up, in=up, looseness=1.5] (1.5,0.5)
        to (1.8,0.5)
        to (1.8,1.5)
        to (0.2,1.5)
        to (0.2,0.5);
    \draw [thick] (0.5,0.5)
        to [out=up, in=up, looseness=1.5] (1.5,0.5);
}
{Compare classical\\information}
\separatetwocells
\nonumber
\newtwocell{
    \draw [white] (0.2,1) to (1.8,1);
    \draw [fill=\fillA, thick] (0.5,1.5)
        to [out=down, in=down, looseness=1.5] (1.5,1.5);
    \draw [thick, white] (1,0.5) to (1,1);
}
{Create uniform\\classical information}
\hspace{\littlegap}&\hspace{\littlegap}
\newtwocell{
    \draw [white] (0.2,-1) to (1.8,-1);
    \draw [fill=\fillA, thick] (0.5,-1.5)
        to [out=up, in=up, looseness=1.5] (1.5,-1.5);
    \draw [thick, white] (1,-0.5) to (1,-1);
}
{Delete classical\\information}
\end{calign}
\begin{equation*}
\newtwocell{
    \draw [white] (0.6,1) to (2.2,1);
    \draw [fill=\fillA, draw=none]
        (0.6,0.5)
        rectangle
        (1,1.5);
    \draw [fill=\fillA, draw=none] (1,1.5)
        to [out=down, in=\nwangle] (1.4,1)
        to [out=\swangle, in=up] (1,0.5);
    \draw [thick]
        (1,0.5)
        to [out=up, in=\swangle] (1.4,1)
        to [out=\nwangle, in=down] (1,1.5);
    \draw [thick]
        (1.8,0.5)
        to [out=up, in=\seangle] (1.4,1)
        to [out=\neangle, in=down] (1.8,1.5);
    \node [Vertex] at (1.4,1) {};
}
{Perform a controlled\\operation}
\end{equation*}
In these diagrams, time flows upwards. Lines bounding the same shaded region represent physical systems whose states are correlated, in a way which is governed by the classical information stored in the associated region.

The form of each of these diagrams is consistent with its interpretation. For example, the diagram representing measurement produces a region of classical information. Ultimately, the physical interpretations we assign to these 2\-cells arise from the equations we require them to satisfy, as we examine in Section~\ref{sec:formalism}. These amount to saying that classical information is \emph{topological}.

We make these diagrams rigorous using the theory of \textit{2\-categories}~\cite{b97-inc}, algebraic structures consisting of \textit{0-cells}, \textit{1\-cells} and \textit{2\-cells}. A standard graphical notation allows the 2\-cells to be represented by diagrams of the above form. 

\subsubsection*{Specifying protocols}

The strength of the formalism lies in the ability to combine these building blocks to form compound operations, or `protocols'. Equations between such compound operations can encode specifications for particular quantum tasks, completely independently of any details of the implementation. For example, the well-known \textit{quantum teleportation} protocol~\cite{b93-teleport} consists of the preparation of an entangled state, a measurement on two systems, and a controlled operation; its correct execution results in the perfect transfer of a quantum system, and the production of uncorrelated uniformly-distributed classical data. As a result, making use of the components described above, its specification has the following graphical form:
\allowdisplaybreaks[1]
\begin{align}
\tag{\ref{eq:teleport}}
\begin{aligned}
\begin{tikzpicture}
\node (V) [Vertex] at (1,0) {};
\node (W) [Vertex] at (2.25,1) {};
\draw [fill=\fillA, thick] (0,2)
    to [] (0,1)
    to [out=down, in=\nwangle] (1,0)
    to [out=\neangle, in=\swangle] (W.center)
    to [out=\nwangle, in=down] (1.5,2);
\draw [thick] (0.5,-1) to [out=up, in=down] (V.s1) to (V.s2) to [out=down, in=down] (3,0) to [out=up, in=\seangle] (W.center) to [out=\neangle, in=down] (3,2);
\end{tikzpicture}
\end{aligned}
\qquad&=\hspace{15pt}
\frac{1}{\sqrt{n}}
\,
\begin{aligned}
\begin{tikzpicture}
\draw [fill=\fillA, thick] (0,2)
    to (0,1.5)
    to [out=down, in=down, looseness=1.5] (1.5,1.5)
    to (1.5,2);
\draw [thick] (0.75,-1) to [out=up, in=down] (3,1.5) to (3,2);
\end{tikzpicture}
\end{aligned}
\end{align}
The entangled state is represented on the left-hand side by a U-shaped loop, as described by the work of Abramsky and Coecke~\cite{ac08-cqm, c03-loe}. This diagrammatic equation completely describes the form of the quantum teleportation protocol and its intended effect in a minimal and elegant way.

In a similar way, we can give a graphical specification for \textit{dense coding}:
\begin{align}
\tag{\ref{eq:densecoding}}
\frac{1}{\sqrt{n}}
\,
\begin{aligned}
\begin{tikzpicture}[scale=1, thick]
\draw [use as bounding box, draw=none] (-1.5,0.5) rectangle (2.5,3.5);
\node (m) [Vertex] at (1.75,2.5) {};
\node (u) [Vertex] at (0.5,1.5) {};
\draw [draw=none, fill=\fillA] (-0.25,0.5) to (-0.25,3.5) to (-1.5,3.5) to (-1.5,0.5);
\draw [fill=\fillA] (-0.25,0.5)
    to [out=up, in=\swangle] (u.center)
    to (u.center)
    to [out=\nwangle, in=down] (-0.25,3)
    to (-0.25,3.5);
\draw (m.center)
    to [out=\swangle, in=\neangle] (u.center)
    to [out=\seangle, in=120] (1,1)
    to [out=-60, in=down, looseness=1.5] (2.5,1.5)
    to [out=up, in=\seangle] (m.center);
\draw [fill=\fillA] (1,3.5)
    to [out=down, in=\nwangle] (m.center)
    to [out=\neangle, in=down] (2.5,3.5);
\end{tikzpicture}
\end{aligned}
\qquad&=\qquad
\begin{aligned}
\begin{tikzpicture}[scale=1, thick]
\draw [draw=none, fill=\fillA]
    (0,3)
    to [out=down, in=down, looseness=2]
    (1.25,3)
    to (2.5,3)
    to (2.5,2.5)
    to [out=down, in=up] (0,0.0)
    to (-1.25,0)
    to (-1.25,3)
    to (0,3);
\draw
    (2.5,3)
    to (2.5,2.5)
    to [out=down, in=up] (0,0.0)
    to (0,0);
\draw
    (0,3)
    to [out=down, in=down, looseness=2]
    (1.25,3);
\end{tikzpicture}
\end{aligned}
\end{align}
The left-hand side represents the preparation of a pair of systems in a maximally-entangled state, followed by a controlled operation on the first system dependent on some pre-existing classical data, and then a joint measurement on both systems. On the right-hand side, the initial classical data is simply copied. Successful execution is represented by the equality between the two sides.

Our 2\-categorical approach works well for understanding the high-level structure of these quantum protocols, and in Section 4 we analyze the structure of quantum teleportation, dense coding, complementary observables and quantum erasure. We introduce the new notion of \emph{horizontally invertible map}, and we see how this gives a new mathematical foundation for these quantum phenomena. We also introduce the notion of the \emph{theory} arising from a specification, and use it to show an equivalence between quantum teleportation protocols and dense coding protocols, as well as demonstrating a link between our formalism and wider topics in representation theory.

Our formalism is well-suited for giving specifications of new types of quantum protocol which have not been described in the literature. A key question for such novel protocols is implementability: can they be realized in finite-dimensional quantum physics? Since the value of a diagram is unchanged by topological deformation, implementability is a \textit{topological invariant} of our specifications. At present, however, this invariant is poorly understood. As an example we consider \emph{interlaced teleportation}, a novel protocol for which we give a graphical specification. By an argument of Fong~\cite{f12-it}, this specification cannot be implemented in conventional quantum theory. However, we are unable to give a high-level explanation for this fact. An open question remains: can we identify the specifications that can be implemented in quantum theory by directly topological means?

\subsubsection*{Implementations}

An \textit{implementation} of a specification is a solution to the graphical equation in a particular symmetric monoidal 2\-category. This involves giving a well-typed choice of 0\-cells, 1\-cells and 2\-cells for regions, lines and vertices of the diagrams, which satisfy the axioms of the formalism, and which satisfy the chosen equation when composed in the appropriate way. Whether or not an equation has a solution, and hence which information-theoretic tasks can be implemented, will depend on the choice of 2\-category, just as the existence of solutions to a conventional equation depends on the ring over which we attempt to solve it. In this sense, we can consider this choice as encoding the `theory of physics', or `model of computation', with which we are working.


Conventional finite-dimensional quantum theory is given by taking interpretations in \cat{2Hilb}, the 2\-category of 2\_Hilbert spaces, first described by Baez~\cite{b97-hda2}. A 2\_Hilbert space is a \textit{category} equipped with extra structure, which generalizes the notion of a Hilbert space as a \textit{set} with extra structure. In Section~\ref{sec:2hilb} we give an introduction to 2\_Hilbert spaces, and explain why applying our formalism in \cat{2Hilb} reproduces conventional quantum theory. This 2\-category can be presented in different ways, each giving a different perspective from which to understand our model. In Appendix~\ref{sec:environmental} we describe one such perspective in detail, presenting \cat{2Hilb} in terms of `environments', quantum systems equipped with `environmental maps', and `protected dynamics', giving a derivation of our model with a strong physical flavour and a close relationship to the concept of decoherence.

In its simplest form, the 2\-categorical mathematics on which our model depends can be thought of as finite-dimensional `higher linear algebra'. In contrast to the vector spaces and matrices of complex numbers which form the monoidal category \cat{Hilb} of finite-dimensional Hilbert spaces, the monoidal 2\-category \cat{2Hilb} is formed from `2--vector spaces', `matrices of vector spaces', and `matrices of matrices of complex numbers'~\cite{e06-2vect, r11-2vect}. These fit together in an elegant way which enables categorifications of various elementary notions from ordinary linear algebra to be described, such as algebras~\cite{eno02-ofc} and their modules~\cite{o03-mc}.

To work `by hand' with these structures can be difficult, since the definitions of horizontal composition, vertical composition and tensor product of 2\-cells are combinatorially intricate. Working with a computer algebra package can reduce these difficulties tremendously. One package in particular has been used to good effect by the author~\cite{r11-2vect}, which will soon be made available to the community. We use this in Section~\ref{sec:quantuminformation} to demonstrate explicitly that conventional implementations of teleportation, dense coding, complementarity and quantum erasure satisfy our 2\-dimensional equations. An accompanying \textit{Mathematica} notebook is available~\cite{v12-highernotebook} which contains these calculations. We anticipate that such tools will grow in popularity as the importance of higher linear algebraic techniques becomes more widely appreciated.

\ignore{An important feature of this formalism is that it is \emph{topological}, in the sense that if any two diagrams are homotopic with respect to their fixed boundaries, then they are equal as 2\-cells. From a quantum perspective, we could say that they give rise to the same `flow of quantum information', even though their individual components might be quite different. This is one manifestation of the idea that quantum information is topological, as developed particularly by Abramsky and Coecke~\cite{ac08-cqm, c03-loe}.}

\subsection{Scientific context}

\subsubsection*{Categorical quantum mechanics}

This work builds primarily on the categorical quantum mechanics programme initiated by Abramsky and Coecke~\cite{ac04-csqp,ac08-cqm}, which makes use of symmetric monoidal categories to present a high-level language for quantum information. Important tools developed as part of that programme continue to play a central role here, including classical structures~\cite{cpv08-dfb}, their modules~\cite{cpp09-cqs, cp06-qmws}, and the categorical description of mutually unbiased bases~\cite{cdkw12-scnl}. The key distinction between that programme and the one presented here is that our diagrams represent \textit{entire protocols}, including any branching due to quantum measurement. The symmetric monoidal category setting of the categorical quantum mechanics programme can be seen as the sector of our formalism in which all classical information is trivial; categorically, this corresponds to working only with the \emph{scalars} of our symmetric monoidal 2\-category.

There has been substantial activity in the past 10 years on the construction of categorical models encoding the interaction of quantum and classical data. One main body of work focuses on taking spaces of mixed states as the fundamental objects, and completely-positive maps as the appropriate notion of dynamical evolution. This makes sense in a context where entanglement with inaccessible environmental degrees of freedom destroys the purity of the state of the local quantum system, and where global dynamics means there is no local unitary evolution. Selinger and others have described an abstract formalization of this situation~\cite{c08-cpwp, ch11-pcp, s07-dccc} in which classical and quantum data sit alongside each other, and Coecke and Perdrix have used this framework to characterize classical data as a quantum system equipped with a coupling to an environmental degree of freedom, which is then explicitly traced out~\cite{cp08-ecc}.

Our model has a different perspective, since the environmental degrees of freedom are never discarded. If the combined environment and local system begin in a pure state, they will remain in one. Classical information is carried by this joint pure state in the form of entanglement between the environment and the local system. In the quantum information community, these different perspectives are sometimes informally described as the Churches of the Smaller and Larger Hilbert Spaces, the proponents of which differ (theologically at times) in their attitude towards mixed states on local Hilbert spaces: as an appropriate description of the physical system, or as the restriction of some pure state on a global Hilbert space which ought to be modelled directly. Our formalism is naturally associated with the latter camp, giving concrete mathematical structure to classical information in coherent-state quantum theory. It may serve as a useful formal model for areas of research in quantum theory for which full coherence of the quantum state is maintained, such as decoherence~\cite{s05-decoherence} and the many-worlds interpretation~\cite{d10-afu}.

\subsubsection*{Categorification}

This work is also in the spirit of the \emph{categorification} programme~\cite{bd98-cat, cy98-eoc}, named by Crane and Frenkel~\cite{cf94-4d} as an umbrella term for the study of replacing set-based structures with category-based structures. Its relevance for the foundations of physics has been particularly advocated by Baez and Dolan~\cite{b-twf,b97-hda2, b99-hda, bd98-cat, bhw09-hda7, bl11-pnc, bh10-ihgt}, and it has found significant applications in topological quantum field theory~\cite{bd95-hda, cf94-4d, l09-otc} and quantum foundations~\cite{bd01-fsfd, m06-caqm}.

The present work draws from these advances, and extends them to the sphere of quantum information where categorification has not so far been significantly applied. The central role played here by the 2\-category \cat{2Hilb} of 2\_Hilbert spaces, also due to Baez~\cite{b97-hda2}, provides a direct link to much of the modern literature on topological quantum field theory, conformal field theory, quantum groups and higher representation theory, where \cat{2Hilb} is of major importance.

We can give a direct argument for why 2\-category theory should be relevant for quantum theory. Suppose we have a quantum system with a finite-dimensional Hilbert space, seen as an object in the category \cat{Hilb} of finite-dimensional Hilbert spaces. Then if a projector-valued measurement takes place, future dynamics can be different depending on which of the $n$ possible  measurement results occurred. We now effectively have $n$ independent copies of our state space, conveniently described as an object in $\cat{Hilb} ^n$, the $n$-fold Cartesian product of the category of finite-dimensional Hilbert spaces. A categorical setting in which to study the measurement process must therefore include both \cat{Hilb} and $\cat{Hilb}^n$, since these give our mathematical context before and after the measurement takes place; and since they are themselves categories, the correct setting in which to study them will be a 2\-category. The 2\-category \cat{2Hilb} is then a natural choice, since up to equivalence, its objects are precisely the categories of the form $\cat{Hilb}^n$.

\subsection{Acknowledgements}

I would like to thank Samson Abramsky, Mauro D'Ariano, John Baez, Andrew Briggs, Jeremy Butterfield, Bob Coecke, Ross Duncan, Brendan Fong, Chris Heunen and Matthew Pusey for useful comments, and especially Owen Maroney for interesting discussions on the role of decoherence in quantum theory. Early versions of these results were presented to the Quantum Nanoscience working group and the 2012 ONR Quantum Information Science Workshop, both at the University of Oxford, and I am grateful for the useful audience feedback. This work was supported financially by the Centre for Quantum Technologies at the National University of Singapore. Graphics in this article have been prepared using the package \texttt{TikZ}.

This article is written for a diverse audience of Computer Scientists, Mathematicians, Philosophers and Physicists. Comments on the exposition are always welcome.

\section{The formalism}
\label{sec:formalism}

\subsection{Introduction}

Our formalism is based on the theory of \emph{symmetric monoidal 2-categories}~\cite{b97-inc, l98-bb}, in their fully weak form. These are algebraic structures built from \textit{objects}, \textit{1-cells} and \textit{2-cells}, which for us will represent `stores of classical information', `physical systems', and `interactions'. We will make heavy use of the graphical notation for 2\-categories~\cite{l06-faaa} to describe our theory in an accessible way which requires no previous knowledge of 2\-category theory beyond an understanding of the basic definitions.

The formalism is largely independent of any particular theory of physics. While quantum theory directly inspires the formalism, it serves only as one particular \emph{model} of the mathematical structures we develop. Despite this, we will use `quantum intuition' throughout to guide our interpretation of the constructions, often in an informal and imprecise way. This will be made rigorous in Section~\ref{sec:2hilb}, where we will see formally how the familiar quantum notions arise as instances of these constructions.

\subsection{Basic principles}
\label{sec:basicprinciples}

\subsubsection*{Objects}

Given a symmetric monoidal 2\-category, we use the graphical calculus for the underlying 2\-category to represent an object $A$ as a rectangular region of the plane:
\begin{tikzequation}[scale=2]
    \draw [white] (0.2,1) to (1.8,1);
    \begin{pgfonlayer}{foreground}
    \end{pgfonlayer}
    \draw [fill=\fillA, draw=none] (0.2,0.5)
        to (0.2,1.5)
        to (1.8,1.5)
        to (1.8,0.5)
        to (0.2,0.5);
    \node [fill=white, rounded corners, inner sep=2pt] at (1,1) {$A$};
\end{tikzequation}

\noindent
We think of objects as representing stores of classical information. We can imagine them as behaving like sets, whose elements are the possible values the classical information can take. Another useful intuition is to consider them as `environments', which `measure'  the quantum systems put in proximity with them. We have a trivial store of information, represented by the monoidal unit object $I$, and we represent this graphically by the empty region.

\subsubsection*{1-cells}

A  1\-cell $A \sxto S B$ is interpreted as a physical system, with information about its state `stored' in its source object $A$ and target object $B$. We cannot talk in this formalism about a physical system without specifying the objects with which it interacts. We represent a system graphically as a vertical line, which borders the regions corresponding to the source and target objects  on the left and right respectively:
\begin{tikzequation}[scale=2]
    \draw [white] (0.2,1) to (1.8,1);
    \draw [fill=\fillA, draw=none]
        (0.2,0.5)
        rectangle
            node [fill=white, inner sep=2pt, rounded corners] {$A$}
        (1,1.5);
    \draw [fill=\fillB, draw=none]
        (1.8,0.5)
        rectangle
            node [fill=white, inner sep=2pt, rounded corners] {$B$}
        (1,1.5);
    \draw [thick]
        (1,0.5)
            node [below] {$S$}
        to
        (1,1.5);
\end{tikzequation}
Quantum mechanically, we can think of $A$ and $B$ as a commuting pair of observables acting on a Hilbert space $S$, each inducing a decomposition of $S$ into orthogonal subspaces according to the different possible measurement outcomes. Information about $S$ is `stored' in the objects $A$ and $B$, in the sense that selecting a definite value of either observable gives a restriction on the possible states of $S$.

The composition of 1-cells $A \sxto S B$ and $B \sxto T C$ is represented graphically by juxtaposition:
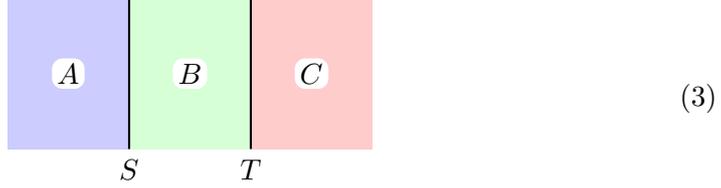
\begin{tikzequation}[scale=2]
    \draw [white] (0.2,1) to (1.8,1);
    \draw [fill=\fillA, draw=none]
        (0.2,0.5)
        rectangle
            node [fill=white, inner sep=2pt, rounded corners] {$A$}
        (1,1.5);
    \draw [fill=\fillB, draw=none]
        (1.8,0.5)
        rectangle
            node [fill=white, inner sep=2pt, rounded corners] {$B$}
        (1,1.5);
    \draw [fill=\fillC, draw=none]
        (2.6,0.5)
        rectangle
            node [fill=white, inner sep=2pt, rounded corners] {$C$}
        (1.8,1.5);
    \draw [thick]
        (1,0.5)
            node [below] {$S$}
        to
        (1,1.5);
    \draw [thick]
        (1.8,0.5)
            node [below] {$T$}
        to
        (1.8,1.5);
\end{tikzequation}
The object $B$ now stores information about both $S$ and $T$. This gives us a restriction on the possible joint states of $S$ and $T$: only those corresponding to the same value of $B$ are allowed. As a result, the states of $S$ and $T$ are correlated. Looking at the values of $A$ and $C$ would give us further independent information about the states of $S$ and $T$ respectively. In the quantum mechanical picture, the state space of the resulting system is not the tensor product of the state spaces for each system, but the largest subspace of it for which the observable $B$ has the same outcome on both subsystems.

\subsubsection*{2-cells}

Dynamical evolution of a physical system is represented by a 2-cell $S \sxto \beta S'$, for physical systems $A \sxto {S,S'} B$. We can think of these as representing physical processes which can take place. We represent this graphically as a vertex:
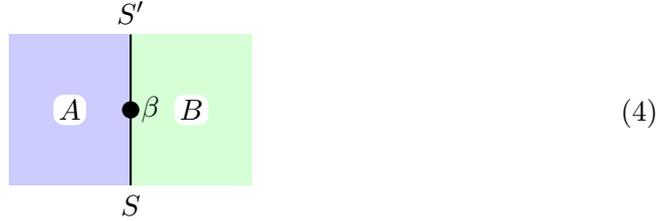
\begin{tikzequation}[scale=2]
    \draw [white] (0.2,1) to (1.8,1);
    \draw [fill=\fillA, draw=none]
        (0.2,0.5)
        rectangle
            node [fill=white, inner sep=2pt, rounded corners] {$A$}
        (1,1.5);
    \draw [fill=\fillB, draw=none]
        (1.8,0.5)
        rectangle
            node [fill=white, inner sep=2pt, rounded corners] {$B$}
        (1,1.5);
    \draw [thick]
        (1,0.5)
            node [below] {$S$}
        to
        (1,1.5)
            node [above] {$S'$};
    \node [Vertex] at (1,1) {};
    \node at (1.13,1) {$\beta$};
\end{tikzequation}
Such a 2\-cell maps states of $S$ into states of $S'$, in a way that preserves the information stored by $A$ and $B$. Particular families of 2\-cell have natural interpretations as familiar information-theoretical tasks, as we will see below. However, all 2\-cells represent \textit{some} valid physical process in the `theory of physics' modelled by our 2\-category. In the quantum-mechanical interpretation, $\beta$ is a linear map which leaves invariant the outcomes of the observables $A$ and $B$, treating them as conserved quantities. The subspaces corresponding to different outcomes are `superselection sectors', to use the terminology of theoretical physics.

Diagrams representing 2\-cells can be `pasted' together along common edges, forming larger composite 2\-cells. This can be done in two basic ways: horizontally, written algebraically with the connective $\circ$; and vertically, written algebraically with the connective $\cdot$. The algebraic content of a 2\-category lies in the description of when two such composites of the same overall type are equal.

\subsubsection*{Reversing processes}
\label{sec:dagger}
We require our monoidal 2\-category to come with a good notion of the \emph{reversal} of a process. For each 2\-cell $S \sxto \beta T$, we write the reversed process as $T \sxto {\beta ^\dagger} S$. Quantum mechanically, this reversal is given by the adjoint of the linear map that describes the process. We model it formally as an involutive monoidal endofunctor on our symmetric monoidal 2\-category, written $\dag$, which is the identity on objects and 1\-cells, and which preserves all relevant categorical structures. Graphically, this preservation condition says that the reversal acts by flipping diagrams about a horizontal axis:
\begin{equation}
\begin{aligned}
\begin{tikzpicture}[yscale=1, xscale=1.5]
    \begin{pgfonlayer}{foreground}
        \node (f) [draw=black, thick, fill=white, minimum width=20pt, minimum height=9pt, inner sep=0pt] at (1,1) {\tiny$f$};
        \node (g) [draw=black, thick, fill=white, minimum width=20pt, minimum height=9pt, inner sep=0pt] at (1.5,2) {\tiny$g$};
    \end{pgfonlayer}
    \draw [draw=white, fill=\fillA]
        (-0.0,0) to (1.5,0) to (1.5,3) to (-0.0,3);
    \draw [draw=white, fill=\fillC]
        (1,0) to (2.5,0) to (2.5,3) to (1,3);
    \draw [fill=\fillB, draw=white]
        (0.5,3) to (0.5,2)
        to [out=down, in=up] (f.140)
        to (f.40)
        to [out=up, in=down] (g.-140)
        to (g.90)
        to (1.5,3) to (0.5,3);
    \draw [fill=yellow!70, draw=white]
        (1.5,3)
        to (2.5,3)
        to (2.5,0)
        to (2,0) to (2,1)
        to [out=up, in=down] (g.-40)
        to (g.90);
    \draw [thick] (1,0) to (1,1);
    \draw [thick] (f.40) to [out=up, in=down] (g.-140);
    \draw [thick] (f.140)
        to [out=up, in=down] (0.5,2)
        to (0.5,3);
    \draw [thick] (g.90) to (1.5,3);
    \draw [thick] (g.-40) to [out=down, in=up] (2,1) to (2,0);
\end{tikzpicture}
\end{aligned}
\quad
\stackrel{\dagger}{\mapsto}
\quad
\begin{aligned}
\begin{tikzpicture}[yscale=-1, xscale=1.5]
    \begin{pgfonlayer}{foreground}
        \node (f) [draw=black, thick, fill=white, minimum width=20pt, minimum height=9pt, inner sep=0pt] at (1,1) {\tiny$f ^\dag$};
        \node (g) [draw=black, thick, fill=white, minimum width=20pt, minimum height=9pt, inner sep=0pt] at (1.5,2) {\tiny$g ^\dag$};
    \end{pgfonlayer}
    \draw [draw=white, fill=\fillA]
        (-0.0,0) to (1.5,0) to (1.5,3) to (-0.0,3);
    \draw [draw=white, fill=\fillC]
        (1,0) to (2.5,0) to (2.5,3) to (1,3);
    \draw [fill=\fillB, draw=white]
        (0.5,3) to (0.5,2)
        to [out=down, in=up] (f.-140)
        to (f.-40)
        to [out=up, in=down] (g.140)
        to (g.-90)
        to (1.5,3) to (0.5,3);
    \draw [fill=yellow!70, draw=white]
        (1.5,3)
        to (2.5,3)
        to (2.5,0)
        to (2,0) to (2,1)
        to [out=up, in=down] (g.40)
        to (g.90);
    \draw [thick] (1,0) to (1,1);
    \draw [thick] (f.-40) to [out=up, in=down] (g.140);
    \draw [thick] (f.-140)
        to [out=up, in=down] (0.5,2)
        to (0.5,3);
    \draw [thick] (g.-90) to (1.5,3);
    \draw [thick] (g.40) to [out=down, in=up] (2,1) to (2,0);
\end{tikzpicture}
\end{aligned}
\end{equation}
Since we work exclusively using the graphical calculus, this formulation of the preservation condition is sufficient for our purposes.

With respect to a reversal operation $\dag$, we say that a 2\-cell $\mu$ is an \textit{isometry} if $\mu ^\dag \cdot \mu = \id$, and \emph{unitary} if in addition $\mu \cdot \mu^\dag = \id$. This generalizes standard terms from linear algebra.

\subsection{Witnessing classical information}
\label{sec:witness}

\subsubsection*{Introduction}

In quantum theory we can encode a classical bit as the quantum states $\ket 0$ and $\ket 1$ of a qubit. More generally, given any type of classical information taking some set of possible values, we can describe a quantum system whose Hilbert space has a basis indexed by these values. Preparing our system in one of these basis states then gives a \textit{physical witness} encoding the value of the classical information, and many such witnesses can be freely prepared.

Any particular such witness is perfectly correlated with the classical bit. In our notation this gives rise to systems of type $2 \to 1$ or $1 \to 2$, as described in Section~\ref{sec:basicprinciples}. In this section we give a formal axiomatization of the general structure held by these witnesses, which turns out to be overtly topological.

\subsubsection*{Abstract witnesses}

For an object $A$ representing a classical data type, its \textit{witnesses}, if they exist, are 1\-cells of type $A \sxto {A} I$ and $I \sxto {A} A$. By abuse of notation, we give both of these 1\-cells the same name as the object whose data they are encoding. Graphically, we represent our witnesses in the following way:
\begin{calign}
\label{eq:examplewitnesses}
\begin{aligned}
\begin{tikzpicture}
\draw [white] (0,0) to (2,2);
\draw [fill=\fillA, draw=none] (0,0)
    to (1,0)
    to (1,2)
    to (0,2);
\draw [thick] (1,0) to node [auto, swap] {} (1,2);
\end{tikzpicture}
\end{aligned}
&
\begin{aligned}
\begin{tikzpicture}
\draw [white] (0,0) to (-2,2);
\draw [fill=\fillA, draw=none] (0,0)
    to (-1,0)
    to (-1,2)
    to (0,2);
\draw [thick] (-1,0) to node [auto] {} (-1,2);
\end{tikzpicture}
\end{aligned}
\\
A \sxto A I & I \sxto A A
\end{calign}
For $A \sxto {A} I$ and $I \sxto {A} A$ to be good witnesses, we require the following physical operations to be available as 2\-cells, each having a particular interpretation:
\def\aascale{1.3}
\def\aaspace{\hspace{30pt}}
\setlength\fboxsep{0pt}
\def\sep{10pt}
\def\innersep{3pt}
\renewcommand{\centerdia}[1]{\ensuremath{#1}}
\renewcommand\newtwocell[3]{\centerdia{\begin{aligned}
\begin{tikzpicture}[scale=\aascale]
    #1
    \draw [black!20]
        ([xshift=-0.2cm, yshift=-0.2cm] current bounding box.south west)
        rectangle
        ([xshift=0.2cm, yshift=0.2cm] current bounding box.north east);
\end{tikzpicture}
\end{aligned}
& \hspace{15pt} \makebox[170pt][l]{\vc{#2 \\[\innersep] #3}}}}
\renewcommand\separatetwocells{\\[\sep]}
\allowdisplaybreaks[1]
\begin{align}
\newtwocell{
    \draw [white] (0.2,1.5) to (1.8,1.5);
    \draw [fill=\fillA, thick] (0.5,1.5)
        to [out=down, in=down, looseness=1.5] (1.5,1.5);
    \draw [thick, white] (1,0.5) to (1,1);
}
{Create uniform classical data\label{eq:create}}
{}
\separatetwocells
\newtwocell{
    \draw [white] (0.2,-1) to (1.8,-1);
    \draw [fill=\fillA, thick] (0.5,-1.5)
        to [out=up, in=up, looseness=1.5] (1.5,-1.5);
    \draw [thick, white] (1,-0.5) to (1,-1);
}
{Delete classical data\label{eq:forget}}
{}
\separatetwocells
\newtwocell{
    \draw [fill=\fillA, draw=none] (0.2,-0.5)
        to (0.5,-0.5)
        to [out=down, in=down, looseness=1.5] (1.5, -0.5)
        to (1.8,-0.5)
        to (1.8,-1.5)
        to (0.2,-1.5)
        to (0.2,-0.5);
    \draw [thick] (0.5,-0.5)
        to [out=down, in=down, looseness=1.5] (1.5,-0.5);
}
{Copy classical data\label{eq:copy}}
{}
\separatetwocells
\newtwocell{
    \draw [fill=\fillA, draw=none] (0.2,0.5)
        to (0.5,0.5)
        to [out=up, in=up, looseness=1.5] (1.5,0.5)
        to (1.8,0.5)
        to (1.8,1.5)
        to (0.2,1.5)
        to (0.2,0.5);
    \draw [thick] (0.5,0.5)
        to [out=up, in=up, looseness=1.5] (1.5,0.5);
}
{Compare classical data\label{eq:compare}}
{}
\end{align}
Since the graphical notations for~\eqref{eq:create} and~\eqref{eq:forget} are mirror-images of each other about a horizontal axis, this implies that they are related to each other by the \dag\-operation, as described in Section~\ref{sec:dagger}. The same holds for operations~\eqref{eq:copy} and~\eqref{eq:compare}. This encodes the intuition that, for each of these pairs, each member is the `reversal' of the other.

We can describe intuitively the tasks these 2\-cells represent. In~\eqref{eq:create}, instances of systems $A \sxto {A} I$ and $I \sxto {A} A$ are created, in a superposition of every possible state. The states of these new systems are correlated, and also recorded by the classical data associated to the object $A$. Essentially, this models the quantum creation of uniformly random classical information. In~\eqref{eq:forget}, a reverse process is carried out: regardless of which state two perfectly-correlated copies of $S$ are in, those systems are erased. The 2\-cell~\eqref{eq:copy} has the interpretation of copying classical data: given an original single region of classical data, two new copies of the system $S$ are created, each in the given state. And lastly,~\eqref{eq:compare} represents the process of comparing instances of $A \sxto {A} I$ and $I \sxto {A} A$ and ensuring the classical information recorded about them by the object $A$ is the same. This final operation has the possibility of failing, which in a particular model of the formalism would correspond to the presence of a nontrivial kernel.

\subsubsection*{Topological axioms}

To capture these intuitions mathematically, we require that these 2\-cells satisfy some equations, which have a clear topological interpretation. In categorical terminology, the following encode the statement that $I \sxto A A$ and $A \sxto A I$ are \textit{ambidextrous duals}:
\def\aascale{1.0}
\begin{align}
\label{eq:yank1}
\begin{aligned}
\begin{tikzpicture}[scale=\aascale]
\draw [use as bounding box, draw=none] (-0.5,0) rectangle (2.3,2);
\draw [white] (-0.5,0) to (3.1,2);
\draw [fill=\fillA, draw=none] (-0.5,0) to (0.3,0) to (0.3,1)
    to [out=up, in=up, looseness=1.5] (1.3,1)
    to [out=down, in=down, looseness=1.5] (2.3,1)
    to (2.3,2) to (-0.5,2);
\draw [thick] (0.3,0) to (0.3,1)
    to [out=up, in=up, looseness=1.5] (1.3,1)
    to [out=down, in=down, looseness=1.5] (2.3,1)
    to (2.3,2);
\end{tikzpicture}
\end{aligned}
\quad&=\quad
\begin{aligned}
\begin{tikzpicture}[scale=\aascale]
\draw [use as bounding box, draw=none] (1,0) rectangle (0,2);
\draw [white] (0,0) to (2,2);
\draw [fill=\fillA, draw=none] (0,0)
    to (1,0)
    to (1,2)
    to (0,2);
\draw [thick] (1,0) to (1,2);
\end{tikzpicture}
\end{aligned}
\\[5pt]
\begin{aligned}
\begin{tikzpicture}[scale=\aascale]
\draw [use as bounding box, draw=none] (-0.5,0) rectangle (2.3,-2);
\draw [white] (-0.5,0) to (3.1,-2);
\draw [fill=\fillA, draw=none] (-0.5,0) to (0.3,0) to (0.3,-1)
    to [out=down, in=down, looseness=1.5] (1.3,-1)
    to [out=up, in=up, looseness=1.5] (2.3,-1)
    to (2.3,-2) to (-0.5,-2);
\draw [thick] (0.3,0) to (0.3,-1)
    to [out=down, in=down, looseness=1.5] (1.3,-1)
    to [out=up, in=up, looseness=1.5] (2.3,-1)
    to (2.3,-2);
\end{tikzpicture}
\end{aligned}
\quad&=\quad
\begin{aligned}
\begin{tikzpicture}[scale=\aascale]
\draw [use as bounding box, draw=none] (1,0) rectangle (0,2);
\draw [white] (0,0) to (2,2);
\draw [fill=\fillA, draw=none] (0,0)
    to (1,0)
    to (1,2)
    to (0,2);
\draw [thick] (1,0) to (1,2);
\end{tikzpicture}
\end{aligned}
\\[5pt]
\begin{aligned}
\begin{tikzpicture}[scale=\aascale]
\draw [use as bounding box, draw=none] (-2.3,0) rectangle (0.5,2);
\draw [white] (0.5,0) to (-3.1,2);
\draw [fill=\fillA, draw=none] (0.5,0) to (-0.3,0) to (-0.3,1)
    to [out=up, in=up, looseness=1.5] (-1.3,1)
    to [out=down, in=down, looseness=1.5] (-2.3,1)
    to (-2.3,2) to (0.5,2);
\draw [thick] (-0.3,0) to (-0.3,1)
    to [out=up, in=up, looseness=1.5] (-1.3,1)
    to [out=down, in=down, looseness=1.5] (-2.3,1)
    to (-2.3,2);
\end{tikzpicture}
\end{aligned}
\quad&=\quad
\begin{aligned}
\begin{tikzpicture}[scale=\aascale]
\draw [use as bounding box, draw=none] (-1,0) rectangle (0,2);
\draw [white] (0,0) to (-2,2);
\draw [fill=\fillA, draw=none] (0,0)
    to (-1,0)
    to (-1,2)
    to (-0,2);
\draw [thick] (-1,0) to (-1,2);
\end{tikzpicture}
\end{aligned}
\\[5pt]
\label{eq:yank4}
\begin{aligned}
\begin{tikzpicture}[scale=\aascale]
\draw [use as bounding box, draw=none] (-2.3,0) rectangle (0.5,-2);
\draw [white] (0.5,0) to (-3.1,-2);
\draw [fill=\fillA, draw=none] (0.5,0) to (-0.3,0) to (-0.3,-1)
    to [out=down, in=down, looseness=1.5] (-1.3,-1)
    to [out=up, in=up, looseness=1.5] (-2.3,-1)
    to (-2.3,-2) to (0.5,-2);
\draw [thick] (-0.3,0) to (-0.3,-1)
    to [out=down, in=down, looseness=1.5] (-1.3,-1)
    to [out=up, in=up, looseness=1.5] (-2.3,-1)
    to (-2.3,-2);
\end{tikzpicture}
\end{aligned}
\quad&=\quad
\begin{aligned}
\begin{tikzpicture}[scale=\aascale]
\draw [use as bounding box, draw=none] (-1,0) rectangle (0,2);
\draw [white] (0,0) to (-2,2);
\draw [fill=\fillA, draw=none] (0,0)
    to (-1,0)
    to (-1,2)
    to (0,2);
\draw [thick] (-1,0) to (-1,2);
\end{tikzpicture}
\end{aligned}
\end{align}
These axioms say that bends in the graphical notation can be introduced or straightened-out without changing the value of the diagram as a whole, giving the calculus topological properties.

We also impose two further equations which ensure that our copying and comparison operations have good properties. The first encodes the idea that copying followed by comparison should give the identity:
\begin{equation}
\label{eq:copycompare}
\begin{aligned}
\begin{tikzpicture}
\draw [fill=\fillA, draw=none] (0.5,0.5) rectangle (2.5,2.5);
\draw [fill=white, thick] (1,1.5)
    to [out=up, in=up, looseness=1.5] (2,1.5)
    to [out=down, in=down, looseness=1.5] (1,1.5);
\end{tikzpicture}
\end{aligned}
\quad=\quad
\begin{aligned}
\begin{tikzpicture}
\draw [fill=\fillA, draw=none] (0.5,0.5) rectangle (2.5,2.5);
\end{tikzpicture}
\end{aligned}
\end{equation}
The second says that if we copy classical data, swapping the copies makes no difference. Here we make use of the symmetric monoidal 2\-category structure to interpret `crossings' of 1-cells in our graphical calculus.
\begin{equation}
\label{eq:commplanar}
\begin{aligned}
\begin{tikzpicture}[scale=0.8]
\draw [fill=\fillA, draw=none] (0,0)
    to [out=up, in=down, out looseness=1.5] (2,2)
    to (3,2)
    to [out=down, in=up, out looseness=1.5] (1,0);
\draw [fill=\fillA, draw=none] (3,0)
    to [out=up, in=down, out looseness=1.5] (1,2)
    to (0,2)
    to [out=down, in=up, out looseness=1.5] (2,0);
\draw [fill=\fillA, draw=none] (0,2)
    to (1,2)
    to [out=up, in=up, looseness=1.5] (2,2)
    to (3,2)
    to [out=up, in=down] (2,4)
    to (1,4) to [out=down, in=up] (0,2);
\draw [thick] (0,0)
    to [out=up, in=down, out looseness=1.5] (2,2)
    to [out=up, in=up, looseness=1.5] (1,2)
    to [out=down, in=up, in looseness=1.5] (3,0);
\draw [thick] (1,0)
    to [out=up, in=down, in looseness=1.5] (3,2)
    to [out=up, in=down] (2,4);
\draw [thick] (2,0)
    to [out=up, in=down, in looseness=1.5] (0,2)
    to [out=up, in=down] (1,4);
\end{tikzpicture}
\end{aligned}
\quad=\quad
\begin{aligned}
\begin{tikzpicture}[scale=0.8]
\draw [fill=\fillA, draw=none] (0,0)
    to [out=up, in=down, out looseness=1.5] (0,2)
    to (1,2)
    to [out=down, in=up, out looseness=1.5] (1,0);
\draw [fill=\fillA, draw=none] (3,0)
    to [out=up, in=down, out looseness=1.5] (3,2)
    to (2,2)
    to [out=down, in=up, out looseness=1.5] (2,0);
\draw [fill=\fillA, draw=none] (0,2)
    to (1,2)
    to [out=up, in=up, looseness=1.5] (2,2)
    to (3,2)
    to [out=up, in=down] (2,4)
    to (1,4) to [out=down, in=up] (0,2);
\draw [thick] (3,0)
    to [out=up, in=down, in looseness=1.5] (3,2)
    to [out=up, in=down] (2,4);
\draw [thick] (0,0)
    to [out=up, in=down, in looseness=1.5] (0,2)
    to [out=up, in=down] (1,4);
\draw [thick] (2,0) 
    to (2,2)
    to [out=up, in=up, looseness=1.5] (1,2)
    to (1,0);
\end{tikzpicture}
\end{aligned}
\end{equation}
We say that 1\-cells $A \sxto {A} I$ and $I \sxto {A} A$ provide \emph{witnesses} for the object~$A$ when they are equipped with 2\-cells~(\ref{eq:create}\_\ref{eq:compare}) satisfying equations~(\ref{eq:yank1}\_\ref{eq:commplanar}). Aside from the role played by the monoidal unit object $I$, this equation is the only axiom making use of the monoidal structure on our 2\-category.

The axioms \eqref{eq:yank1}\_\eqref{eq:commplanar} can be neatly summarized by saying that for a given pair of witnesses~\eqref{eq:examplewitnesses}, any two surfaces built from the components~\eqref{eq:create}\_\eqref{eq:compare} are equal if and only if they have the same topology, up to genus. For example, the following equation holds:
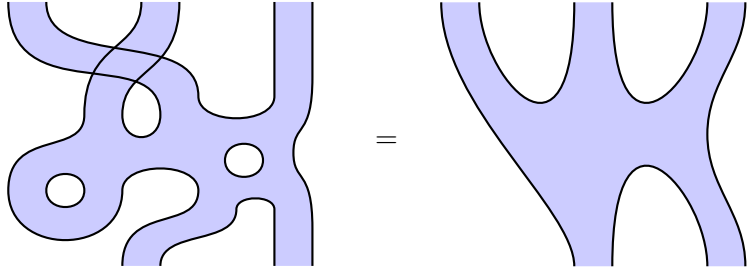
\[
\begin{aligned}
\begin{tikzpicture}[scale=0.5]
\draw [fill=\fillA, draw=none] (8,1)
        to [out=up,in=down] (8,2) to [out=up,in=down, out looseness=2] (7.5,4) to [out=up,in=down, in looseness=2] (8,6) to (8,8) to (7,8) to (7,5.5) to [out=down,in=down, looseness=1.0] (5,5.5) to [out=up,in=down, out looseness=1.0] (1,8) to (0,8) to [out=down,in=up, out looseness=1.5] (4,5) to [out=down, in=down, looseness=2] (3,5) to [out=up, in=down, in looseness=1.5] (4.5,8) to (3.5,8) to [out=down, in=up, in looseness=1.5] (2,5) to [out=down, in=up, in looseness=1.5] (0,3) to [out=down, in=down, looseness=1.5] (3,3) to [out=up, in=up] (5,3) to [out=down, in=up, in looseness=1.5] (3,1) to (4,1) to [out=up, in=down, in looseness=1.0] (6,2.5) to [out=up, in=up] (7,2.5) to (7,1);
\draw [fill=white, thick] (1,3) to [out=up, in=up, looseness=1.5] (2,3) to [out=down, in=down, looseness=1.5] (1,3);
\draw [fill=white, thick] (5.7,3.8) to [out=up, in=up, looseness=1.5] (6.7,3.8) to [out=down, in=down, looseness=1.5] (5.7,3.8);
\draw [thick] (8,1)
        to [out=up,in=down] (8,2) to [out=up,in=down, out looseness=2] (7.5,4) to [out=up,in=down, in looseness=2] (8,6) to (8,8);
\draw [thick] (7,8) to (7,5.5) to [out=down,in=down] (5,5.5) to [out=up,in=down, out looseness=1.0] (1,8);
\draw [thick] (0,8) to [out=down,in=up, out looseness=1.5] (4,5) to [out=down, in=down, looseness=2] (3,5) to [out=up, in=down, in looseness=1.5] (4.5,8);
\draw [thick] (3.5,8) to [out=down, in=up, in looseness=1.5] (2,5) to [out=down, in=up, in looseness=1.5] (0,3) to [out=down, in=down, looseness=1.5] (3,3) to [out=up, in=up] (5,3) to [out=down, in=up, in looseness=1.5] (3,1);
\draw [thick] (4,1) to [out=up, in=down, in looseness=1.0] (6,2.5) to [out=up, in=up] (7,2.5) to (7,1);
\end{tikzpicture}
\end{aligned}
\,\,\,\,\,\quad=\quad
\begin{aligned}
\begin{tikzpicture}[scale=0.5]
\draw [fill=\fillA, draw=none] (8,0) to [out=up, in=down] (7,3.5) to [out=up, in=down] (8,7) to (7,7) to [out=down, in=down, in looseness=5, out looseness=2] (4.5,7) to (3.5,7) to [out=down, in=down, out looseness=5, in looseness=2] (1,7) to (0,7) to [out=down, in=up, out looseness=0.9, in looseness=0.5] (3.5,0) to (4.5,0) to [out=up, in=up, out looseness=5, in looseness=2] (7,0);
\draw [thick] (8,0) to [out=up, in=down] (7,3.5) to [out=up, in=down] (8,7);
\draw [thick] (7,7) to [out=down, in=down, in looseness=5, out looseness=2] (4.5,7);
\draw [thick] (3.5,7) to [out=down, in=down, out looseness=5, in looseness=2] (1,7);
\draw [thick] (0,7) to [out=down, in=up, out looseness=0.9, in looseness=0.5] (3.5,0);
\draw [thick] (4.5,0) to [out=up, in=up, out looseness=5, in looseness=2] (7,0);
\end{tikzpicture}
\end{aligned}
\]
If the intuitions described above for the 2\-cells~(\ref{eq:create}\_\ref{eq:compare}) are accepted, then every such equation is a meaningful story about the consequences of creating, forgetting, copying and comparing classical data. However, formally, we take this topological property to be the defining consistency condition which makes these intuitions reasonable. Evidence for the reasonableness of these axioms and intuitions is that they match when applied in \cat{2Hilb}, as described in Section~\ref{sec:2hilb}.

\ignore{
\subsubsection*{Extracting copies of witnesses}

Given witnesses $A \sxto {A} I$ and $I \sxto {A} A$ for an object $A$, the composite $I \sxto A A \sxto A I$ to represents a single copy of the underlying physical system which is used to witness $A$, shorn of its correlations. 
}

\subsection{Measurements and controlled operations}

\tikzset{thicker/.style={line width=2pt}}

\subsubsection*{Nondegenerate measurement}

Measurement in our formalism is a dynamical process, in which the result of the measurement is encoded by correlations between physical systems.

Given systems $A \sxto {A} I$ and $I \sxto {A} A$ witnessing the object $A$, a \textit{nondegenerate measurement} is a unitary 2\-cell of the following form:
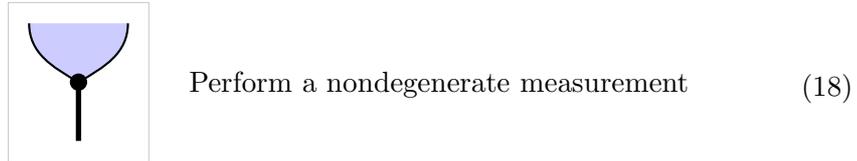
\def\aascale{1.3}
\begin{align}
\label{eq:measurementdefinition}
\newtwocell{
    \node (V) [Vertex] at (1.5,1) {};
    \draw [fill=\fillA, draw=none]
        (2,1.6)
        to [out=down, in=\neangle] (V.center)
        to [out=\nwangle, in=down] (1,1.6);
    \draw [thicker]
        (1.5,0.4)
        to (1.5,1.0);
    \draw [thick] (1,1.6)
        to [out=down, in=\nwangle] (V.center);
    \draw [thick] (V.center)
        to [out=\neangle, in=down] (2,1.6);
}
{Perform a nondegenerate measurement}
{}
\end{align}
The thick line in the lower part of the picture represents the physical system to be measured, uncorrelated with any classical information. After the measurement we have the composite  $I \sxto A A \sxto {A} I$, representing a pair of systems  which are perfectly correlated with respect to classical data associated to the object~$A$. This classical data indexes the possible outcomes of the measurement, and the two resulting systems are produced in a superposition of correlated states which encode these outcomes.

Invertibility of the measurement 2\-cell~\eqref{eq:measurementdefinition} is enormously important, and befits a notion of measurement in a formalism inspired by decoherence. There is no `collapse of the wavefunction', only the development of correlations between physical systems. In this case, those physical systems are the two correlated systems which form the left- and right-hand boundaries of the shaded region. As we will see, these act as `gateways' to the classical information: they can produce new quantum systems, initialized to carry a copy of the classical data that was obtained from the measurement. The shaded region represents a `future information cone', analogous to a future light cone in special relativity, in which the result of the measurement is available.

The space of isomorphisms of the form~\eqref{eq:measurementdefinition} represents the space of nondegenerate measurements that can be performed on the initial system to yield a measurement outcome valued in $A$. For example, as we explore in detail in Section~\ref{sec:2hilb}, if $A$ represents a classical bit and the initial system is a qubit with state space $\mathbb{C} ^2$, then the space of unitaries of the form~\eqref{eq:measurementdefinition} is precisely the space of orthonormal bases of $\mathbb{C}^2$.  Or it might be that $A$ has the wrong number of degrees of freedom to encode the outcomes of nondegenerate measurements on our system: for example, in quantum theory, we cannot use the 3-element set to index outcomes of a nondegenerate measurement on a qubit. In this case, no isomorphisms of the appropriate type will exist.

Since the measurement 2\-cell~\eqref{eq:measurementdefinition} is required to be unitary, the physical process it represents can be fully inverted. The inverse process has the following graphical representation:
\begin{align}
\newtwocell{
\label{eq:removecorrelations}
    \node (V) [Vertex] at (1.5,-1) {};
    \draw [fill=\fillA, draw=none]
        (2,-1.6)
        to [out=up, in=\seangle] (V.center)
        to [out=\swangle, in=up] (1,-1.6);
    \draw [thicker]
        (1.5,-0.4)
        to (1.5,-1.0);
    \draw [thick] (1,-1.6)
        to [out=up, in=\swangle] (V.center)
        to [out=\seangle, in=up] (2,-1.6);
}
{Encode classical information}
{}
\end{align}
This can be interpreted as the preparation of  a quantum system, in a way which depends on the value of the classical information, such that the classical information can be perfectly recovered from the quantum system. It removes the correlations between the original systems, returning
a single system whose state encodes the original classical information. The classical information could still be directly available; for example, if it had previously been copied by a vertex of the form~\eqref{eq:copy}.

The invertibility equations take the following form:
\begin{calign}
\begin{aligned}
\begin{tikzpicture}[scale=1.5]
    \node (V) [Vertex] at (1.4,-1) {};
    \node (W) [Vertex] at (1.4,0) {};
    \draw [fill=\fillA, draw=none]
        (1.8,-1.5)
        to [out=up, in=\seangle] (V.center)
        to [out=\swangle, in=up] (1,-1.5);
    \draw [thicker]
        (1.4,-0.0)
        to (1.4,-1.0);
    \draw [thick] (1,-1.5)
        to [out=up, in=\swangle] (V.center)
        to [out=\seangle, in=up] (1.8,-1.5);
    \draw [fill=\fillA, draw=none]
        (1.8,0.5)
        to [out=down, in=\neangle] (W.center)
        to [out=\nwangle, in=down] (1,0.5);
    \draw [thick]
        (1.8,0.5)
        to [out=down, in=\neangle] (W.center)
        to [out=\nwangle, in=down] (1,0.5);
\end{tikzpicture}
\end{aligned}
\,\,=\quad
\begin{aligned}
\begin{tikzpicture}[scale=1.5]
\draw [fill=\fillA, draw=none] (0,0) rectangle (0.8,2);
\draw [thick] (0,0) to (0,2);
\draw [thick] (0.8,0) to (0.8,2);
\end{tikzpicture}
\end{aligned}
&
\begin{aligned}
\begin{tikzpicture}[scale=1.5]
    \node (V) [Vertex] at (1.4,1) {};
    \node (W) [Vertex] at (1.4,2) {};
    \draw [fill=\fillA, draw=none]
        (1.8,1.5)
        to [out=down, in=\neangle] (V.center)
        to [out=\nwangle, in=down] (1,1.5)
        to [out=up, in=\swangle] (W.center)
        to [out=\seangle, in=up] (1.8,1.5);
    \draw [thicker] (1.4,0.5) to (1.4,1.0);
    \draw [thicker] (1.4,2.5) to (1.4,2.0);
    \draw [thick] (1,1.5)
        to [out=down, in=\nwangle] (V.center)
        to [out=\neangle, in=down] (1.8,1.5)
        to [out=up, in=\seangle] (W.center)
        to [out=\swangle, in=up] (1,1.5);
\end{tikzpicture}
\end{aligned}
\quad=\quad
\begin{aligned}
\begin{tikzpicture}[scale=1.5]
\draw [white] (0.1,0.5) to (-0.0,0.5);
\draw [thicker] (0,0) to (0,2);
\end{tikzpicture}
\end{aligned}
\end{calign}
The first of these says that if we have some classical information which we encode using some basis in the state of a physical system, measuring that system in the same basis recovers the classical information. The second says that if we measure a physical system to obtain some classical information, and then encode this information back into the state of a physical system, then we have done nothing at all. From the perspective of quantum mechanics, a possible conceptual problem is that the result of the original measurement could be considered as still available in principle, in which case performing these two procedures does indeed have some nontrivial overall effect. Behind this argument is the assumption that correlations recording the measurement outcome rapidly become established with the environment. However, in principle, it is possible to reverse these correlations, and that is what our idealized encoding map~\eqref{eq:removecorrelations} represents. This is completely consistent with a pure-state perspective on quantum measurement, as described in particular by the decoherence programme~\cite{s05-decoherence}.

\subsubsection*{Arbitrary projective measurement}

More generally, we can describe an arbitrary projective measurement as a unitary 2\-cell of the following form:
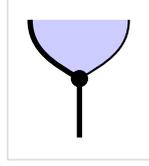
\begin{align}
\label{eq:arbitrarymeasurement}
\newtwocell{
    \node (V) [Vertex] at (1.5,1) {};
    \draw [fill=\fillA, draw=none]
        (2,1.6)
        to [out=down, in=\neangle] (V.center)
        to [out=\nwangle, in=down] (1,1.6);
    \draw [thicker]
        (1.5,0.4)
        to node [auto, swap] {} (1.5,1.0);
    \draw [thicker] (1,1.6)
        to [out=down, in=\nwangle] node [auto, swap, pos=0.4, inner sep=1pt] {} (V.center);
    \draw [thick] (V.center)
        to [out=\neangle, in=down] (2,1.6);
}
{Perform a projective measurement}
{}
\end{align}
Thicker lines represent represent arbitrary 1\-cells of the correct type, while the thinner line represents an object witnessing classical data as per Section~\ref{sec:witness}. The source of this 2\-cell is   a 1\-cell $I \sxto T I$ representing our original physical system, and the target is some composite $I \sxto {T_A} A \sxto A I$ representing the effect of the measurement. The 1\-cell $I \sxto {T_A} A$ is interpreted  as the original system $T$, modified such that information about the state of the system is stored in the object $A$. Quantum mechanically, this would correspond to the decomposition of $T$ into a family of orthogonal subspaces, indexed by the values of $A$. This system $I \sxto{T_A} A$ is correlated with the system $A \sxto A I$, in a way which is conditional on the value of the classical information associated to the object $A$.

\newcommand\supp{\text{im}}

For a concrete quantum-mechanical model, we can take the object $A$ to be single classical bit with values $\{0,1\}$, the 1\-cell $A \sxto A I$ to be the Hilbert space $\mathbb{C}^2$ with basis $\{ \ket 0 , \ket 1\}$ indexed by the values of $A$, the 1\-cell $I \sxto T I$ to be some Hilbert space, and the 1\-cell $I \sxto{T_A} A$ to be the Hilbert space $T$ equipped with a pair of orthogonal projectors $P_0, P_1:T \to T$ indexed by the values of $A$. The composite $I \sxto {T_A} A \sxto A I$ represents the subspace of $T \otimes \mathbb{C}^2$ spanned by states which are perfectly correlated with respect to the data monitored by the observable $A$, which is the space spanned by the images of the projectors $P_0 \otimes \ket 0 \bra 0$ and $P_1 \otimes \ket 1 \bra 1$. The measurement vertex then represents the linear map
\begin{equation}
\ket\phi \mapsto P_0 \ket \phi \otimes \ket 0 + P_1 \ket \phi \otimes \ket 1,
\end{equation}
where $\ket\phi$ is the original quantum state of the system $T$. This is indeed a unitary as required. A full proof of the equivalence of this notion of measurement with the ordinary one is given in Section~\ref{sec:2hilb}.

\subsubsection*{Controlled operations}

The final fundamental construction that we will consider is a controlled operation. It takes the following form:
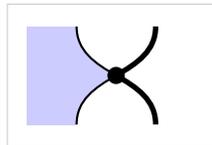
\begin{align}
\newtwocell{
\label{eq:controlled}
    \draw [white] (0.7,1) to (2.2,1);
    \draw [fill=\fillA, draw=none]
        (0.5,0.5)
        rectangle
        (1,1.5);
    \draw [fill=\fillA, draw=none] (1,1.5)
        to [out=down, in=\nwangle] (1.4,1)
        to [out=\swangle, in=up] (1,0.5);
    \draw [thick]
        (1,0.5)
        to [out=up, in=\swangle] (1.4,1)
        to [out=\nwangle, in=down] (1,1.5);
    \draw [thicker]
        (1.8,0.5)
        to [out=up, in=\seangle] (1.4,1)
        to [out=\neangle, in=down] (1.8,1.5);
    \node [Vertex] at (1.4,1) {};
}
{Perform a controlled operation}
{}
\end{align}
This is required to satisfy no equations: it is specified only by its type. It acts on a system witnessing classical information on the left-hand side, along with a separate uncorrelated system on the right-hand side.

The system encoding classical information cannot have its stage changed by this process, since it is perfectly correlated with the classical information which forms the context for the operation. However, the state of the uncorrelated system can be modified in an arbitrary fashion, and this can be done in a way which \textit{depends} on the state of the correlated system. This is precisely the definition of a controlled operation: a family of operations to be performed on one system, the choice of which depends on the state of another system.

In quantum theory we are often particularly interested in controlled families of unitary operations. This special case is obtained in our formalism by requiring the controlled operation vertex~\eqref{eq:controlled} to itself be unitary, as defined in Section~\ref{sec:dagger}.

\subsection{Internal algebras and modules}

Our formalism gives rise to algebras and modules in the symmetric monoidal category of scalars of our symmetric monoidal 2\-category. Intuitively, a witness for an object exhibits it as the category of modules for an algebra, a nondegenerate measurement gives rise to particular algebra, and an arbitrary measurement gives rise to a module for an algebra. All these algebras are internal commutative \dag\-Frobenius algebras. This gives a concrete connection to the categorical quantum mechanics research programme~\cite{cp06-qmws, cpv08-dfb} in which special commutative \dag\-Frobenius algebras, there called \textit{classical structures},  play a central role as abstract characterizations of bases for finite-dimensional Hilbert spaces.
\begin{theorem}
\label{thm:classicalstructure}
A witness for an object gives rise to a canonical special commutative \dag\-Frobenius algebra in the symmetric monoidal category of scalars.
\end{theorem}
\begin{proof}
Given witnesses $I \sxto A A$ and $A \sxto A I$, there is a commutative \dag\-Frobenius algebra on the object $I \sxto A A \sxto A I$. The multiplication is given by the following 2\-cell:
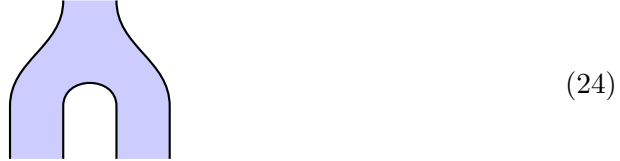
\begin{equation}
\label{eq:inducedmultiplication}
\begin{aligned}
\begin{tikzpicture}[scale=0.7]
\draw [fill=\fillA, draw=none] (0,1)
    to [out=up, in=down, out looseness=1.5] (0,2)
    to (1,2)
    to [out=down, in=up, out looseness=1.5] (1,1);
\draw [fill=\fillA, draw=none] (3,1)
    to [out=up, in=down, out looseness=1.5] (3,2)
    to (2,2)
    to [out=down, in=up, out looseness=1.5] (2,1);
\draw [fill=\fillA, draw=none] (0,2)
    to (1,2)
    to [out=up, in=up, looseness=1.5] (2,2)
    to (3,2)
    to [out=up, in=down] (2,4)
    to (1,4) to [out=down, in=up] (0,2);
\draw [thick] (3,1)
    to [out=up, in=down, in looseness=1.5] (3,2)
    to [out=up, in=down] (2,4);
\draw [thick] (0,1)
    to [out=up, in=down, in looseness=1.5] (0,2)
    to [out=up, in=down] (1,4);
\draw [thick] (2,1) 
    to (2,2)
    to [out=up, in=up, looseness=1.5] (1,2)
    to (1,1);
\end{tikzpicture}
\end{aligned}
\end{equation}
The unit is the 2\-cell~\eqref{eq:create} which creates uniform classical data. The associativity, unit and Frobenius equations follow from the duality equations~(\ref{eq:yank1}\_\ref{eq:yank4}) and the interchange rule for 2\-categories. The specialness axiom follows from the hole-deletion axiom~\eqref{eq:copycompare}.
\end{proof}

\noindent
The 1\-cell $I \sxto A A \sxto A I$ represents the witnessing system shorn of its interaction with the object $A$. The multiplication structure compares the classical data encoded in each tensor factor of the domain.

The following theorems examine the structures imposed by measurement operations.
\begin{theorem}
\label{thm:measurementclassicalstructure}
A nondegenerate measurement gives rise to a commutative \dag\-Frobenius algebra on the system being measured.
\end{theorem}
\begin{proof}
Immediate; given a unitary isomorphism between $I \sxto A A \sxto A I$ and some Hilbert space $H$, the classical structure constructed in Theorem~\ref{thm:classicalstructure} can be transported along the unitary to $H$. The result has the following representation:
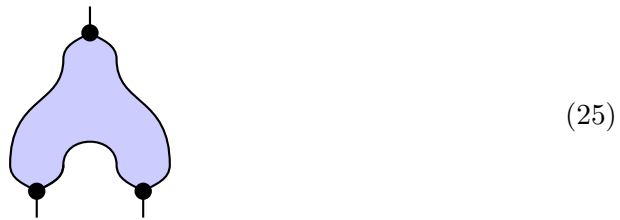
\begin{equation}
\begin{aligned}
\begin{tikzpicture}[scale=0.7]
    \node (V) [Vertex] at (1.5,0.5) {};
    \node (W) [Vertex] at (2.5,3.5) {};
    \node (X) [Vertex] at (3.5,0.5) {};
    \draw [fill=\fillA, draw=none]
        (2.5,3.5)
        to [out=\seangle, in=up] (3,3)
        to [out=down, in=up, in looseness=1.5] (4,1)
        to [out=down, in=\neangle] (3.5,0.5)
        to [in=down, out=\nwangle] (3,1)
        to [out=up, in=up, looseness=1.5] (2,1)
        to [out=down, in=\neangle] (V.center)
        to [out=\nwangle, in=down] (1,1)
        to [out=up, in=down, out looseness=1.5] (2,3)
        to [out=up, in=\swangle] (2.5,3.5);
    \draw [thick]
        (2.5,3.5)
        to [out=\seangle, in=up] (3,3)
        to [out=down, in=up, in looseness=1.5] (4,1)
        to [out=down, in=\neangle]  (3.5,0.5);
    \draw [thick] (3.5,0)
        to (3.5,0.5)
        to [out=\nwangle, in=down] (3,1)
        to [out=up, in=up, looseness=1.5] (2,1)
        to [out=down, in=\neangle] (V.center);
    \draw [thick]
        (1.5,0)
        to node [auto, swap] {} (1.5,0.5);
    \draw [thick] (2.5,4) to (2.5,3.5)
        to [out=\swangle, in=up] (2,3)
        to [out=down, in=up, in looseness=1.5] (1,1)
        to [out=down, in=\nwangle]
            node [auto, swap, pos=0.4, inner sep=1pt] {}
            (V.center);
    \draw [thick] (V.center)
        to [out=\neangle, in=down] (2,1);
\end{tikzpicture}
\end{aligned}
\end{equation}
\end{proof}

\begin{theorem}
A projective measurement gives rise to a module for the commutative \dag-Frobenius algebra arising from the associated witness.
\end{theorem}
\begin{proof}
A projective measurement is represented by a unitary of the following type:
\begin{equation}
\begin{aligned}
\begin{tikzpicture}[scale=1]
    \node (V) [Vertex] at (1.5,1) {};
    \draw [fill=\fillA, draw=none]
        (2,1.6)
        to [out=down, in=\neangle] (V.center)
        to [out=\nwangle, in=down] (1,1.6);
    \draw [thicker]
        (1.5,0.4)
        to node [auto, swap] {} (1.5,1.0);
    \draw [thicker] (1,1.6)
        to [out=down, in=\nwangle] node [auto, swap, pos=0.4, inner sep=1pt] {} (V.center);
    \draw [thick] (V.center)
        to [out=\neangle, in=down] (2,1.6);
\end{tikzpicture}
\end{aligned}
\end{equation}
We define our module action to be the following 2\-cell:
\begin{equation}
\begin{aligned}
\begin{tikzpicture}[scale=0.7]
    \node (V) [Vertex] at (1.5,0.5) {};
    \node (W) [Vertex] at (2.5,3.5) {};
    \draw [fill=\fillA, draw=none]
        (2.5,3.5)
        to [out=\seangle, in=up] (3,3)
        to [out=down, in=up, in looseness=1.5] (4,1) to (4,0) to (3,0) to (3,1)
        to [out=up, in=up, looseness=1.5] (2,1)
        to [out=down, in=\neangle] (V.center)
        to [out=\nwangle, in=down] (1,1)
        to [out=up, in=down, out looseness=1.5] (2,3)
        to [out=up, in=\swangle] (2.5,3.5);
    \draw [thick]
        (2.5,3.5)
        to [out=\seangle, in=up] (3,3)
        to [out=down, in=up, in looseness=1.5] (4,1)
        to (4,0);
    \draw [thick]  (3,0) to (3,1)
        to [out=up, in=up, looseness=1.5] (2,1)
        to [out=down, in=\neangle] (V.center);
    \draw [thicker]
        (1.5,0)
        to node [auto, swap] {} (1.5,0.5);
    \draw [thicker] (2.5,4) to (2.5,3.5)
        to [out=\swangle, in=up] (2,3)
        to [out=down, in=up, in looseness=1.5] (1,1)
        to [out=down, in=\nwangle]
            node [auto, swap, pos=0.4, inner sep=1pt] {}
            (V.center);
    \draw [thick] (V.center)
        to [out=\neangle, in=down] (2,1);
\end{tikzpicture}
\end{aligned}
\end{equation}
The associativity axiom and unit module axioms follow from the duality equations~(\ref{eq:yank1}\_\ref{eq:yank4}) and the interchange rule for 2\-categories.
\end{proof}

\section{2--Hilbert spaces}
\label{sec:2hilb}

\subsection{Introduction}

Applying the formalism described in Section~\ref{sec:formalism} in different symmetric monoidal 2\-categories gives us different interpretations of the constructions the formalism describes: classical data, physical systems, measurements and controlled operations. In this sense, we can think of a choice of target 2\-category as a `theory of physics' or `model of computation' in which we can work. To recover finite-dimensional quantum physics, we must work in \cat{2Hilb}, the symmetric monoidal 2\-category of finite-dimensional 2\_Hilbert spaces~\cite{b97-hda2}.

This 2\-category can be presented in different, equivalent ways. Each of these gives us an independent way to see why applying our formalism in \cat{2Hilb} should recover ordinary quantum theory. We will consider the following perspectives:
\begin{enumerate}
\item
\label{item:matrices}
Objects are natural numbers, 1\-cells are matrices of finite-dimensional Hilbert spaces, 2\-cells are matrices of bounded linear maps.
\item
\label{item:environments}
Objects are environments storing classical information, 1\-cells are quantum systems equipped with environmental interactions, 2\-cells are dynamical maps protected from decoherence.
\end{enumerate}
This section will focus on perspective~\ref{item:matrices}. We first describe the basic theory of 2\_Hilbert spaces. Then, taking perspective~\ref{item:matrices} as primary, we show in a series of theorems how applying the formal axiomatizations of classical information, measurement and controlled operations given in Section~\ref{sec:formalism} reproduces the conventional quantum notions. Perspective~\ref{item:environments} listed above will be discussed in Appendix~\ref{sec:environmental}.

Expounding these different perspectives has several benefits. Primarily, it provides a well-rounded intuition for why \cat{2Hilb} is the correct 2\-category for describing quantum theory. But it also places these ideas in a broader mathematical context, suggesting different paths to desirable generalizations, such as infinite-dimensional quantum theory or different theories of physics.

\subsection{Basic theory}
\label{sec:basic2hilb}

\subsubsection*{Definition}
We begin by summarizing some basic aspects of the theory of 2\_Hilbert spaces, as developed by Baez~\cite{b97-hda2}. A \textit{2\_Hilbert space} is an abelian category enriched over \cat{Hilb}, the category of finite-dimensional Hilbert spaces,  with a $\dag$\-structure, conjugate-linear on the hom-sets, satisfying
\begin{equation}
\langle f \circ g, h \rangle = \langle g , f^\dag \circ h \rangle = \langle f , h \circ g^\dag \rangle
\end{equation}
for all morphisms $f,g,h$ for which these composites make sense. This definition is analogous to the definition of an ordinary Hilbert space, as a complex vector space equipped with a positive-definite inner product. The 2\-category \cat{2Hilb} of 2\_Hilbert spaces has 2\_Hilbert spaces as objects, linear functors as 1\-cells, and natural transformations as 2\-cells. However, thanks to a structure theorem for finite-dimensional 2\_Hilbert spaces which we describe below, we will not need to work directly with this definition. 

\subsubsection*{Bases}

A \emph{basis} for a 2\_Hilbert space is a set of objects from which every object can be constructed by taking finite direct sums. A zero object is considered as an empty direct sum for this purpose. Every 2\_Hilbert space has a basis, and the \emph{dimension} $\dim(H)$ of a 2\_Hilbert space $H$ is the cardinality of its smallest basis. A 2\_Hilbert space is \emph{finite-dimensional} when $\dim(H)$ is finite. These definitions are directly analogous to the corresponding notions for a finite-dimensional Hilbert space. From now on we will assume all our 2\_Hilbert spaces are finite-dimensional. A functor $F$ between 2\_Hilbert spaces is \emph{linear} if, for all morphisms $f,g$ in the same homset of the domain, we have $F(f+g)=F(f)+F(g)$. We define \cat{2Hilb} to be the 2\-category with finite-dimensional 2\_Hilbert spaces as objects, linear functors as 1\-cells, and natural transformations as 2\-cells. Our concern in this paper is finite-dimensional quantum theory, and so from now on we focus on finite-dimensional 2\_Hilbert spaces; there is a general theory of 2\_Hilbert spaces as representation categories of arbitrary von Neumann algebras~\cite{bbfw08-idr}, which would be relevant for further investigations into the appropriate 2\-categorical structure for infinite-dimensional quantum theory.

A structure theorem says that for every finite-dimensional 2\_Hilbert space $H$, there is an equivalence
\begin{equation}
\label{eq:structuretheorem}
H \simeq \cat{Hilb} ^{\dim(H)},
\end{equation}
where the right-hand side represents the Cartesian product of $\dim(H)$ copies of the category of finite-dimensional Hilbert spaces. This gives an analogy once again to the theory of ordinary Hilbert spaces, where for every finite-dimensional Hilbert space $J$, we have an isomorphism
\begin{equation}
J \simeq \mathbb{C} ^{\dim(J)}.
\end{equation}
2\_Hilbert spaces are therefore quite tractable mathematical objects: up to equivalence, their objects are simply $n$-tuples of finite-dimensional Hilbert spaces, and their morphisms are $n$-tuples of linear maps.

Taking advantage of this structure theorem, we restrict our attention to 2\_Hilbert spaces of the form $\cat{Hilb}^n$ where $n$ is a natural number. In a 2\_Hilbert space, an object $S$ is \emph{simple} if it is not a zero object, and cannot be written as a direct sum of other objects in  nontrivial way. In the 2\_Hilbert space $\cat{Hilb}^n$ there are $n$ isomorphism classes of simple objects, of which a convenient family of representatives are  $(\mathbb{C},0,\ldots,0)$, $(0, \mathbb{C}, \ldots, 0)$, $\ldots$, $(0, 0, \ldots, \C)$. This family of objects provides a basis for $\cat{Hilb}^n$.

\subsubsection*{Matrix notation}

A linear functor $\cat{Hilb} ^n \to \cat{Hilb} ^m$ is completely defined up to isomorphism by its action on a basis of simple objects of $\cat{Hilb}^n$, each of which will be mapped independently by $F$ to an object of $\cat{Hilb}^m$. We can use this to represent a linear functor as a matrix of Hilbert spaces:
\begin{equation}
\label{eq:matrixexample}
\matrix{F_{1,1} & F_{1,2} & \cdots & F_{1,n}
\\
F_{2,1} & F_{2,2} & \cdots & F _{2,n}
\\
\vdots & \vdots & \ddots & \vdots
\\
F_{m,1} & F_{m,2} & \cdots & F_{m,n}}
\end{equation}
This is analogous to the matrix representation of a bounded linear map between finite-dimensional Hilbert spaces with chosen basis.

Up to isomorphism, composition of functors is given by matrix composition, with direct sum and tensor product taking the place of addition and multiplication for ordinary matrix multiplication. For example, we can compose the following functors of type $\cat{Hilb}^2 \to \cat{Hilb}^2$ in the following way:
\begin{align}
\nonumber
&\matrix{G_{1,1} & G_{1,2}
\\
G_{2,1} & G_{2,2}}
\circ
\matrix{H_{1,1} & H_{1,2}
\\
H_{2,1} & H_{2,2}}
\\
&\quad\simeq
\matrix{
(G_{1,1} \! \otimes \! H_{1,1}) \oplus (G_{1,2} \! \otimes \! H_{2,1}) & (G_{1,1} \! \otimes \! H_{1,2} )\oplus (G_{1,2} \! \otimes \! H_{2,2})
\\
(G_{2,1} \! \otimes \! H_{1,1}) \oplus (G_{2,2} \! \otimes \! H_{2,1}) & (G_{2,1} \! \otimes \! H_{1,2} )\oplus (G_{2,2} \! \otimes \! H_{2,2})}
\end{align}
This is only isomorphic to the composite, rather than strictly equal, since composition of functors is strictly associative, but this composition operation is not. We have a tradeoff here that is common in higher category theory: by making our space of 1\-cells easier to understand, a form of \emph{skeletality}, we lose good properties of composition of 1\-cells, a form of \emph{strictness}.

We can also use the matrix calculus to describe objects of our 2\_Hilbert spaces. Up to isomorphism, objects of a 2\_Hilbert space $\cat{Hilb}^n$ correspond to  functors $\cat{Hilb} \to \cat{Hilb}^n$, by considering the value taken by the functor on the object $\mathbb{C}$ in \cat{Hilb}. Using the matrix notation, an object $(H_1, \ldots, H_n)$ of $\cat{Hilb}^n$ corresponds to the following functor:
\begin{equation}
\matrix{H_1
\\
H_2
\\
\vdots
\\
H_n}
\end{equation}
The action of functors on objects can be calculated up to isomorphism by composition of functors.

\subsubsection*{Natural transformations}

A natural transformation $L: F \Rightarrow G$ between two matrices is given by a family of bounded linear maps $L_{i,j} : F_{i,j} \to G_{i,j}$. We write this as a matrix of linear maps, in the following way:
\begin{equation}
\matrix{ F_{1,1} & F_{1,2} \\ F_{2,1} & F_{2,2} }
\xrightarrow{ \matrix{ L_{1,1} & L_{1,2} \\ L_{2,1} & L_{2,2} } }
\matrix{ G_{1,1} & G_{1,2} \\ G_{2,1} & G_{2,2} }
\end{equation}
Vertical composition of 2\-cells, denoted in the graphical calculus by vertical juxtaposition, acts elementwise by composition of linear maps. Horizontal composition and tensor product act in a more complicated way, which we do not describe directly here.

There is a reversal operation $\dag$ on \cat{2Hilb}, in the sense of Section~\ref{sec:dagger}. In our matrix representation, it sends matrices of linear maps to matrices of their adjoints. So acting on the example above, it gives the following result:
\begin{equation}
\matrix{ G_{1,1} & G_{1,2} \\ G_{2,1} & G_{2,2} }
\xrightarrow{ \matrix{ L_{1,1} {}^\dag & L_{1,2} {}^\dag \\ L_{2,1} {}^\dag & L_{2,2} {}^\dag } }
\matrix{ F_{1,1} & F_{1,2} \\ F_{2,1} & F_{2,2} }
\end{equation}
This $\dag$-operation, called ``$*$'' in the paper~\cite{b97-hda2}, acts as the identity on 0\-cells and 1\-cells.

In the category of finite-dimensional Hilbert spaces, we have the following basic property of endomorphisms.
\begin{lemma}
In the category of finite-dimensional Hilbert spaces, for any two morphisms $H \sxto{f,g} H$, for any object $H$, we have $f \circ g = \id_H \Leftrightarrow g \circ f = \id_H$.
\end{lemma}
\begin{proof}
Suppose we have $f \circ g = \id_H$. Then $\det(f)\det(g) > 0$, so both $f$ and $g$ are invertible. So there exists some $g'$ with $g' \circ f = \id _H$. But then $g' = g' \circ f \circ g = g$.
\end{proof}

\noindent
This straightforwardly extends to 2\-cells in \cat{2Hilb}.

\begin{lemma}
\label{lem:fg1gf1}
In the 2\-category of finite-dimensional 2\_Hilbert spaces, for any two 2\-cells $F\smash{\xdoubleto{\smash{\sigma, \tau}}}F$, for any 1\-cell $F$, we have $\sigma \cdot \tau = \id _F \Leftrightarrow \tau \cdot \sigma = \id_F$.
\end{lemma}
\begin{proof}
Immediate from the previous lemma, since vertical composition of 2\-cells is given by composition of the constituent linear maps.
\end{proof}

\subsubsection*{Adjoints}

Every linear functor between finite-dimensional 2\_Hilbert spaces has a simultaneous left and right adjoint. In terms of the matrix formalism, this is constructed by constructing the dual Hilbert space for each entry in the matrix, and then taking the transpose of the matrix. This is analogous to the conjugate transpose operation that constructs the adjoint of a matrix representing a linear map between Hilbert spaces.

To take a concrete example, the adjoint $F^\dag$ to the functor $F$ presented above in expression~\eqref{eq:matrixexample} has the following form:
\begin{equation}
\label{eq:adjointmatrixexample}
\matrix{F_{1,1} ^* & F_{2,1} ^* & \cdots & F_{m,1} ^*
\\
F_{1,2} ^* & F_{2,2} ^* & \cdots & F _{m,2} ^*
\\
\vdots & \vdots & \ddots & \vdots
\\
F_{1,n} ^* & F_{2,n} ^* & \cdots & F_{m,n} ^*}
\end{equation}
To present the adjunction $F \dashv F ^\dag$, we must construct natural transformations $\id_{\cat{Hilb} _n} \smash{{} \xdoubleto {\smash{\sigma}}{}} F^\dag \circ F$ and $F \circ F^\dag \smash{{}\xdoubleto {\smash{\tau}}{}} \id_{\cat{Hilb}_m}$. These are defined in the following way, given unit and counit maps $\eta _{i,j} : \C \to F_{i,j} ^* \otimes F_{i,j}^\pstar$ and $\epsilon _{i,j} : F_{i,j} ^\pstar \otimes F_{i,j} ^* \to \C$:
\begin{calign}
\sigma = \matrix{\bigoplus_j \eta _{1,j} & 0 & \cdots & 0
\\
0 & \bigoplus_j \eta_{2,j} & \cdots & 0
\\
\vdots & \vdots & \ddots & \vdots
\\
0 & 0 & \cdots & \bigoplus_j \eta_{n,j}}
\\
\tau = \matrix{\bigoplus_j \epsilon _{j,1} & 0 & \cdots & 0
\\
0 & \bigoplus_j \epsilon_{j,2} & \cdots & 0
\\
\vdots & \vdots & \ddots & \vdots
\\
0 & 0 & \cdots & \bigoplus_j \epsilon_{j,n}}
\end{calign}
Conversely, every adjunction $F \dashv F^\dag$ is of this form for some family of unit and counit maps. Taking the adjoint of each of these linear maps gives data for the opposite adjunction $F ^\dag \dashv F$.

\subsection{Interpreting the formalism}

\subsubsection*{Interpreting objects, 1\-cells and 2\-cells}

\textbf{Objects.} In our abstract formalism, objects represent types of classical information. In \cat{2Hilb} objects are characterized up to equivalence by the natural numbers, so these form our classical data types.

\vspace{5pt}
\noindent
\textbf{1-cells.} A 1\-cell $A \sxto F B$ in our formalism represents a physical system, with classical information about its state encoded in the objects $A$ and $B$. In \cat{2Hilb}, such a 1\-cell is described up to equivalence by a matrix of Hilbert spaces, by the discussion in Section~\ref{sec:basic2hilb}. This represents a system with Hilbert space ${\bigoplus} _{ij} F _{ij}$, the direct sum of all the entries of the matrix. Knowing the values of the classical data types $A$ and $B$ gives a restriction to a row and column of the matrix, telling us the subspace in which our state must be located.

\vspace{5pt}
\noindent
\textbf{2-cells.} Given two systems $A \sxto{F,G} B$, a 2\-cell $F \doubleto G$ is a matrix of linear maps, one for each element of the set $A \times B$. This gives rise to a linear map ${\bigoplus} _{ij} F _{ij} \! \to {\bigoplus}_{ij} G_{ij}$, and hence provides quantum dynamics transforming the system represented by $F$ into the system represented by $G$ in the ordinary sense. However, by construction, this overall map preserves the partitions defined on these spaces. The row and column indexes enumerating the partitions are dynamical constants, which the quantum evolution is required to preserve.

\subsubsection*{Witnesses for classical data}

 Following Section~\ref{sec:witness}, a witness for an object $\cat{Hilb}^n$ of \cat{2Hilb} is a pair of 1\-cells $\cat{Hilb} \sxto n \cat{Hilb}^n$ and $\cat{Hilb}^n \sxto n \cat{Hilb}$, equipped with 2\-cells~(\ref{eq:create}\_\ref{eq:compare}) satisfying equations~(\ref{eq:yank1}\_\ref{eq:commplanar}).

\newcommand\drawboundary{    \draw [black!20]
        ([xshift=-\boxmargin, yshift=-\boxmargin] current bounding box.south west)
        rectangle
        ([xshift=\boxmargin, yshift=\boxmargin] current bounding box.north east);
}
\begin{theorem}
\label{thm:2hilbwitness}
In \cat{2Hilb}, each object $\cat{Hilb}^n$ has a unique witness up to unitary isomorphism, for which the 1\-cell part consists of the $n$-element matrices
\begin{calign}
\begin{aligned}
\begin{tikzpicture}[scale=1]
\draw [white] (0,0) to (1,2);
\draw [fill=\fillA, draw=none] (0,0)
    to (1,0)
    to (1,2)
    to (0,2);
\draw [thick] (0,0) to (0,2);
\end{tikzpicture}
\end{aligned}
\,\,=\,\,
\matrix{\C \\ \C \\ \vdots \\ \C }
&
\begin{aligned}
\begin{tikzpicture}[scale=1]
\draw [white] (0,0) to (1,2);
\draw [fill=\fillA, draw=none] (0,0)
    to (1,0)
    to (1,2)
    to (0,2);
\draw [thick] (1,0) to (1,2);
\end{tikzpicture}
\end{aligned}
\,\,=\,\,
\matrix{\,\C & \C & \cdots & \C \!}
\end{calign}
and the 2\-cell part consists of the following structures, where each matrix has either 1 or $n$ rows or columns:
\def\extraspace{\hspace*{0pt}}
\begin{align}
\label{eq:concretecopy}
\begin{array}{c}
\begin{aligned}
\begin{tikzpicture}
    \draw [fill=\fillA, draw=none] (0.2,-0.5)
        to (0.5,-0.5)
        to [out=down, in=down, looseness=1.5] (1.5,-0.5)
        to (1.8,-0.5)
        to (1.8,-1.5)
        to (0.2,-1.5)
        to (0.2,-0.5);
    \draw [thick] (0.5,-0.5)
        to [out=down, in=down, looseness=1.5] (1.5,-0.5);
    \drawboundary
\end{tikzpicture}
\end{aligned}
\\
\text{\em Copy}
\end{array}
&=
\,\,
\matrix{
\C & 0 & \cdots & 0
\\
0 & \C & \cdots & 0
\\
\vdots & \vdots & \ddots & \vdots
\\
0 & 0 & \cdots & \C}
\xrightarrow{\matrix{
\matrix{1} & \matrix{0_{0,\C}} & \cdots & \matrix{0_{0,\C}}
\\
\matrix{0_{0,\C}} & \matrix{1} & \cdots & \matrix{0_{0,\C}}
\\
\vdots & \vdots & \ddots & \vdots
\\
\matrix{0_{0,\C}} & \matrix{0_{0,\C}} & \cdots & \matrix{1}}
}
\matrix{
\C & \C & \cdots & \C
\\
\C & \C & \cdots & \C
\\
\vdots & \vdots & \ddots & \vdots
\\
\C & \C & \cdots & \C}
\extraspace
\\
\label{eq:concretecompare}
\begin{array}{c}
\begin{aligned}
\begin{tikzpicture}
    \draw [fill=\fillA, draw=none] (0.2,0.5)
        to (0.5,0.5)
        to [out=up, in=up, looseness=1.5] (1.5,0.5)
        to (1.8,0.5)
        to (1.8,1.5)
        to (0.2,1.5)
        to (0.2,0.5);
    \draw [thick] (0.5,0.5)
        to [out=up, in=up, looseness=1.5] (1.5,0.5);
    \drawboundary
\end{tikzpicture}
\end{aligned}
\\
\text{\em Compare}
\end{array}
&=\,\,
\matrix{
\C & \C & \cdots & \C
\\
\C & \C & \cdots & \C
\\
\vdots & \vdots & \ddots & \vdots
\\
\C & \C & \cdots & \C} 
\xrightarrow{\matrix{
\matrix{1} & \matrix{0_{\C,0}} & \cdots & \matrix{0_{\C,0}}
\\
\matrix{0_{\C,0}} & \matrix{1} & \cdots & \matrix{0_{\C,0}}
\\
\vdots & \vdots & \ddots & \vdots
\\
\matrix{0_{\C,0}} & \matrix{0_{\C,0}} & \cdots & \matrix{1}}
}
\matrix{
\C & 0 & \cdots & 0
\\
0 & \C & \cdots & 0
\\
\vdots & \vdots & \ddots & \vdots
\\
0 & 0 & \cdots & \C}
\extraspace
\\
\label{eq:concretecreate}
\begin{array}{c}
\begin{aligned}
\begin{tikzpicture}
    \draw [white] (0.2,1) to (1.8,1);
    \draw [fill=\fillA, thick] (0.5,1.5)
        to [out=down, in=down, looseness=1.5] (1.5,1.5);
    \draw [thick, white] (1,0.5) to (1,1);
    \drawboundary
\end{tikzpicture}
\end{aligned}
\\
\text{\em Create}
\end{array}
&=\,\,
\matrix { \C }
\xrightarrow{
\matrix{
\matrix{1
\\
1
\\
\vdots
\\
1
}}}
\matrix { \C ^n }
\\[20pt]
\label{eq:concretedelete}
\begin{array}{c}
\begin{aligned}
\begin{tikzpicture}
    \draw [white] (0.2,-1) to (1.8,-1);
    \draw [fill=\fillA, thick] (0.5,-1.5)
        to [out=up, in=up, looseness=1.5] (1.5,-1.5);
    \draw [thick, white] (1,-0.5) to (1,-1);
    \drawboundary
\end{tikzpicture}
\end{aligned}
\\
\text{\em Delete}
\end{array}
&=\,\,
\matrix { \C^n }
\xrightarrow{
\Big(
\matrix{1
&
1
&
\cdots
&
1\hspace{-6pt}
} \Big)}
\matrix { \C }
\end{align}
\end{theorem}
\begin{proof}
We suppose the witness $\cat{Hilb}^n \sxto n \cat{Hilb}$ has the fully general form $(V_1 \, V_2 \, \cdots \, V_n)$. This induces a comparison 2\-cell~\eqref{eq:compare} as follows, since it is the counit for an adjunction:
\begin{equation}
\hspace{-100pt}
\matrix{
V_1 ^* \otimes V_1 ^\pstar & V_1 ^* \otimes V_2 ^\pstar & \cdots & V_1 ^* \otimes V_n ^\pstar
\\
V_2 ^* \otimes V_1 ^\pstar & V_2 ^* \otimes V_2 ^\pstar & \cdots & V_2 ^* \otimes V_n ^\pstar
\\
\vdots & \vdots & \ddots & \vdots
\\
V_n ^* \otimes V_1 ^\pstar & V_n ^* \otimes V_2 ^\pstar & \cdots & V_n ^* \otimes V_n ^\pstar
}
\xrightarrow{\matrix{
\epsilon _1 & 0 & \cdots & 0
\\
0 & \epsilon_2 & \cdots & 0
\\
\vdots & \vdots & \ddots & \vdots
\\
0 & 0 & \cdots & \epsilon_n
}}
\matrix{\C & 0 & \cdots & 0
\\
0 & \C & \cdots & 0
\\
\vdots & \vdots & \ddots & \vdots
\\
0 & 0 & \cdots & \C}
\hspace{-100pt}
\end{equation}
Here the linear maps $V_i ^* \otimes V_i ^\pstar \sxto {\epsilon_i} \C$ are part of a pair witnessing the duality between the vector spaces $V_i ^\pstar$ and $V_i ^*$.

We now consider the multiplication 2\-cell~\eqref{eq:inducedmultiplication} formed by composing this 2\-cell horizontally with our witnesses. This is a 2\-cell
\begin{equation}
\bigoplus_{i,j} V_i ^\pstar \otimes V_i^* \otimes V_j ^\pstar \otimes V_j ^* \to \bigoplus _k V_k ^\pstar \otimes V_k ^*,
\end{equation}
which for elements $\phi, \chi \in V_i$ and $\psi, \omega \in V_j$ acts in the following way on the $i,j$ summand:
\begin{equation}
\phi \otimes \chi ^* \otimes \psi \otimes \omega ^* \mapsto \delta_{i,j} \, \phi \otimes \epsilon_i (\chi ^* \otimes \psi) \otimes \omega^*
\end{equation}
We now impose requirement~\eqref{eq:commplanar}, which says that this composite must be commutative. The swap operator takes the $i,j$ summand to the $j,i$ summand, swapping the factors $V_i ^\pstar \otimes V_i^*$ and $V_j ^\pstar \otimes V_j ^*$ using the symmetry map of the category of Hilbert spaces. The commutativity law therefore gives the following equation for all $\phi, \chi, \psi, \omega \in V_i$:
\begin{equation}
\phi \otimes \epsilon_i (\chi ^* \otimes \psi) \otimes \omega^* = \psi \otimes \epsilon_i (\omega ^* \otimes \phi) \otimes \chi^*
\end{equation}
Graphically, this has the following representation: \begin{equation}
\begin{aligned}
\begin{tikzpicture}[scale=0.7]
\draw [thick] (0,0)
    to [out=up, in=down, out looseness=1.5] (2,2)
    to [out=up, in=up, looseness=1.5] (1,2)
    to [out=down, in=up, in looseness=1.5] (3,0);
\draw [thick] (1,0)
    to [out=up, in=down, in looseness=1.5] (3,2)
    to [out=up, in=down] (2,4);
\draw [thick] (2,0)
    to [out=up, in=down, in looseness=1.5] (0,2)
    to [out=up, in=down] (1,4);
\end{tikzpicture}
\end{aligned}
\quad=\quad
\begin{aligned}
\begin{tikzpicture}[scale=0.7]
\draw [thick] (3,0)
    to [out=up, in=down, in looseness=1.5] (3,2)
    to [out=up, in=down] (2,4);
\draw [thick] (0,0)
    to [out=up, in=down, in looseness=1.5] (0,2)
    to [out=up, in=down] (1,4);
\draw [thick] (2,0) 
    to (2,2)
    to [out=up, in=up, looseness=1.5] (1,2)
    to (1,0);
\end{tikzpicture}
\end{aligned}
\end{equation}
This can only hold if each vector space $V_i$ is 1\-dimensional.

From the normalization equation~\eqref{eq:copycompare}, we obtain that each of the coefficients $\epsilon_i$ is a unit complex number. Since we work only up to unitary isomorphism, we may choose these all equal to 1. This fixes the form of~\eqref{eq:concretecopy} and~\eqref{eq:concretecompare}. The topological equations~(\ref{eq:yank1})\_(\ref{eq:yank4}) then directly imply the forms of expressions~\eqref{eq:concretecreate} and~\eqref{eq:concretedelete}.
\end{proof}

\noindent
This theorem tells us that the classical data encoded by the object $\cat{Hilb}^n$ is witnessed by a family of $n$ one-dimensional quantum systems. This confirms the fact that it represents the classical $n$-element set in our theory.

\subsubsection*{Measurement}

\begin{theorem}
\label{thm:2hilbnondegenmeasurement}
In \cat{2Hilb}, a nondegenerate measurement on a Hilbert space corresponds  an ordered orthonormal basis.
\end{theorem}
\begin{proof}
A nondegenerate measurement is a unitary 2\-cell with domain given by some Hilbert space $H$, and target given by the composite $\cat{Hilb} \sxto n \cat{Hilb}^n \sxto n \cat{Hilb}$ of witnesses for the object $\cat{Hilb}^n$. By Theorem~\ref{thm:2hilbwitness}, we can use the matrix calculus to write this composite as a 2\-cell of the form
\begin{equation}
\matrix{\,\C & \C & \cdots & \C \!\!}
\circ
\matrix{\C \\ \C \\ \vdots \\ \C}
\,\,=\,\,
\C \oplus \C \oplus \cdots \oplus \C
\end{equation}
with $n$ summands. A unitary isomorphism between this space and some Hilbert space $H$ exactly comprises an ordered orthonormal basis for $H$.
\end{proof}

\begin{theorem}
\label{thm:2hilbmeasurement}
In \cat{2Hilb}, a general projective measurement corresponds to a decomposition of a Hilbert space into a set of orthogonal subspaces.
\end{theorem}
\begin{proof}
A general projective measurement is defined as a unitary 2\-cell with target given by a composite $\cat{Hilb} \sxto F \cat{Hilb}^n \sxto n \cat{Hilb}$, where $F$ is an arbitrary 1\-cell and $n$ is a 1\-cell acting as a witness for the object $\cat{Hilb}^n$. In the matrix calculus, we can perform this composition in the following way:
\begin{align}
\nonumber
\matrix{\,\C & \C & \cdots & \C \!\!}
\circ
\matrix{F_1 \\ F_2 \\ \vdots \\ F_n}
&\,\,=\,\, (F_1 \otimes \C) \oplus (F_2\otimes \C) \oplus \cdots \oplus (F_n \otimes \C)
\\[-20pt]
&\hspace{40pt}\simeq\,\, F_1 \oplus F_2 \oplus \cdots \oplus F_n
\end{align}
A unitary isomorphism from some Hilbert space $H$ to this space exactly corresponds to an orthogonal decomposition of $H$ by the subspaces $F_i$.
\end{proof}

\subsubsection*{Controlled operations}

\begin{theorem}
\label{thm:controlled}
In \cat{2Hilb}, a controlled operation on $H$ indexed by an object $\cat{Hilb}^n$ corresponds to an independent choice of $n$ endomorphisms of $H$.
\end{theorem}
\begin{proof}
The source and target of a controlled operation in our formalism is a 1\-cell of the the following form:
\begin{calign}
\begin{aligned}
\begin{tikzpicture}
\draw [white] (0,0) to (2,2);
\draw [fill=\fillA, draw=none] (0,0)
    to (1,0)
    to (1,2)
    to (0,2);
\draw [thick] (1,0) to node [auto, swap] {} (1,2);
\draw [thicker] (2,0) to (2,2);
\end{tikzpicture}
\end{aligned}
\end{calign}
Here, the left-hand line is a physical witness for an object $\cat{Hilb}^n$, shaded red in the diagram, and the right-hand line represents some Hilbert space $H$. Writing them out in terms of the matrix calculus, we can evaluate their composite as follows:
\begin{equation}
\matrix{\C \\ \C \\ \vdots \\ \C}
\circ
\matrix {H}
\,\,\simeq\,\,
\matrix{H \\ H \\ \vdots \\ H}
\end{equation}
The column matrices here each have $n$ rows. The space of endomorphisms of this composite is clearly a list of $n$ endomorphisms of $H$.
\end{proof}

\noindent
It is clear that a particular endomorphism of the composite is invertible or unitary iff each individual endomorphism of $H$ is invertible or unitary.

\section{Application to quantum information}
\label{sec:quantuminformation}

\subsection{Introduction}

\subsubsection*{Overview}

The specifications for some quantum protocols can be encoded abstractly as equations between composites of our graphical components. For example, in the introduction we gave the example of teleportation, with the following specification:
\begin{equation}
\label{eq:teleport}
\begin{aligned}
\begin{tikzpicture}
\node (V) [Vertex] at (1,0) {};
\node (W) [Vertex] at (2.25,1) {};
\draw [fill=\fillA, thick] (0,2)
    to [] (0,1)
    to [out=down, in=\nwangle] (1,0)
    to [out=\neangle, in=\swangle] (W.center)
    to [out=\nwangle, in=down] (1.5,2);
\draw [thick] (0.5,-1) to [out=up, in=down] (V.s1) to (V.s2) to [out=down, in=down] (3,0) to [out=up, in=\seangle] (W.center) to [out=\neangle, in=down] (3,2);
\end{tikzpicture}
\end{aligned}
\qquad=\qquad
\begin{aligned}
\begin{tikzpicture}
\draw [fill=\fillA, thick] (0,2)
    to (0,1.5)
    to [out=down, in=down, looseness=1.5] (1.5,1.5)
    to (1.5,2);
\draw [thick] (0.75,-1) to [out=up, in=down] (3,1.5) to (3,2);
\end{tikzpicture}
\end{aligned}
\end{equation}
The vertices represent arbitrary unitary operations of the correct type, so  this is an equation in two unknowns. Implementations of the specification in quantum theory --- that is, actual sets of instructions for carrying out quantum teleportation in the real world --- correspond to solutions to the equation in the 2\-category \cat{2Hilb}. In this section we investigate quantum teleportation, dense coding, complementarity and quantum erasure from this perspective, gaining new insights into their mathematical structure.

\subsubsection*{Computational verification}

Given a suggested solution to a specification --- that is, a suggested implementation of a quantum procedure --- manually verifying whether the relevant graphical equation is satisfied is not necessarily straightforward. While elegant and natural at a high level, the nuts-and-bolts of the required 2\-dimensional algebra can be cumbersome and difficult to manipulate in their raw form, involving matrices of matrices which can be composed in three distinct ways. For this reason, it makes sense to use a computer algebra package to take care of the low-level details.

We use a preview version of \texttt{TwoVect.m}~\cite{r11-2vect}, an add-on to \textit{Mathematica} that implements all the basic operations of the 2\-category of 2\_vector spaces. This is soon to be made publicly available. Using such a package, verifying that a graphical equation is solved by particular choices of 2\-cells becomes an easy calculation, often only requiring a single line of code. In addition, since the fragment of \cat{2Hilb} of relevance to this work is represented in a strict and skeletal way within \texttt{TwoVect.m}, no additional complications are introduced by structural isomorphisms.

For the quantum teleportation, dense coding, complementarity  and quantum erasure examples, we provide code to demonstrate that the conventional implementations provide solutions to the graphical equations we provide. A \textit{Mathematica} notebook containing these calculations is available at the author's website~\cite{v12-highernotebook}. In that notebook, and in places throughout this section, we use the following  standard notation for the basic 1\-cells and 2\-cells of the formalism:
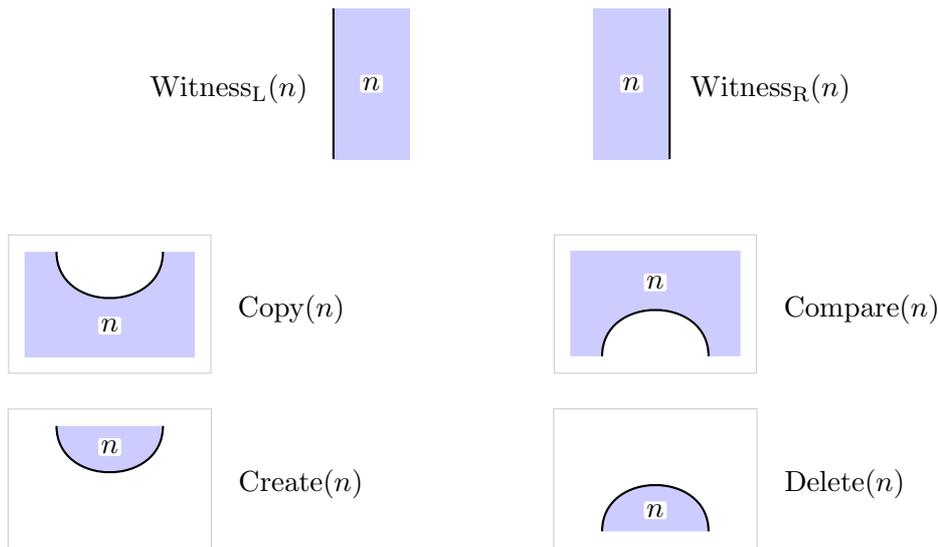
\begin{calign}
\nonumber
\text{Witness}_\text{L}(n)
\begin{aligned}
\begin{tikzpicture}
\draw [fill=\fillA, draw=none] (0,0)
    to (-1,0)
    to (-1,2)
    to (0,2);
\draw [thick] (-1,0) to node [auto] {} (-1,2);
    \node [fill=white, rounded corners=1pt, inner sep=1pt] at (-0.5,1) {$n$};
\end{tikzpicture}
\end{aligned}
&
\begin{aligned}
\begin{tikzpicture}
\draw [fill=\fillA, draw=none] (0,0)
    to (1,0)
    to (1,2)
    to (0,2);
\draw [thick] (1,0) to node [auto, swap] {} (1,2);
    \node [fill=white, rounded corners=1pt, inner sep=1pt] at (0.5,1) {$n$};
\end{tikzpicture}
\end{aligned}
\text{Witness}_\text{R}(n)
\end{calign}
\def\aascale{1.4}
\def\aaspace{\hspace{30pt}}
\setlength\fboxsep{0pt}
\def\sep{5pt}
\def\innersep{3pt}
\def\littlegap{12pt}
\def\boxmargin{0.15cm}
\renewcommand{\centerdia}[1]{#1}
\renewcommand\newtwocell[2]{\begin{aligned}
\begin{tikzpicture}[scale=\aascale]
    #1
    \draw [black!20]
        ([xshift=-\boxmargin, yshift=-\boxmargin] current bounding box.south west)
        rectangle
        ([xshift=\boxmargin, yshift=\boxmargin] current bounding box.north east);
\end{tikzpicture}
\end{aligned}
\hspace{10pt} \makebox[80pt][l]{\vc{#2}}}
\renewcommand\separatetwocells{\\[\sep]}
\allowdisplaybreaks[2]
\begin{calign}
\nonumber
\newtwocell{
    \draw [fill=\fillA, draw=none] (0.2,-0.5)
        to (0.5,-0.5)
        to [out=down, in=down, looseness=1.5] (1.5, -0.5)
        to (1.8,-0.5)
        to (1.8,-1.5)
        to (0.2,-1.5)
        to (0.2,-0.5);
    \draw [thick] (0.5,-0.5)
        to [out=down, in=down, looseness=1.5] (1.5,-0.5);
    \node [fill=white, rounded corners=1pt, inner sep=1pt] at (1,-1.2) {$n$};
}
{Copy($n$)}
\hspace{\littlegap}&\hspace{\littlegap}
\newtwocell{
    \draw [fill=\fillA, draw=none] (0.2,0.5)
        to (0.5,0.5)
        to [out=up, in=up, looseness=1.5] (1.5,0.5)
        to (1.8,0.5)
        to (1.8,1.5)
        to (0.2,1.5)
        to (0.2,0.5);
    \draw [thick] (0.5,0.5)
        to [out=up, in=up, looseness=1.5] (1.5,0.5);
    \node [fill=white, rounded corners=1pt, inner sep=1pt] at (1,1.2) {$n$};
}
{Compare($n$)}
\separatetwocells
\nonumber
\newtwocell{
    \draw [white] (0.2,1) to (1.8,1);
    \draw [fill=\fillA, thick] (0.5,1.5)
        to [out=down, in=down, looseness=1.5] (1.5,1.5);
    \draw [thick, white] (1,0.5) to (1,1);
    \node [fill=white, rounded corners=1pt, inner sep=1pt] at (1,1.3) {$n$};
}
{Create($n$)}
\hspace{\littlegap}&\hspace{\littlegap}
\newtwocell{
    \draw [white] (0.2,-1) to (1.8,-1);
    \draw [fill=\fillA, thick] (0.5,-1.5)
        to [out=up, in=up, looseness=1.5] (1.5,-1.5);
    \draw [thick, white] (1,-0.5) to (1,-1);
    \node [fill=white, rounded corners=1pt, inner sep=1pt] at (1,-1.3) {$n$};
}
{Delete($n$)}
\end{calign}
The label $n$ here indicates that the region corresponds to the object $\cat{Hilb}^n$ in \cat{2Hilb}, playing the role of a piece of classical data that can take $n$ values. The 2\-cells corresponding to these these components in \cat{2Hilb} are as described in Theorem~\ref{thm:2hilbwitness}. Also, we write $Q$ for the Hilbert space $\mathbb{C}^2$ seen as a 1\-cell of type $\cat{Hilb} \to \cat{Hilb}$ in \cat{2Hilb}.

\subsection{Mathematical aspects}
\label{sec:mathematicalaspects}

\subsubsection*{Introduction}

Here we introduce two mathematical developments which are useful for our analysis later in this section.

\subsubsection*{The theory of a specification}

Given the specification for a quantum procedure, we define the \textit{theory} of the specification to be the free symmetric monoidal 2\-category on generators given by the components used in the specification, modulo the relation given by the specification equation itself. A good reference for the theory of presentations of monoidal 2\-categories by generators and relations is the thesis of Schommer-Pries~\cite{sp11-thesis}.

This formalism has the benefit of making implementations of specifications easy to talk about using the language of category theory. For the theory $\cat T$ of some specification S, an implementation is precisely a symmetric monoidal 2\-functor $F$ of the following type:
\begin{equation}
\label{eq:implementationfunctor}
\cat T_{\text{S}} \sxto {F} \cat{2Hilb}
\end{equation}
In this way, the study of implementations of quantum protocols becomes a part of representation theory. In particular, we are close to the definition of a TQFT, as a symmetric monoidal $n$-functor
\begin{equation}
\cat{nCob} \sxto Z \cat{nHilb}.
\end{equation}
Exploring these connections has the potential to lead to closer relationships between representation theory and quantum information.

Working with theories for specifications also gives a natural way to describe equivalences between specifications. Taken together, Theorems~\ref{thm:teleportationdoublyunitary} and~\ref{thm:densecodingdoublyunitary} later in this section show that implementations of dense coding and teleportation are equivalent. Formally, the most elegant way to say this is that we obtain a \emph{symmetric monoidal equivalence} between the theories of teleportation and dense coding:
\begin{equation}
\begin{aligned}
\begin{tikzpicture}
\node (T) at (0,0) {$\cat{T} _{\text{T}}$};
\node (DC) at (2,0) {$\cat{T} _{\text{DC}}$};
\draw [->] (T) to [out=45, in=135] (DC);
\draw [->] (DC) to [out=-135, in=-45] (T);
\node at (1,0) {$\simeq$};
\end{tikzpicture}
\end{aligned}
\end{equation}
This allows us to convert an implementation of teleportation into an implementation of dense coding, and vice-versa, by pre-composing the appropriate equivalence with the implementation functor~\eqref{eq:implementationfunctor}.

\subsubsection*{Horizontal invertibility}

Suppose we have a 2\-cell of the following form:
\def\maxheight{0.8}
\begin{calign}
\label{eq:example2cell}
\begin{aligned}
\begin{tikzpicture}[thick]
\begin{pgfonlayer}{foreground}
\node (f) [draw, minimum width=30pt, thick, fill=white] at (0,0) {$f$};
\end{pgfonlayer}
\draw [fill=\fillD, draw=none] (0,\maxheight -| f.40) to (1,\maxheight) to (1,-\maxheight) to (0,-\maxheight -| f.40);
\draw [fill=\fillA, draw=none] (0,\maxheight -| f.140) to (-1,\maxheight) to (-1,-\maxheight) to (0,-\maxheight -| f.140);
\draw [fill=\fillB, draw=none]  (f.40) to (0,\maxheight -| f.40) to (0,\maxheight -| f.140) to (f.140);
\draw [fill=\fillC, draw=none]  (f.-40) to (0,-\maxheight -| f.-40) to (0,-\maxheight -| f.-140) to (f.-140);
\draw (f.40) to (0,\maxheight -| f.40) node [above] {$G$};
\draw (f.140) to (0,\maxheight -| f.140) node [above] {$F$};
\draw (f.-40) to (0,-\maxheight -| f.-40) node [below] {$J$};
\draw (f.-140) to (0,-\maxheight -| f.-140) node [below] {$H$};
\end{tikzpicture}
\end{aligned}
\ignore{
&
\begin{aligned}
\begin{tikzpicture}[thick]
\begin{pgfonlayer}{foreground}
\node (f) [draw, minimum width=30pt, thick, fill=white] at (0,0) {$g\vphantom{f}$};
\end{pgfonlayer}
\draw [fill=\fillD, draw=none] (0,\maxheight -| f.40) to (1,\maxheight) to (1,-\maxheight) to (0,-\maxheight -| f.40);
\draw [fill=\fillA, draw=none] (0,\maxheight -| f.140) to (-1,\maxheight) to (-1,-\maxheight) to (0,-\maxheight -| f.140);
\draw [fill=\fillC, draw=none]  (f.40) to (0,\maxheight -| f.40) to (0,\maxheight -| f.140) to (f.140);
\draw [fill=\fillB, draw=none]  (f.-40) to (0,-\maxheight -| f.-40) to (0,-\maxheight -| f.-140) to (f.-140);
\draw (f.40) to (0,\maxheight -| f.40) node [above] {$J$};
\draw (f.140) to (0,\maxheight -| f.140) node [above] {$H$};
\draw (f.-40) to (0,-\maxheight -| f.-40) node [below] {$G$};
\draw (f.-140) to (0,-\maxheight -| f.-140) node [below] {$F$};
\end{tikzpicture}
\end{aligned}}
\end{calign}
Then it is invertible if there exists a 2\-cell $g$ of the opposite type such that the following equations are satisfied:
\begin{calign}
\begin{aligned}
\begin{tikzpicture}[thick]
\begin{pgfonlayer}{foreground}
\node (f) [draw, minimum width=30pt, thick, fill=white] at (0,0) {$f$};
\node (g) [draw, minimum width=30pt, thick, fill=white] at (0,1.1) {$g\vphantom{f}$};
\end{pgfonlayer}
\draw [fill=\fillD, draw=none] (0,1.1+\maxheight -| f.40) to (1,1.1+\maxheight) to (1,-\maxheight) to (0,-\maxheight -| f.40);
\draw [fill=\fillA, draw=none] (0,1.1+\maxheight -| f.140) to (-1,1.1+\maxheight) to (-1,-\maxheight) to (0,-\maxheight -| f.140);
\draw [fill=\fillC, draw=none]  (0,1.1+\maxheight -| f.-40) to (0,-\maxheight -| f.-40) to (0,-\maxheight -| f.-140) to (0,1.1+\maxheight -| f.-140);
\draw [fill=\fillB, draw=none]  (f.40) to (0,\maxheight -| f.40) to (0,\maxheight -| f.140) to (f.140);
\draw (f.40) to (0,\maxheight -| f.40);
\draw (f.140) to (0,\maxheight -| f.140);
\draw (g.40) to (0,1.1+\maxheight -| g.40);
\draw (g.140) to (0,1.1+\maxheight -| g.140);
\draw (f.-40) to (0,-\maxheight -| f.-40);
\draw (f.-140) to (0,-\maxheight -| f.-140);
\end{tikzpicture}
\end{aligned}
\,\,\,=\,\,\,
\begin{aligned}
\begin{tikzpicture}[thick]
\draw [fill=\fillD, draw=none] (0,1.1+\maxheight -| f.40) to (1,1.1+\maxheight) to (1,-\maxheight) to (0,-\maxheight -| f.40);
\draw [fill=\fillA, draw=none] (0,1.1+\maxheight -| f.140) to (-1,1.1+\maxheight) to (-1,-\maxheight) to (0,-\maxheight -| f.140);
\draw [fill=\fillC, draw=none]  (0,1.1+\maxheight -| f.-40) to (0,-\maxheight -| f.-40) to (0,-\maxheight -| f.-140) to (0,1.1+\maxheight -| f.-140);
\draw (0,-\maxheight -| f.40) to (0,1.1+\maxheight -| f.40);
\draw (0,-\maxheight -| f.140) to (0,1.1+\maxheight -| f.140);
\end{tikzpicture}
\end{aligned}
&
\begin{aligned}
\begin{tikzpicture}[thick]
\begin{pgfonlayer}{foreground}
\node (f) [draw, minimum width=30pt, thick, fill=white] at (0,0) {$g\vphantom{f}$};
\node (g) [draw, minimum width=30pt, thick, fill=white] at (0,1.1) {$f$};
\end{pgfonlayer}
\draw [fill=\fillD, draw=none] (0,1.1+\maxheight -| f.40) to (1,1.1+\maxheight) to (1,-\maxheight) to (0,-\maxheight -| f.40);
\draw [fill=\fillA, draw=none] (0,1.1+\maxheight -| f.140) to (-1,1.1+\maxheight) to (-1,-\maxheight) to (0,-\maxheight -| f.140);
\draw [fill=\fillB, draw=none]  (0,1.1+\maxheight -| f.-40) to (0,-\maxheight -| f.-40) to (0,-\maxheight -| f.-140) to (0,1.1+\maxheight -| f.-140);
\draw [fill=\fillC, draw=none]  (f.40) to (0,\maxheight -| f.40) to (0,\maxheight -| f.140) to (f.140);
\draw (f.40) to (0,\maxheight -| f.40);
\draw (f.140) to (0,\maxheight -| f.140);
\draw (g.40) to (0,1.1+\maxheight -| g.40);
\draw (g.140) to (0,1.1+\maxheight -| g.140);
\draw (f.-40) to (0,-\maxheight -| f.-40);
\draw (f.-140) to (0,-\maxheight -| f.-140);
\end{tikzpicture}
\end{aligned}
\,\,\,=\,\,\,
\begin{aligned}
\begin{tikzpicture}[thick]
\draw [fill=\fillD, draw=none] (0,1.1+\maxheight -| f.40) to (1,1.1+\maxheight) to (1,-\maxheight) to (0,-\maxheight -| f.40);
\draw [fill=\fillA, draw=none] (0,1.1+\maxheight -| f.140) to (-1,1.1+\maxheight) to (-1,-\maxheight) to (0,-\maxheight -| f.140);
\draw [fill=\fillB, draw=none]  (0,1.1+\maxheight -| f.-40) to (0,-\maxheight -| f.-40) to (0,-\maxheight -| f.-140) to (0,1.1+\maxheight -| f.-140);
\draw (0,-\maxheight -| f.40) to (0,1.1+\maxheight -| f.40);
\draw (0,-\maxheight -| f.140) to (0,1.1+\maxheight -| f.140);
\end{tikzpicture}
\end{aligned}
\end{calign}
This sort of invertibility is a standard concept in 2\-category theory. Since the 2\-cells are connected vertically, we could also call this \emph{vertical invertibility}; the 2\-cells $f$ and $g$ are then \emph{vertical inverses}.

However, in a 2\-category we can also stack morphisms side-by-side. This suggests another notion of invertibility. We define a 2\-cell of the form~\eqref{eq:example2cell} to be \emph{horizontally invertible} if there exist adjunctions ${}^* \! F \dashv F$, ${}^* \! J \dashv J$, $G \dashv G^*$ and $H \dashv H^*$, and a 2\-cell of the form
\begin{equation}
\begin{aligned}
\begin{tikzpicture}[thick]
\begin{pgfonlayer}{foreground}
\node (f) [draw, minimum width=30pt, thick, fill=white] at (0,0) {$f'$};
\end{pgfonlayer}
\draw [fill=\fillA, draw=none] (0,\maxheight -| f.40) to (1,\maxheight) to (1,-\maxheight) to (0,-\maxheight -| f.40);
\draw [fill=\fillD, draw=none] (0,\maxheight -| f.140) to (-1,\maxheight) to (-1,-\maxheight) to (0,-\maxheight -| f.140);
\draw [fill=\fillB, draw=none]  (f.40) to (0,\maxheight -| f.40) to (0,\maxheight -| f.140) to (f.140);
\draw [fill=\fillC, draw=none]  (f.-40) to (0,-\maxheight -| f.-40) to (0,-\maxheight -| f.-140) to (f.-140);
\draw (f.40) to (0,\maxheight -| f.40) node [above] {${}^* \! F$};
\draw (f.140) to (0,\maxheight -| f.140) node [above] {$G^*$};
\draw (f.-40) to (0,-\maxheight -| f.-40) node [below] {$H^*$};
\draw (f.-140) to (0,-\maxheight -| f.-140) node [below] {${}^* \! J$};
\end{tikzpicture}
\end{aligned}
\end{equation}
such that the following equations hold:
\def\maxheight{1}
\begin{calign}
\begin{aligned}
\begin{tikzpicture}[thick]
\begin{pgfonlayer}{foreground}
\node (f1) [draw, minimum width=30pt, thick, fill=white] at (-2,0) {$f$};
\node (f2) [draw, minimum width=30pt, thick, fill=white] at (0,0) {$f'$};
\end{pgfonlayer}
\draw [fill=\fillA, draw=none] (0,\maxheight -| f2.40) to (1,\maxheight) to (1,-\maxheight) to (0,-\maxheight -| f2.40);
\draw [fill=\fillB, draw=none]  (f2.40) to (0,\maxheight -| f2.40) to (0,\maxheight -| f1.140) to (f1.140);
\draw [fill=\fillC, draw=none]  (f2.-40) to (0,-\maxheight -| f2.-40) to (0,-\maxheight -| f1.-140) to (f1.-140);
\draw [fill=\fillA, draw=none] (0,\maxheight -| f1.140) to (-3,\maxheight) to (-3,-\maxheight) to (0,-\maxheight -| f1.140);
\draw [fill=\fillD, draw=none] (f1.40)
    to [out=up, in=up] (f2.140)
    to (f2.-140)
    to [out=down, in=down] (f1.-40);
\draw (f2.-140) to [out=down, in=down] (f1.-40);
\draw (f2.140) to [out=up, in=up] (f1.40);
\draw (f2.40) to (0,\maxheight -| f2.40);
\draw (f2.-40) to (0,-\maxheight -| f2.-40);
\draw (f1.140) to (0,\maxheight -| f1.140);
\draw (f1.-140) to (0,-\maxheight -| f1.-140);
\end{tikzpicture}
\end{aligned}
\quad=\quad
\begin{aligned}
\begin{tikzpicture}[thick]
\draw [fill=\fillA, draw=none] (-3,\maxheight) to (1,\maxheight) to (1,-\maxheight) to (-3,-\maxheight);
\draw [fill=\fillC] (0,-\maxheight -| f2.-40)
    to [out=up, in=up, looseness=0.8] (0,-\maxheight -| f1.-140);
\draw [fill=\fillB] (0,\maxheight -| f2.-40)
    to [out=down, in=down, looseness=0.8] (0,\maxheight -| f1.-140);
\end{tikzpicture}
\end{aligned}
\\
\begin{aligned}
\begin{tikzpicture}[thick]
\begin{pgfonlayer}{foreground}
\node (f1) [draw, minimum width=30pt, thick, fill=white] at (-2,0) {$f'$};
\node (f2) [draw, minimum width=30pt, thick, fill=white] at (0,0) {$f$};
\end{pgfonlayer}
\draw [fill=\fillD, draw=none] (0,\maxheight -| f2.40) to (1,\maxheight) to (1,-\maxheight) to (0,-\maxheight -| f2.40);
\draw [fill=\fillB, draw=none]  (f2.40) to (0,\maxheight -| f2.40) to (0,\maxheight -| f1.140) to (f1.140);
\draw [fill=\fillC, draw=none]  (f2.-40) to (0,-\maxheight -| f2.-40) to (0,-\maxheight -| f1.-140) to (f1.-140);
\draw [fill=\fillD, draw=none] (0,\maxheight -| f1.140) to (-3,\maxheight) to (-3,-\maxheight) to (0,-\maxheight -| f1.140);
\draw [fill=\fillA, draw=none] (f1.40)
    to [out=up, in=up] (f2.140)
    to (f2.-140)
    to [out=down, in=down] (f1.-40);
\draw (f2.-140) to [out=down, in=down] (f1.-40);
\draw (f2.140) to [out=up, in=up] (f1.40);
\draw (f2.40) to (0,\maxheight -| f2.40);
\draw (f2.-40) to (0,-\maxheight -| f2.-40);
\draw (f1.140) to (0,\maxheight -| f1.140);
\draw (f1.-140) to (0,-\maxheight -| f1.-140);
\end{tikzpicture}
\end{aligned}
\quad=\quad
\begin{aligned}
\begin{tikzpicture}[thick]
\draw [fill=\fillD, draw=none] (-3,\maxheight) to (1,\maxheight) to (1,-\maxheight) to (-3,-\maxheight);
\draw [fill=\fillC] (0,-\maxheight -| f2.-40)
    to [out=up, in=up, looseness=0.8] (0,-\maxheight -| f1.-140);
\draw [fill=\fillB] (0,\maxheight -| f2.-40)
    to [out=down, in=down, looseness=0.8] (0,\maxheight -| f1.-140);
\end{tikzpicture}
\end{aligned}
\end{calign}

\noindent
In this case we say that $f$ and $f'$ are \textit{horizontal inverses}. Just as with ordinary inverses, the existence of a horizontal inverse to a 2\-cell of a given composite type is a property, not a structure. However, if a horizontal inverse exists, its identity depends up to isomorphism on the choice of duality 2\-cells used to witness the relevant adjunctions of 1\-cells. This concept is also interesting to study in 1\-object 2\-categories, better known as monoidal categories.

\begin{lemma}
\label{lem:horizontalfromvertical}
In a 2\-category, a 2\-cell
\begin{calign}
\begin{aligned}
\begin{tikzpicture}[thick]
\begin{pgfonlayer}{foreground}
\node (f) [draw, minimum width=30pt, thick, fill=white] at (0,0) {$f$};
\end{pgfonlayer}
\draw [fill=\fillD, draw=none] (0,\maxheight -| f.40) to (1,\maxheight) to (1,-\maxheight) to (0,-\maxheight -| f.40);
\draw [fill=\fillA, draw=none] (0,\maxheight -| f.140) to (-1,\maxheight) to (-1,-\maxheight) to (0,-\maxheight -| f.140);
\draw [fill=\fillB, draw=none]  (f.40) to (0,\maxheight -| f.40) to (0,\maxheight -| f.140) to (f.140);
\draw [fill=\fillC, draw=none]  (f.-40) to (0,-\maxheight -| f.-40) to (0,-\maxheight -| f.-140) to (f.-140);
\draw (f.40) to (0,\maxheight -| f.40) node [above] {$G$};
\draw (f.140) to (0,\maxheight -| f.140) node [above] {$F$};
\draw (f.-40) to (0,-\maxheight -| f.-40) node [below] {$J$};
\draw (f.-140) to (0,-\maxheight -| f.-140) node [below] {$H$};
\end{tikzpicture}
\end{aligned}
\end{calign}
is horizontally invertible iff one, and hence both, of the following composites are vertically invertible:
\def\maxheight{1.1}
\begin{calign}
\label{eq:bentexamples}
\begin{aligned}
\begin{tikzpicture}[thick]
\begin{pgfonlayer}{foreground}
\node (f) [draw, minimum width=30pt, thick, fill=white] at (0,0) {$f$};
\end{pgfonlayer}
\draw [fill=\fillB, draw=none]  (-1.5,\maxheight) rectangle (f.40 |- 0,-\maxheight);
\draw [fill=\fillC, draw=none]  (f.-140) rectangle (1.5,-\maxheight);
\draw [fill=\fillC, draw=none]  (0,\maxheight -| f.40) rectangle (1.5,-\maxheight);
\draw [fill=\fillD, draw=none] (0,\maxheight -| f.40)
    to (1,\maxheight)
    to (1,0 |- f.-40)
    to [out=down, in=down, looseness=2] (f.-40);
\draw [fill=\fillA, draw=none] (f.140)
    to [out=up, in=up, looseness=2] (-1,0 |- f.140)
    to (-1,-\maxheight)
    to (0,-\maxheight -| f.140);
\draw (f.40) to (0,\maxheight -| f.40) node [above] {$G$};
\draw (f.140)
    to [out=up, in=up, looseness=2] (-1,0 |- f.140)
    to (-1,-\maxheight)
        node [below] {$F^*$};
\draw (f.-40)
    to [out=down, in=down, looseness=2] (1,0 |- f.-40)
    to (1,\maxheight)
       node [above] {$J^*$};
\draw (f.-140) to (0,-\maxheight -| f.-140) node [below] {$H$};
\end{tikzpicture}
\end{aligned}
&
\begin{aligned}
\begin{tikzpicture}[thick, yscale=-1]
\begin{pgfonlayer}{foreground}
\node (f) [draw, minimum width=30pt, thick, fill=white] at (0,0) {$f$};
\end{pgfonlayer}
\draw [fill=\fillC, draw=none] (-1.5,\maxheight)
    rectangle (f.-40 |- 0,-\maxheight);
\draw [fill=\fillB, draw=none] (f.140) rectangle (1.5,-\maxheight);
\draw [fill=\fillB, draw=none] (0,\maxheight -| f.-40)
    rectangle (1.5,-\maxheight);
\draw [fill=\fillD, draw=none] (0,\maxheight -| f.-40)
    to (1,\maxheight)
    to (1,0 |- f.40)
    to [out=down, in=down, looseness=2] (f.40);
\draw [fill=\fillA, draw=none] (f.-140)
    to [out=up, in=up, looseness=2] (-1,0 |- f.-140)
    to (-1,-\maxheight)
    to (0,-\maxheight -| f.-140);
\draw (f.40) to (0,\maxheight -| f.-40) node [below] {$J$};
\draw (f.-140)
    to [out=up, in=up, looseness=2] (-1,0 |- f.-140)
    to (-1,-\maxheight)
        node [above] {$H^*$};
\draw (f.40)
    to [out=down, in=down, looseness=2] (1,0 |- f.40)
    to (1,\maxheight)
       node [below] {$G^*$};
\draw (f.140) to (0,-\maxheight -| f.140) node [above] {$F$};
\end{tikzpicture}
\end{aligned}
\end{calign}
\end{lemma}
\begin{proof}
Immediate by applying the string diagram calculus for 2\-categories.
\end{proof}

If our 2-category is equipped with a $\dag$-operation which is coherent with the rest of the structure, as described at the end of Section~\ref{sec:dagger}, we can extend this definition in a natural way, by defining a 2\-cell of the form~\eqref{eq:example2cell} to be \textit{horizontally unitary} if one of the composites in~\eqref{eq:bentexamples}, and hence both, are vertically unitary.

\subsection{Quantum teleportation}
\def\Bell{\ensuremath{\psi_\text{Bell}}}
\def\Create{\text{Create}}
\def\Copy{\text{Copy}}
\def\Wa{\ensuremath{\textrm{W}_\text{L}}}
\def\Wb{\ensuremath{\textrm{W}_\text{R}}}
\def\Q{\textrm{Q}}
\def\xs{\hspace*{3pt}}

\label{sec:conventionalteleportation}

\subsubsection*{Conventional teleportation}
In this section we  analyze quantum teleportation. First introduced in 1993~\cite{b93-teleport}, it is one of the most important quantum procedures. It is specified in our framework by the following graphical identity, as given in the introduction:
\begin{equation}
\label{eq:mainteleport}
\begin{aligned}
\begin{tikzpicture}
\node (V) [Vertex] at (1,0) {};
\node (W) [Vertex] at (2.25,1) {};
\draw [fill=\fillA, thick] (0,2)
    to [] (0,1)
    to [out=down, in=\nwangle] (1,0)
    to [out=\neangle, in=\swangle] (W.center)
    to [out=\nwangle, in=down] (1.5,2);
\draw [thick] (0.5,-1) to [out=up, in=down] (V.s1) to (V.s2) to [out=down, in=down] (3,0) to [out=up, in=\seangle] (W.center) to [out=\neangle, in=down] (3,2);
\end{tikzpicture}
\end{aligned}
\qquad=\hspace{15pt}
\frac{1}{\sqrt{n}}\,
\begin{aligned}
\begin{tikzpicture}
\draw [fill=\fillA, thick] (0,2)
    to (0,1.5)
    to [out=down, in=down, looseness=1.5] (1.5,1.5)
    to (1.5,2);
\draw [thick] (0.75,-1) to [out=up, in=down] (3,1.5) to (3,2);
\end{tikzpicture}
\end{aligned}
\end{equation}
On the left-hand side of the equation a system is introduced, an entangled state is prepared, a measurement is performed, and a controlled operation is  executed. The right-hand diagram is much simpler: the original system  exits without interacting, and the classical data is uniformly-generated and disconnected from the free quantum system. The statement that teleportation works is exactly the assertion that a solution to this equation can be found in the 2\-category~\cat{2Hilb}, using components which satisfy the axioms of our formalism.

We now demonstrate that ordinary qubit teleportation gives a solution in this sense. First, we rewrite our specification to put it in a more convenient form.
\begin{equation}
\label{eq:teleportexplicit}
\begin{aligned}
\begin{tikzpicture}[scale=1.2, thick]
    \begin{pgfonlayer}{foreground}
    \node (U) [minimum width=1.7cm, minimum height=18pt, fill=white, draw=black, thick] at (2,2.0) {$U _\text{Bell}$};
    \node (M) [minimum width=1.7cm, minimum height=18pt, fill=white, draw=black, thick, anchor=center] at (1.0,1.0) {$M _\text{Bell}$};
    \node (Bell) [minimum width=1.7cm, minimum height=18pt, fill=white, draw=black, anchor=center, thick] at (2,0.0) {$\Bell$};
    \end{pgfonlayer}
    \draw [fill=\fillA] (0.5,3)
        to (0.5,1.0)
        to (1.5,1.0)
        to (1.5,3);
    \draw (0.5,-1) to (0.5,1.5);
    \draw (1.5,1.6 |- M.south)
        to [out=down, in=up, looseness=1.5]
        (1.5,0.0);
    \draw (2.5,0.0) to (2.5,3.0);
    \node [fill=white, rounded corners=1pt, inner sep=1pt] at (1,2.1) {4};
\end{tikzpicture}
\end{aligned}
\quad=\quad
\frac{1}{2}
\,
\begin{aligned}
\begin{tikzpicture}[scale=1.2, thick]
    \begin{pgfonlayer}{foreground}
    \node (M) [minimum width=1.7cm, minimum height=18pt, fill=white, draw=black, thick, anchor=center] at (1.0,1.5) {$\Create(4)$};
    \end{pgfonlayer}
    \draw [fill=\fillA] (0.5,3.5)
        to (0.5,1.5)
        to (1.5,1.5)
        to (1.5,3.5);
    \draw (2.5,-0.5) to (2.5,3.5);
    \node [fill=white, rounded corners=1pt, inner sep=1pt] at (1,2.6) {4};
\end{tikzpicture}
\end{aligned}
\end{equation}
Since the measurement operation for qubit teleportation has 4 outcomes, we label the regions of classical data with the number 4, indicating that they correspond to the object $\cat{Hilb}^4$ in \cat{2Hilb}. We label creation of the Bell state as ``\Bell'', the measurement operation as ``$M _\text{Bell}$'', the controlled unitary as ``$U _\text{Bell}$'', and the creation of 4\-valued uniform classical data as ``\Create(4)''. The factor of $\frac{1}{2}$ on the right-hand side ensures that the process of creating uniform classical data is an isometry. 

To proceed, we must define a 2\-cell in \cat{2Hilb} for each of the four components $M_\text{Bell}$, $U_\text{Bell}$, and \Bell{}, the 2\-cell representing $\Create(4)$ having been already defined in Theorem~\ref{thm:2hilbwitness}.

The 2\-cell \Bell{} represents the creation of the Bell state $\smash{\frac{1}{\sqrt{2}}}(\ket {00} + \ket{11})$ as an element of $\C^2 \otimes \C^2$, the state space of two qubits. It has the following explicit form: 
\begin{align}
\Bell \,&=\, \matrix{ {\textstyle \frac{1}{\sqrt{2}}} \matrix{1 \\ 0 \\ 0 \\ 1}}
\intertext{\parbox{\linewidth}{
This is a 1-by-1 matrix, the single element of which is a linear map $\C \to \C ^2 \otimes \C ^2$.
\\
\hspace*{15pt}We then perform a measurement in the Bell basis, which comprises the following states:}}
\nonumber
\ket{\rBell_1} &= {\textstyle \frac{1}{\sqrt{2}}} \big( \ket{00} + \ket{11} \big)
\\
\nonumber
\ket{\rBell_2} &= {\textstyle \frac{1}{\sqrt{2}}} \big( \ket{00} - \ket{11} \big)
\\
\nonumber
\ket{\rBell_3} &= {\textstyle \frac{1}{\sqrt{2}}} \big( \ket{01} + \ket{10} \big)
\\
\nonumber
\ket{\rBell_4} &= {\textstyle \frac{1}{\sqrt{2}}} \big( \ket{01} - \ket{10} \big)
\intertext{The unitary measurement 2\-cell $M_\text{Bell}$ is constructed from this basis in the manner described by Theorem~\ref{thm:2hilbnondegenmeasurement}, with the following result:}
M_\text{Bell} \,&:=\,
\matrix{
\frac 1 {\sqrt{2}}
\matrix{\xs 1 \xs & 0 & 0 & 1
\\
\xs 1 \xs & 0 & 0 & \!\!\!\smash{-1}\!\!\!
\\
\xs 0 \xs & 1 & 1 & 0
\\
\xs 0 \xs & 1 & -1 & 0
}
}
\intertext{\parbox{\linewidth}{This is a 1-by-1 matrix with a single entry, given by the change of basis matrix that rotates the Bell basis into the computational basis.
\\
\hspace*{15pt}The final stage of the protocol involves performing a controlled unitary operation $U_\text{Bell}$ on the remaining qubit. This is represented by the following 2-cell, as according to Theorem~\ref{thm:controlled}:}}
U \,&:=\, \matrix{
\matrix{1&0\\0&\,\,1\,}
\\
\matrix{1&0\\0&-1}
\\
\matrix{0&1\\1&\,\,0\,}
\\
\matrix{0&1\\ {-1} &0}}
\end{align}
It is a column matrix of four 2-by-2 matrices. Each of these 2-by-2 matrices is a different unitary operation, the particular one to be applied depending on the value of the classical data produced by the measurement 2-cell $M_\text{Bell}$.

The \textit{Mathematica} package \texttt{TwoVect.m} can be used to straightforwardly verify that these 2\-cells provide a solution to equation~\eqref{eq:teleportexplicit}. After entering the definitions of the 1\-cells \Wa(4), \Wb(4) and \Q, and the 2\-cells \Bell, $M$, $U$ and \Create(4), we test equation~\eqref{eq:teleportexplicit} with the following code~\cite{v12-highernotebook}:
\begin{align}
( \Wa(4) \circ U _\text{Bell}) \cdot (M_\text{Bell} \circ \Q) \cdot (\Q\circ \Bell) == (\Create(4) \circ \Q)
\end{align}
Mathematica returns the result \texttt{true}, demonstrating that the equation is satisfied.

\subsubsection*{Axiomatization}

We now analyze the form of the quantum teleportation specification more closely. 

\begin{theorem}
\label{thm:teleportationdoublyunitary}
A vertically unitary controlled operation
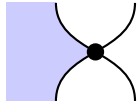
\[
\phantom{\frac{1}{\sqrt n}}
\begin{aligned}
\begin{tikzpicture}[scale=1.3, thick]
    \draw [white] (0.7,1) to (2.2,1);
    \draw [fill=\fillA, draw=none]
        (0.5,0.5)
        rectangle
        (1,1.5);
    \draw [fill=\fillA, draw=none] (1,1.5)
        to [out=down, in=\nwangle] (1.4,1)
        to [out=\swangle, in=up] (1,0.5);
    \draw
        (1,0.5)
        to [out=up, in=\swangle] (1.4,1)
        to [out=\nwangle, in=down] (1,1.5);
    \draw
        (1.8,0.5)
        to [out=up, in=\seangle] (1.4,1)
        to [out=\neangle, in=down] (1.8,1.5);
    \node [Vertex] at (1.4,1) {};
    \end{tikzpicture}
    \end{aligned}
\]
forms part of a solution to the teleportation specification if and only if it is also horizontally unitary, up to a scalar factor.
\end{theorem}
\begin{proof}
We begin with the following explicit version of the quantum teleportation specification, in which $U$ is the vertically unitary controlled operation of the hypothesis, and $M$ is a vertically unitary measurement operation.
\begin{equation}
\frac{1}{\sqrt{n}}\,
\begin{aligned}
\begin{tikzpicture}[scale=1.2, thick]
    \begin{pgfonlayer}{foreground}
    \node (U) [minimum width=1.7cm, minimum height=18pt, fill=white, draw=black, thick] at (2,2.0) {$U$};
    \node (M) [minimum width=1.7cm, minimum height=18pt, fill=white, draw=black, thick, anchor=center] at (1.0,1.0) {$M$};
    \end{pgfonlayer}
    \draw [fill=\fillA] (0.5,3)
        to (0.5,1.0)
        to (1.5,1.0)
        to (1.5,3);
    \draw (0.5,0) to (0.5,1.5);
    \draw (1.5,1.6 |- M.south)
        to [out=down, in=down, looseness=1.5]
        (2.5,1.6 |- M.south);
    \draw (2.5,1.0 |- M.south) to (2.5,3.0);
    \node [fill=white, rounded corners=1pt, inner sep=1pt] at (1,2.1) {$n$};
\end{tikzpicture}
\end{aligned}
\quad=\quad
\frac{1}{n}\,
\begin{aligned}
\begin{tikzpicture}[scale=1.2, thick]
    \draw [fill=\fillA] (0.5,3.5)
        to (0.5,2)
        to [out=down, in=down, looseness=2] (1.5,2)
        to (1.5,3.5);
    \draw (2.5,0.5) to (2.5,3.5);
    \node [fill=white, rounded corners=1pt, inner sep=1pt] at (1,2.6) {$n$};
\end{tikzpicture}
\end{aligned}
\end{equation}
Suppose we have an $M$ and a $U$ satisfying this equation. Applying the adjoint of the 2\-cell $U$ and simplifying the normalization factors, this is equivalent to the following expression:
\begin{equation}
\begin{aligned}
\begin{tikzpicture}[scale=1.2, thick]
    \begin{pgfonlayer}{foreground}
    \node (M) [minimum width=1.7cm, minimum height=18pt, fill=white, draw=black, thick, anchor=center] at (1.0,1.0) {$M$};
    \end{pgfonlayer}
    \draw [fill=\fillA] (0.5,2.0)
        to (0.5,1.0)
        to (1.5,1.0)
        to (1.5,2.0);
    \draw (0.5,0) to (0.5,1.5);
    \draw (1.5,1.6 |- M.south)
        to [out=down, in=down, looseness=1.5]
        (2.5,1.6 |- M.south);
    \draw (2.5,1.0 |- M.south) to (2.5,2.0);
    \node [fill=white, rounded corners=1pt, inner sep=1pt] at (1,1.6) {$n$};
\end{tikzpicture}
\end{aligned}
\quad=\quad
\frac{1}{\sqrt{n}}\,
\begin{aligned}
\begin{tikzpicture}[scale=1.2, thick]
    \begin{pgfonlayer}{foreground}
    \node (U) [minimum width=1.7cm, minimum height=18pt, fill=white, draw=black, thick] at (2,2.5) {$U ^\dag$};
    \end{pgfonlayer}
    \draw [fill=\fillA] (0.5,3.5)
        to (0.5,2 |- U.south)
        to [out=down, in=down, looseness=1.5] (1.5,2 |- U.south)
        to (1.5,3.5);
    \draw (2.5,1.5) to (2.5,3.5);
    \node [fill=white, rounded corners=1pt, inner sep=1pt] at (1,2.6) {$n$};
\end{tikzpicture}
\end{aligned}
\end{equation}
Making use of the topological properties of the Bell state for $n$-dimensional quantum systems, this can be rewritten in the following way:
\begin{equation}
\begin{aligned}
\begin{tikzpicture}[scale=1.2, thick]
    \begin{pgfonlayer}{foreground}
    \node (M) [minimum width=1.7cm, minimum height=18pt, fill=white, draw=black, thick, anchor=center] at (1.0,1.0) {$M$};
    \end{pgfonlayer}
    \draw [fill=\fillA] (0.5,2.0)
        to (0.5,1.0)
        to (1.5,1.0)
        to (1.5,2.0);
    \draw (0.5,0) to (0.5,1.5);
    \draw (1.5,0) to (1.5,1.5);
    \node [fill=white, rounded corners=1pt, inner sep=1pt] at (1,1.6) {$n$};
\end{tikzpicture}
\end{aligned}
\quad=\quad
\frac{1}{\sqrt{n}}\,
\begin{aligned}
\begin{tikzpicture}[scale=1.2, thick]
    \begin{pgfonlayer}{foreground}
    \node (U) [minimum width=1.7cm, minimum height=18pt, fill=white, draw=black, thick] at (2,2.5) {$U ^\dag$};
    \end{pgfonlayer}
    \draw [fill=\fillA] (0.5,3.5)
        to (0.5,2 |- U.south)
        to [out=down, in=down, looseness=1.5] (1.5,2 |- U.south)
        to (1.5,3.5);
    \node [fill=white, rounded corners=1pt, inner sep=1pt] at (1,2.6) {$n$};
    \draw (2.5,1.5) to (2.5,1.5 |- U.north) to [out=up, in=up, looseness=1.5] (3.5,1.5 |- U.north) to (3.5,1.5);
\end{tikzpicture}
\end{aligned}
\end{equation}
By the definition of horizontally unitary map, this gives us our result, since $M$ is vertically unitary. The converse direction is the reversal of this argument.
\end{proof}

\noindent
We see that our controlled measurement operation must be unitary both horizontally \emph{and} vertically, up to a scalar factor. We call this property \emph{double unitarity}.

\subsubsection*{More general teleportation procedures}

We now focus on the following question: given a basis of a bipartite Hilbert space $H \otimes H$, where $H$ is finite-dimensional, can a measurement in this basis form part of a teleportation protocol? This is a different perspective on generalized teleportation to that often taken in the literature, where attention is more usually devoted to the computational power of the shared resource than to the measurement basis.

In general, where the shared resource is not of Bell type, horizontal unitarity is not required for a measurement vertex to form part of a successful teleportation protocol. However, the following theorem shows that the weaker property of horizontal \textit{invertibility} is required. We are left with an interesting open question: does every horizontally-invertible bipartite measurement vertex form part of a deterministic quantum teleportation protocol for some shared resource?

\begin{theorem}
If a bipartite measurement vertex is not horizontally invertible, then it cannot form part of the solution to the quantum teleportation specification, even with general shared resource.
\end{theorem}

\begin{proof}
We consider a quantum teleportation protocol with general shared resource, $\psi$, not necessarily a Bell state. This gives us the following specification:
\begin{equation}
\begin{aligned}
\begin{tikzpicture}[scale=1.2, thick]
    \node (v) [Vertex] at (1,1.5) {};
    \begin{pgfonlayer}{foreground}
    \node (U)
        [minimum width=1.7cm, minimum height=18pt,
        fill=white, draw=black, thick,
        anchor=center]
        at (2,2.2) {$U$};
    \end{pgfonlayer}
    \draw [fill=\fillA] (0.5,3)
        to (0.5,2 |- U.south)
        to [out=down, in=\nwangle] (v.center)
        to [out=\neangle, in=down] (1.5,2 |- U.south)
        to (1.5,3);
    \draw [thick] (v.s1) to (v.s1 |- 0.5,-0.1);
    \draw [thick] (v.s2)
        to [out=down, in=up]
            (1.5,0.6);
    \draw [thick] (2.5,2) to (2.5,3);
    \draw [thick] (2.5,0.6)
        to (2.5,0 |- U.south);
    \draw [thick] (1.3,0.6)
        -- (2.7,0.6)
        -- (2,0.1)
        -- cycle;
    \node [anchor=center] at (2,0.4) {$\psi$};
\end{tikzpicture}
\end{aligned}
\quad
=
\quad
\begin{aligned}
\begin{tikzpicture}[scale=1.2]
    \draw [fill=\fillA] (0.3,3)
        to (0.3,2)
        to [out=down, in=down, looseness=2] (1.5,2)
        to (1.5,3);
    \draw [thick] (0.3,3)
        to (0.3,2)
        to [out=down, in=down, looseness=2] (1.5,2)
        to (1.5,3);
    \draw [thick] (v.s1 |- 0.5,-0.1)
        to [out=up, in=down, out looseness=0.9, in looseness=1]
            (2.4,2.1) to (2.4,3);
\end{tikzpicture}
\end{aligned}
\end{equation}
Making use of the topological behaviour of classical information, and the self-duality of finite-dimensional Hilbert spaces, we obtain the following equivalent version of this equation:
\begin{align}
\begin{aligned}
\begin{tikzpicture}[scale=1.2]
    \node (v) [Vertex] at (0.9,1.3) {};
    \draw [fill=\fillA, draw=none] (-0.1,0.5)
        to (-0.1,1)
        to [out=up, in=left] (0.3,1.7)
        to [out=right, in=\nwangle] (v.nw)
        to (v.ne)
        to [out=\neangle, in=down] (1.5,2)
        to (1.5,3.7)
        to ([xshift=-0.8cm] current bounding box.north west)
        to (current bounding box.south west);
    \draw [thick] (-0.1,0.5)
        to (-0.1,1)
        to [out=up, in=left] (0.3,1.7)
        to [out=right, in=\nwangle] (v.nw)
        to (v.ne)
        to [out=\neangle, in=down] (1.5,2)
        to (1.5,3.7);
    \draw [thick] (v.s1) to (v.s1 |- 0.5,0.5);
    \node (f) [minimum width=1.5cm, minimum height=18pt, fill, draw=black, thick, fill=white]
        at (1.95,3.1) {$U$};
    \begin{scope}[xshift=20pt,yshift=40pt]
        \draw [thick] (1.3,0.8) -- (2.5,0.8) -- (1.9,0.3) -- cycle;
        \node [anchor=center] at (1.9,0.60) {$\psi$};
        \node (P1) at (1.7,0.8) {};
        \node (P2) at (2.1,0.8) {};
    \end{scope}
    \draw [thick] (v.s2)
        to [out=down, in=down] (1.8,1.3)
        to [out=up, in=down]
            (1.8,2.2)
        to [out=up, in=up, looseness=1.5] (P1.center);
    \draw [thick] (P2.center)
        to [out=up, in=down]
        (f.-32);
    \draw [thick] (f.32) to (f.32 |- 0.5,3.7);
    \draw [dotted] (1,1.6) rectangle (3.5,3.5);
\end{tikzpicture}
\end{aligned}
\quad&=\quad
\begin{aligned}
\begin{tikzpicture}[scale=1.2]
    \draw [fill=\fillA, draw=none] (1,0.5)
        to [out=up, in=down] (1,3.7)
        to ([xshift=-1.5cm] current bounding box.north west)
        to (current bounding box.south west);
    \draw [thick] (1,0.5)
        to [out=up, in=down] (1,3.7);
    \draw [thick] (2,0.5)
        to [out=up, in=down] (2,3.7);
\end{tikzpicture}
\end{aligned}
\intertext{All unknown quantities have been moved into the same part of the diagram, and surrounded by a dotted box to emphasize that they can be treated as a single composite 2\-cell. However, it is clear that the contents of the box simply constitute a new controlled operation, albeit not necessarily unitary. Writing this composite as $C$, we can simplify our equation further:}
\begin{aligned}
\begin{tikzpicture}[scale=1.2]
    \node (v) [Vertex] at (0.9,1.3) {};
    \draw [fill=\fillA, draw=none] (-0.1,0.5)
        to (-0.1,1)
        to [out=up, in=left] (0.3,1.7)
        to [out=right, in=\nwangle] (v.nw)
        to (v.ne)
        to [out=\neangle, in=down] (1.5,2)
        to (1.5,3.2)
        to ([xshift=-0.8cm] current bounding box.north west)
        to (current bounding box.south west);
    \draw [thick] (-0.1,0.5)
        to (-0.1,1)
        to [out=up, in=left] (0.3,1.7)
        to [out=right, in=\nwangle] (v.nw)
        to (v.ne)
        to [out=\neangle, in=down] (1.5,2)
        to (1.5,3.2);
    \draw [thick] (v.s1) to (v.s1 |- 0.5,0.5);
    \node (f) [minimum width=1.5cm, minimum height=18pt, fill, draw=black, thick, fill=white]
        at (1.95,2.5) {$C$};
    \draw [thick] (v.s2)
        to [out=down, in=down] (f.-32 |- 0.5,1.3)
        to (f.-32);
    \draw [thick] (f.32) to (f.32 |- 0.5,3.2);
\end{tikzpicture}
\end{aligned}
\quad&=\quad
\begin{aligned}
\begin{tikzpicture}[scale=1.2]
    \draw [fill=\fillA, draw=none] (1,0.5)
        to [out=up, in=down] (1,3.2)
        to ([xshift=-1.5cm] current bounding box.north west)
        to (current bounding box.south west);
    \draw [thick] (1,0.5)
        to [out=up, in=down] (1,3.2);
    \draw [thick] (2,0.5)
        to [out=up, in=down] (2,3.2);
\end{tikzpicture}
\end{aligned}
\intertext{We see that $C$ is an inverse on one side to our `twisted' measurement vertex. But since $C$ is an endomorphism, it must therefore be a two-sided inverse, by Lemma~\ref{lem:fg1gf1}. This gives us the following equation:}
\begin{aligned}
\begin{tikzpicture}[scale=1.2]
    \node [minimum width=1.5cm, minimum height=18pt, fill, draw=black, thick, fill=white, white] at (1.95,2) {};
    \node (v) [Vertex] at (0.9,1.3) {};
    \draw [fill=\fillA, draw=none] (-0.1,-0.2)
        to (-0.1,1)
        to [out=up, in=left] (0.3,1.7)
        to [out=right, in=\nwangle] (v.nw)
        to (v.ne)
        to [out=\neangle, in=down] (1.5,2)
        to (1.5,2.5)
        to ([xshift=-0.8cm] current bounding box.north west)
        to (current bounding box.south west);
    \draw [thick] (-0.1,-0.2)
        to (-0.1,1)
        to [out=up, in=left] (0.3,1.7)
        to [out=right, in=\nwangle] (v.nw)
        to (v.ne)
        to [out=\neangle, in=down] (1.5,2)
        to (1.5,2.5);
    \draw [thick] (v.s1) to (v.s1 |- 0.5,-0.2);
    \node (f) [minimum width=1.5cm, minimum height=18pt, fill, draw=black, thick, fill=white] at (0.4,0.5) {$C$};
    \draw [thick] (v.s2)
        to [out=down, in=down] (2.35,1.5 |- v.s2)
        to (2.35,2.5);
\end{tikzpicture}
\end{aligned}
\quad&=\quad
\begin{aligned}
\begin{tikzpicture}[scale=1.2]
    \draw [fill=\fillA, draw=none] (1,0.5)
        to [out=up, in=down] (1,3.2)
        to ([xshift=-1.5cm] current bounding box.north west)
        to (current bounding box.south west);
    \draw [thick] (1,0.5)
        to [out=up, in=down] (1,3.2);
    \draw [thick] (2,0.5)
        to [out=up, in=down] (2,3.2);
\end{tikzpicture}
\end{aligned}
\end{align}
But the existence of such an invertible 2\-cell $C$ is precisely the condition for the original measurement vertex to be horizontally invertible.
\end{proof}

\subsection{Dense coding}

\subsubsection*{Specification}

We can use our formalism to write down the following specification for the dense coding protocol:
\begin{equation}
\label{eq:densecoding}
\frac{1}{\sqrt{n}}
\,
\begin{aligned}
\begin{tikzpicture}[scale=1, thick]
\draw [use as bounding box, draw=none] (-1.5,0.5) rectangle (2.5,3.5);
\node (m) [Vertex] at (1.75,2.5) {};
\node (u) [Vertex] at (0.5,1.5) {};
\draw [draw=none, fill=\fillA] (-0.25,0.5) to (-0.25,3.5) to (-1.5,3.5) to (-1.5,0.5);
\draw [fill=\fillA] (-0.25,0.5)
    to [out=up, in=\swangle] (u.center)
    to (u.center)
    to [out=\nwangle, in=down] (-0.25,3)
    to (-0.25,3.5);
\draw (m.center)
    to [out=\swangle, in=\neangle] (u.center)
    to [out=\seangle, in=120] (1,1)
    to [out=-60, in=down, looseness=1.5] (2.5,1.5)
    to [out=up, in=\seangle] (m.center);
\draw [fill=\fillA] (1,3.5)
    to [out=down, in=\nwangle] (m.center)
    to [out=\neangle, in=down] (2.5,3.5);
\end{tikzpicture}
\end{aligned}
\qquad=\qquad
\begin{aligned}
\begin{tikzpicture}[scale=1, thick]
\draw [draw=none, fill=\fillA]
    (0,3)
    to [out=down, in=down, looseness=2]
    (1.25,3)
    to (2.5,3)
    to (2.5,2.5)
    to [out=down, in=up] (0,0.0)
    to (-1.25,0)
    to (-1.25,3)
    to (0,3);
\draw
    (2.5,3)
    to (2.5,2.5)
    to [out=down, in=up] (0,0.0)
    to (0,0);
\draw
    (0,3)
    to [out=down, in=down, looseness=2]
    (1.25,3);
\end{tikzpicture}
\end{aligned}
\end{equation}
On the left-hand side, we begin with a witness for classical data. An entangled state is then prepared, and a controlled unitary operation is performed on the left-hand part of of the entangled state. Finally, a bipartite measurement is carried out on the two systems. As a result, classical information is copied.

\subsubsection*{Implementation}

As with teleportation, we now verify that the conventional procedure for dense coding with qubits gives an implementation in our sense, as providing a solution to the equation~\eqref{eq:densecoding}.

As before, we begin by rewriting our specification in a manner which makes each component clearly identifiable.
\begin{equation}
\label{eq:explicitdensecoding}
\begin{aligned}
\begin{tikzpicture}[scale=1.2, thick]
    \begin{pgfonlayer}{foreground}
    \node (U) [minimum width=1.7cm, minimum height=18pt, fill=white, draw=black, thick] at (2,2) {$M _\text{Bell}$};
    \node (M) [minimum width=1.7cm, minimum height=18pt, fill=white, draw=black, thick, anchor=center] at (1.0,1.0) {$U _\text{Bell}$};
    \node (Bell) [minimum width=1.7cm, minimum height=18pt, fill=white, draw=black, thick, anchor=center] at (2,0.0) {$\Bell$};
    \end{pgfonlayer}
    \draw [draw=none, fill=\fillA] (0.5,-1) rectangle (-0.5,3);
    \draw [fill=\fillA, draw=none, thick] (1.5,3)
        to (1.5,2)
        to (2.5,2)
        to (2.5,3);
    \draw (0.5,-1) to (0.5,3.0);
    \draw (2.5,0.0) to (2.5,3);
    \draw (1.5,0.0) to (1.5,3);
    \node [fill=white, rounded corners=1pt, inner sep=1pt] at (-0.1,1) {4};
    \node [fill=white, rounded corners=1pt, inner sep=1pt] at (2,2.65) {4};
\end{tikzpicture}
\end{aligned}
\quad=\quad
\begin{aligned}
\begin{tikzpicture}[scale=1.2, thick]
    \begin{pgfonlayer}{foreground}
    \node (M) [minimum width=1.7cm, minimum height=18pt, fill=white, draw=black, thick, anchor=center] at (1,1.5) {$\Copy(4)$};
    \end{pgfonlayer}
    \draw [fill=\fillA, draw=none] (-0.5,3.5)
        rectangle (2.5,-0.5);
    \draw [draw=none, fill=white] (0.5,1.5) rectangle (1.5,3.5);
    \draw (2.5,-0.5) to (2.5,3.5);
    \draw (0.5,1.5) to (0.5,3.5);
    \draw (1.5,1.5) to (1.5,3.5);
    \node [fill=white, rounded corners=1pt, inner sep=1pt] at (1,0.4) {4};
\end{tikzpicture}
\end{aligned}
\end{equation}

\noindent
The normalization factor of~\eqref{eq:densecoding} is no longer present, as we absorb it into the preparation of the Bell state $\psi_\text{Bell}$. The definitions of the 2\-cells \Bell, $U_\text{Bell}$ and $M_\text{Bell}$ are the same as in the case of quantum teleportation. The \Copy(4) 2\-cell is defined according to Theorem~\ref{thm:2hilbwitness}.

To verify equation~\eqref{eq:explicitdensecoding} we evaluate the following expression in Mathematica, as available in the notebook accompanying this article~\cite{v12-highernotebook}:
\begin{equation}
(\Wa(4) \circ M_\text{Bell}) \cdot (U_\text{Bell} \circ \Q)\cdot (\Wa(4)\circ \Bell)==(\Copy(4) \circ \Wa(4))
\end{equation}
The result is \texttt{true}, meaning we have a correct implementation of dense coding.

\subsubsection*{Axiomatization}

Implementations of the dense coding specification again correspond exactly to doubly-unitary maps of a particular form, as was the case with teleportation. As a result, we have an exact correspondence between implementations of quantum teleportation and implementations of dense coding, as has been described by other more traditional techniques in the literature~\cite{w01-tdc}. As discussed in Section~\ref{sec:mathematicalaspects}, this correspondence is most elegantly described in terms of the \emph{theory} of a specification.
\begin{theorem}
\label{thm:densecodingdoublyunitary}
A vertically unitary controlled operation
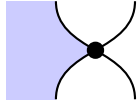
\[
\phantom{\frac{1}{\sqrt n}}
\begin{aligned}
\begin{tikzpicture}[scale=1.3, thick]
    \draw [white] (0.7,1) to (2.2,1);
    \draw [fill=\fillA, draw=none]
        (0.5,0.5)
        rectangle
        (1,1.5);
    \draw [fill=\fillA, draw=none] (1,1.5)
        to [out=down, in=\nwangle] (1.4,1)
        to [out=\swangle, in=up] (1,0.5);
    \draw
        (1,0.5)
        to [out=up, in=\swangle] (1.4,1)
        to [out=\nwangle, in=down] (1,1.5);
    \draw
        (1.8,0.5)
        to [out=up, in=\seangle] (1.4,1)
        to [out=\neangle, in=down] (1.8,1.5);
    \node [Vertex] at (1.4,1) {};
    \end{tikzpicture}
    \end{aligned}
\]
forms part of a solution to the dense coding specification if and only if it is also horizontally unitary, up to a scalar factor.
\end{theorem}
\begin{proof}
Our proof has the same form as the proof of Theorem~\ref{thm:teleportationdoublyunitary}, although it differs in its details since teleportation and dense coding have different specifications.

We begin with the following explicit version of the dense coding specification, in which $U$ is the vertically unitary controlled operation of the hypothesis, and $M$ is a vertically unitary measurement operation.
\begin{equation}
\frac{1}{\sqrt{n}}\,
\begin{aligned}
\begin{tikzpicture}[scale=1.2, thick]
    \begin{pgfonlayer}{foreground}
    \node (M) [minimum width=1.7cm, minimum height=18pt, fill=white, draw=black, thick] at (2,2) {$M$};
    \node (U) [minimum width=1.7cm, minimum height=18pt, fill=white, draw=black, thick, anchor=center] at (1.0,1.0) {$U$};
    \end{pgfonlayer}
    \draw [draw=none, fill=\fillA] (0.5,0) rectangle (-0.5,3);
    \draw [fill=\fillA, draw=none, thick] (1.5,3)
        to (1.5,2)
        to (2.5,2)
        to (2.5,3);
    \draw (0.5,0) to (0.5,3.0);
    \draw (1.5,3) to (1.5,0 |- U.south)
        to [out=down, in=down, looseness=1.5] (2.5,0 |- U.south)
        to (2.5,3);
    \node [fill=white, rounded corners=1pt, inner sep=1pt] at (0,2.6) {$n$};
    \node [fill=white, rounded corners=1pt, inner sep=1pt] at (2,2.6) {$n$};
\end{tikzpicture}
\end{aligned}
\quad=\quad
\begin{aligned}
\begin{tikzpicture}[scale=1.2, thick]
    \draw [fill=\fillA, draw=none] (-0.5,3)
        rectangle (2.5,0);
    \draw (2.5,0) to (2.5,3);
    \draw [fill=white] (0.5,3) to (0.5,1.7) to [out=down, in=down, looseness=1.5] (1.5,1.7) to (1.5,3);
    \node [fill=white, rounded corners=1pt, inner sep=1pt] at (1,0.7) {$n$};
\end{tikzpicture}
\end{aligned}
\end{equation}
Suppose we have an $M$ and a $U$ satisfying this equation. Applying the adjoint of the 2\-cell $M$ and using the topological behaviour of classical information, we obtain the following equivalent expression:
\begin{equation}
\frac{1}{\sqrt{n}}\,
\begin{aligned}
\begin{tikzpicture}[scale=1.2, yscale=-1, thick]
    \begin{pgfonlayer}{foreground}
    \node (U) [minimum width=1.7cm, minimum height=18pt, fill=white, draw=black, thick] at (2,2.5) {$U$};
    \end{pgfonlayer}
    \draw [fill=\fillA] (0.5,3.5)
        to (0.5,2 |- U.north)
        to [out=down, in=down, looseness=1.5] (1.5,2 |- U.north)
        to (1.5,3.5);
    \node [fill=white, rounded corners=1pt, inner sep=1pt] at (1,3.1) {$n$};
    \draw (2.5,1.5) to (2.5,1.5 |- U.south) to [out=up, in=up, looseness=1.5] (3.5,1.5 |- U.south) to (3.5,1.5);
\end{tikzpicture}
\end{aligned}
\quad=\quad
\begin{aligned}
\begin{tikzpicture}[scale=1.2, yscale=-1, thick]
    \begin{pgfonlayer}{foreground}
    \node (M) [minimum width=1.7cm, minimum height=18pt, fill=white, draw=black, thick, anchor=center] at (1.0,1.0) {$M ^\dag$};
    \end{pgfonlayer}
    \draw [fill=\fillA] (0.5,2.0)
        to (0.5,1.0)
        to (1.5,1.0)
        to (1.5,2.0);
    \draw (0.5,0) to (0.5,1.5);
    \draw (1.5,0) to (1.5,1.5);
    \node [fill=white, rounded corners=1pt, inner sep=1pt] at (1,1.6) {$n$};
\end{tikzpicture}
\end{aligned}
\end{equation}
By the definition of horizontally unitary map, this gives us our result, since $M$ is vertically unitary. The converse direction is the reversal of this argument.
\end{proof}

\ignore{
\subsubsection*{Impossibility of W-state teleportation}

To apply this theory, we consider measurement in a basis of W states. An orthonormal basis of the 3\-qubit state space was presented in \cite{mb05-dme}, such that each element of the basis is locally unitarily equivalent to the standard W state (\ket{\W_1} below). The basis consists of the following vectors:
\def\prefactor{\textstyle \frac{1}{\sqrt{3}}}
\newcommand\wstate[1]{\textstyle \frac{1}{\sqrt{3}} \big( \hspace{-3pt} #1 \big)}
\begin{align*}
\ket{\W_1} &= \wstate{+\ket{001} + \ket{010} + \ket{100}}
\\
\ket{\W_2} &= \wstate{+\ket{000} + \ket{011} - \ket{101}}
\\
\ket{\W_3} &= \wstate{-\ket{011} + \ket{000} + \ket{110}}
\\
\ket{\W_4} &= \wstate{-\ket{010} + \ket{001} - \ket{111}}
\\
\ket{\W_5} &= \wstate{+\ket{101} - \ket{110} + \ket{000}}
\\
\ket{\W_6} &= \wstate{+\ket{100} - \ket{111} - \ket{001}}
\\
\ket{\W_7} &= \wstate{-\ket{111} - \ket{100} + \ket{010}}
\\
\ket{\W_8} &= \wstate{-\ket{110} - \ket{101} - \ket{011}}
\end{align*}
Suppose we measure 3 qubits in this basis. Using the formalism introduced above, we can ask whether it is possible to then perform some controlled operation which results in the successful teleportation of  one of the originally-measured qubits. By expression~\eqref{eq:teleport-retractible}, this is possible if and only if the following 2\-cell is retractible:
\begin{equation}
\label{eq:wteleport}
\begin{aligned}
\begin{tikzpicture}[scale=1.3]
    \node (v) [Vertex] at (0.9,1.3) {};
    \draw [fill=\myfill, draw=none] (-0.1,0.5)
        to (-0.1,1)
        to [out=up, in=left] (0.3,1.7)
        to [out=right, in=\nwangle] (v.nw)
        to (v.ne)
        to [out=\neangle, in=down] (1.5,2)
        to (1.5,2.7)
        to ([xshift=-0.8cm] current bounding box.north west)
        to (current bounding box.south west);
    \draw [triple] (-0.1,0.5)
        to (-0.1,1)
        to [out=up, in=left] (0.3,1.7)
        to [out=right, in=\nwangle] (v.nw)
        to (v.ne)
        to [out=\neangle, in=down] (1.5,2)
        to (1.5,2.7);
    \draw [thick] (v.sA) to (v.sA |- 0.5,0.5);
    \node [fill=white, inner sep=1pt, rounded corners=2pt]
        at (0.3,2.2) {W};
    \node at (0.9,0.5) [below] {$Q\vphantom{{}^{\otimes 3}}$};
    \node at (-0.1,0.5) [below] {$Q^{\otimes 3}$};
    \node at (1.5,2.7) [above] {$Q^{\otimes 3}$};
    \node at (2.1,2.7) [above] {$Q$};
    \node at (2.7,2.7) [above] {$Q$};
    \draw [thick] (v.sB)
        to [out=down, in=down, looseness=1] (2.7,1.3)
        to (2.7,2.7);
    \draw [thick] (v.sC)
        to [out=down, in=down] (2.1,1.3)
        to (2.1,2.7);
\end{tikzpicture}
\end{aligned}
\end{equation}
Performing the appropriate calculation, it turns out that this composite is \emph{not} retractable. As a result, measurement in a basis of W states cannot form part of a quantum teleportation protocol. Thanks to our analysis above, we know that this holds regardless of the entangled resource shared by the two parties. This result seems to be new, and sheds light on the computational power of higher entangled states. One might ask what extra computational power \textit{at all} is given by a measurement in a basis of W~states, as compared to a measurement in a separable basis. Although some work has been done on the computational power of W~states~\cite{dp06-cpw}, a good answer to this question does not suggest itself.

In contrast, when the basis consists of maximally-entangled GHZ states, the analogous composite to~\eqref{eq:wteleport} \textit{does} have a retraction. Teleportation built on a GHZ-basis measurement is therefore physically achievable.
}

\subsection{Complementarity}

\subsubsection*{Introduction}

Suppose we have a pair of orthonormal bases $a_i, b_i$ for the same Hilbert space $H$. We say they are \emph{complementary} if
\begin{equation}
|\langle a_i | b_j \rangle| ^2 = \frac{1}{\dim H}
\end{equation}
for all $i,j$ indexing the basis elements. Such pairs of bases have the attractive property that complete knowledge of a system's state in one basis implies complete ignorance of its state in the other basis. A categorical axiomatization of the complementarity condition has been given by Coecke and Duncan~\cite{cd07-gcnc}. We will use our formalism to describe two completely new axiomatizations, one physical and one mathematical, and then demonstrate equivalence between our axioms and that of Coecke and Duncan.

\subsubsection*{Axiomatization}

Complementarity has a clear physical meaning: if we measure our system with respect to the basis $a_i$, and then again with respect to the basis $b_i$, the results are uncorrelated. Using our graphical calculus, drawing one observable in red and the other in green, we can represent two sequential measurements of this kind in the following way:
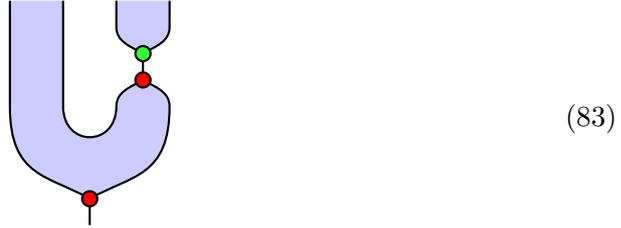
\def\compscale{0.7}
\begin{equation}
\begin{aligned}
\begin{tikzpicture}[scale=\compscale, thick]
\node (a) [Vertex, vertex colour=red] at (0.5,0.25  ) {};
\node (b) [Vertex, vertex colour=red] at (1.5, 2.5) {};
\node (c) [Vertex, vertex colour=green] at (1.5, 3.00) {};
\draw (0.5,-0.25 -| a.center) to (a.center);
\draw [fill=\fillcomp] (-1,4)
    to (-1,2)
    to [out=down, in=\nwangle, out looseness=1.5] (a.center)
    to [out=\neangle, in=down, in looseness=1.5] (2,2)
    to [out=up, in=\seangle] (b.center)
    to [out=\swangle, in=up] +(-0.5,-0.5)
    to [out=down, in=down, looseness=2] +(-1,0)
    to (0.0,4);
\draw (b.center)
    to [out=up, in=down] (c.center);
\draw [fill=\fillcomp] (1,4) to (1,3.5)
    to [out=down, in=\nwangle] (c.center)
    to [out=\neangle, in=down] +(0.5,0.5)
    to (2,4);
\end{tikzpicture}
\end{aligned}
\end{equation}
In order, this composite corresponds to performing a nondegenerate measurement in the red basis, copying the measurement result, encoding the second copy as the state of a new quantum system using the red basis, and performing a nondegenerate measurement in the green basis on this new system. To say that the results are uncorrelated is to say that, up to the application of a global phase which could depend on each measurement outcome, the resulting total quantum state factorizes. We can think of such a choice of phases as a `controlled phase', dependent on the two measurement outcomes. Such a controlled phase has the following representation in our formalism:
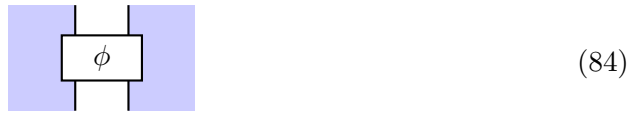
\begin{equation}
\begin{aligned}
\begin{tikzpicture}[scale=\compscale, thick]
\path [fill=\fillcomp] (-0.25,-0.5) to (-0.25,1.5) to (1,1.5) to (1,-0.5);
\path [fill=\fillcomp] (2,-0.5) to (2,1.5) to (3.25,1.5) to (3.25,-0.5);
\draw (1,-0.5) to (1,1.5);
\draw (2,-0.5) to (2,1.5);
\node [minimum width=30pt, draw, fill=white] at (1.5,0.5) {$\phi$};
\end{tikzpicture}
\end{aligned}
\end{equation}
This is an endomorphism of the empty diagram, corresponding to the Hilbert space $\C$, controlled by the values of both types of classical data.

We combine these ideas to obtain an equational realization of our mathematical condition. We define a pair of nondegenerate measurements indexed by classical information of dimension~$n$ to be \emph{physically complementary} if there exists a controlled phase $\phi$ satisfying the following equation:
\begin{equation}
\label{eq:physicallycomplementary}
\begin{aligned}
\begin{tikzpicture}[scale=\compscale, thick]
\node (a) [Vertex, vertex colour=red] at (0.5,0.25  ) {};
\node (b) [Vertex, vertex colour=red] at (1.5, 2.5) {};
\node (c) [Vertex, vertex colour=green] at (1.5, 3.00) {};
\draw (0.5,-0.25 -| a.center) to (a.center);
\draw [fill=\fillcomp] (-1,4)
    to (-1,2)
    to [out=down, in=\nwangle, out looseness=1.5] (a.center)
    to [out=\neangle, in=down, in looseness=1.5] (2,2)
    to [out=up, in=\seangle] (b.center)
    to [out=\swangle, in=up] +(-0.5,-0.5)
    to [out=down, in=down, looseness=2] +(-1,0)
    to (0.0,4);
\draw (b.center)
    to [out=up, in=down] (c.center);
\draw [fill=\fillcomp] (1,4) to (1,3.5)
    to [out=down, in=\nwangle] (c.center)
    to [out=\neangle, in=down] +(0.5,0.5)
    to (2,4);
\end{tikzpicture}
\end{aligned}
\quad=\quad
\frac{1}{\sqrt{n}}\,
\begin{aligned}
\begin{tikzpicture}[scale=\compscale, thick]
\node (a) [Vertex, vertex colour=red] at (-0.5,2) {};
\draw (-0.5,1.5) to (a.center);
\draw [fill=\fillcomp] (-1,5.75)
    to (-1,2.5)
    to [out=down, in=\nwangle] (a.center)
    to [out=\neangle, in=down] (0,2.5)
    to +(0,3.25);
\draw [fill=\fillcomp] (1,5.75)
    to (1,3)
    to [out=down, in=left] (1.5,2.5)
    to [out=right, in=down] +(0.5,0.5)
    to +(0,2.75);
\node [minimum width=30pt, draw, fill=white] at (0.5,4.25) {$\phi$};
\end{tikzpicture}
\end{aligned}
\end{equation}
The normalization constant ensures that the vertex representing uniform creation of classical data is an isometry.

We define the following vertices to represent the multiplications and comultiplications of the commutative \dag\-Frobenius algebras arising from our two measurement bases, in the manner of Theorem~\ref{thm:measurementclassicalstructure}:
\def\tempwidth{0.1cm}
\def\tempscale{0.3}
\begin{calign}
\begin{aligned}
\begin{tikzpicture}[scale=\tempscale]
\node (a) [Vertex, vertex colour=red] at (0,0) {};
\draw [thick] (0,2.5) to (a.center);
\draw [thick] (a.center) to [out=\swangle, in=up] (-1,-2.5);
\draw [thick] (a.center) to [out=\seangle, in=up] (1,-2.5);
\end{tikzpicture}
\end{aligned}
\hspace{\tempwidth}
:=
\hspace{\tempwidth}
\begin{aligned}
\begin{tikzpicture}[scale=\tempscale, thick]
    \node (V) [Vertex, vertex colour=red] at (1.5,0.5) {};
    \node (W) [Vertex, vertex colour=red] at (2.5,3.5) {};
    \node (X) [Vertex, vertex colour=red] at (3.5,0.5) {};
    \draw [fill=\fillcomp, draw=none]
        (2.5,3.5)
        to [out=\seangle, in=up] (3,3)
        to [out=down, in=up, in looseness=1.5] (4,1)
        to [out=down, in=\neangle] (3.5,0.5)
        to [in=down, out=\nwangle] (3,1)
        to [out=up, in=up, looseness=1.5] (2,1)
        to [out=down, in=\neangle] (V.center)
        to [out=\nwangle, in=down] (1,1)
        to [out=up, in=down, out looseness=1.5] (2,3)
        to [out=up, in=\swangle] (2.5,3.5);
    \draw
        (2.5,3.5)
        to [out=\seangle, in=up] (3,3)
        to [out=down, in=up, in looseness=1.5] (4,1)
        to [out=down, in=\neangle]  (3.5,0.5);
    \draw [thick] (3.5,0)
        to (3.5,0.5)
        to [out=\nwangle, in=down] (3,1)
        to [out=up, in=up, looseness=1.5] (2,1)
        to [out=down, in=\neangle] (V.center);
    \draw [thick] (2.5,4) to (2.5,3.5)
        to [out=\swangle, in=up] (2,3)
        to [out=down, in=up, in looseness=1.5] (1,1)
        to [out=down, in=\nwangle]
            node [auto, swap, pos=0.4, inner sep=1pt] {}
            (V.center);
    \draw (V.center) to +(0,-1);
    \draw (W.center) to +(0,1);
    \draw (X.center) to +(0,-1);
\end{tikzpicture}
\end{aligned}
&
\begin{aligned}
\begin{tikzpicture}[scale=\tempscale, yscale=-1]
\node (a) [Vertex, vertex colour=red] at (0,0) {};
\draw [thick] (0,2.5) to (a.center);
\draw [thick] (a.center) to [out=\swangle, in=up] (-1,-2.5);
\draw [thick] (a.center) to [out=\seangle, in=up] (1,-2.5);
\end{tikzpicture}
\end{aligned}
\hspace{\tempwidth}
:=
\hspace{\tempwidth}
\begin{aligned}
\begin{tikzpicture}[scale=\tempscale, yscale=-1]
    \node (V) [Vertex, vertex colour=red] at (1.5,0.5) {};
    \node (W) [Vertex, vertex colour=red] at (2.5,3.5) {};
    \node (X) [Vertex, vertex colour=red] at (3.5,0.5) {};
    \draw [fill=\fillcomp, draw=none]
        (2.5,3.5)
        to [out=\seangle, in=up] (3,3)
        to [out=down, in=up, in looseness=1.5] (4,1)
        to [out=down, in=\neangle] (3.5,0.5)
        to [in=down, out=\nwangle] (3,1)
        to [out=up, in=up, looseness=1.5] (2,1)
        to [out=down, in=\neangle] (V.center)
        to [out=\nwangle, in=down] (1,1)
        to [out=up, in=down, out looseness=1.5] (2,3)
        to [out=up, in=\swangle] (2.5,3.5);
    \draw [thick]
        (2.5,3.5)
        to [out=\seangle, in=up] (3,3)
        to [out=down, in=up, in looseness=1.5] (4,1)
        to [out=down, in=\neangle]  (3.5,0.5);
    \draw [thick] (3.5,0)
        to (3.5,0.5)
        to [out=\nwangle, in=down] (3,1)
        to [out=up, in=up, looseness=1.5] (2,1)
        to [out=down, in=\neangle] (V.center);
    \draw [thick]
        (1.5,0)
        to node [auto, swap] {} (1.5,0.5);
    \draw [thick] (2.5,4) to (2.5,3.5)
        to [out=\swangle, in=up] (2,3)
        to [out=down, in=up, in looseness=1.5] (1,1)
        to [out=down, in=\nwangle]
            node [auto, swap, pos=0.4, inner sep=1pt] {}
            (V.center);
    \draw [thick] (V.center)
        to [out=\neangle, in=down] (2,1);
    \draw [thick] (V.center) to +(0,-1);
    \draw [thick] (W.center) to +(0,1);
    \draw [thick] (X.center) to +(0,-1);
\end{tikzpicture}
\end{aligned}
&
\begin{aligned}
\begin{tikzpicture}[scale=\tempscale]
\node (a) [Vertex, vertex colour=red] at (0,0) {};
\draw [thick] (0,2.5) to (a.center);
\draw [white] (0,-2.5) to (a.center);
\end{tikzpicture}
\end{aligned}
\hspace{\tempwidth}
:=
\hspace{\tempwidth}
\begin{aligned}
\begin{tikzpicture}[scale=\tempscale]
\node (a) [Vertex, vertex colour=red] at (0,0.5) {};
\draw [thick] (0,2.5) to (a.center);
\draw [white] (0,-2.5) to (a.center);
\draw [thick, fill=\fillcomp] (0,0.5)
    to [out=\seangle, in=up] (0.5,0)
    to [out=down, in=down, looseness=2] (-0.5,0)
    to [out=up, in=\swangle] (0,0.5);
\end{tikzpicture}
\end{aligned}
&
\begin{aligned}
\begin{tikzpicture}[scale=\tempscale, yscale=-1]
\node (a) [Vertex, vertex colour=red] at (0,0) {};
\draw [thick] (0,2.5) to (a.center);
\draw [white] (0,-2.5) to (a.center);
\end{tikzpicture}
\end{aligned}
\hspace{\tempwidth}
:=
\hspace{\tempwidth}
\begin{aligned}
\begin{tikzpicture}[scale=\tempscale, yscale=-1]
\node (a) [Vertex, vertex colour=red] at (0,0.5) {};
\draw [thick] (0,2.5) to (a.center);
\draw [white] (0,-2.5) to (a.center);
\draw [thick, fill=\fillcomp] (0,0.5)
    to [out=\seangle, in=up] (0.5,0)
    to [out=down, in=down, looseness=2] (-0.5,0)
    to [out=up, in=\swangle] (0,0.5);
\end{tikzpicture}
\end{aligned}
\\
\begin{aligned}
\begin{tikzpicture}[scale=\tempscale]
\node (a) [Vertex, vertex colour=green] at (0,0) {};
\draw [thick] (0,2.5) to (a.center);
\draw [thick] (a.center) to [out=\swangle, in=up] (-1,-2.5);
\draw [thick] (a.center) to [out=\seangle, in=up] (1,-2.5);
\end{tikzpicture}
\end{aligned}
\hspace{\tempwidth}
:=
\hspace{\tempwidth}
\begin{aligned}
\begin{tikzpicture}[scale=\tempscale]
    \node (V) [Vertex, vertex colour=green] at (1.5,0.5) {};
    \node (W) [Vertex, vertex colour=green] at (2.5,3.5) {};
    \node (X) [Vertex, vertex colour=green] at (3.5,0.5) {};
    \draw [fill=\fillcomp, draw=none]
        (2.5,3.5)
        to [out=\seangle, in=up] (3,3)
        to [out=down, in=up, in looseness=1.5] (4,1)
        to [out=down, in=\neangle] (3.5,0.5)
        to [in=down, out=\nwangle] (3,1)
        to [out=up, in=up, looseness=1.5] (2,1)
        to [out=down, in=\neangle] (V.center)
        to [out=\nwangle, in=down] (1,1)
        to [out=up, in=down, out looseness=1.5] (2,3)
        to [out=up, in=\swangle] (2.5,3.5);
    \draw [thick]
        (2.5,3.5)
        to [out=\seangle, in=up] (3,3)
        to [out=down, in=up, in looseness=1.5] (4,1)
        to [out=down, in=\neangle]  (3.5,0.5);
    \draw [thick] (3.5,0)
        to (3.5,0.5)
        to [out=\nwangle, in=down] (3,1)
        to [out=up, in=up, looseness=1.5] (2,1)
        to [out=down, in=\neangle] (V.center);
    \draw [thick]
        (1.5,0)
        to node [auto, swap] {} (1.5,0.5);
    \draw [thick] (2.5,4) to (2.5,3.5)
        to [out=\swangle, in=up] (2,3)
        to [out=down, in=up, in looseness=1.5] (1,1)
        to [out=down, in=\nwangle]
            node [auto, swap, pos=0.4, inner sep=1pt] {}
            (V.center);
    \draw [thick] (V.center)
        to [out=\neangle, in=down] (2,1);
    \draw [thick] (V.center) to +(0,-1);
    \draw [thick] (W.center) to +(0,1);
    \draw [thick] (X.center) to +(0,-1);
\end{tikzpicture}
\end{aligned}
&
\begin{aligned}
\begin{tikzpicture}[scale=\tempscale, yscale=-1]
\node (a) [Vertex, vertex colour=green] at (0,0) {};
\draw [thick] (0,2.5) to (a.center);
\draw [thick] (a.center) to [out=\swangle, in=up] (-1,-2.5);
\draw [thick] (a.center) to [out=\seangle, in=up] (1,-2.5);
\end{tikzpicture}
\end{aligned}
\hspace{\tempwidth}
:=
\hspace{\tempwidth}
\begin{aligned}
\begin{tikzpicture}[scale=\tempscale, yscale=-1]
    \node (V) [Vertex, vertex colour=green] at (1.5,0.5) {};
    \node (W) [Vertex, vertex colour=green] at (2.5,3.5) {};
    \node (X) [Vertex, vertex colour=green] at (3.5,0.5) {};
    \draw [fill=\fillcomp, draw=none]
        (2.5,3.5)
        to [out=\seangle, in=up] (3,3)
        to [out=down, in=up, in looseness=1.5] (4,1)
        to [out=down, in=\neangle] (3.5,0.5)
        to [in=down, out=\nwangle] (3,1)
        to [out=up, in=up, looseness=1.5] (2,1)
        to [out=down, in=\neangle] (V.center)
        to [out=\nwangle, in=down] (1,1)
        to [out=up, in=down, out looseness=1.5] (2,3)
        to [out=up, in=\swangle] (2.5,3.5);
    \draw [thick]
        (2.5,3.5)
        to [out=\seangle, in=up] (3,3)
        to [out=down, in=up, in looseness=1.5] (4,1)
        to [out=down, in=\neangle]  (3.5,0.5);
    \draw [thick] (3.5,0)
        to (3.5,0.5)
        to [out=\nwangle, in=down] (3,1)
        to [out=up, in=up, looseness=1.5] (2,1)
        to [out=down, in=\neangle] (V.center);
    \draw [thick]
        (1.5,0)
        to node [auto, swap] {} (1.5,0.5);
    \draw [thick] (2.5,4) to (2.5,3.5)
        to [out=\swangle, in=up] (2,3)
        to [out=down, in=up, in looseness=1.5] (1,1)
        to [out=down, in=\nwangle]
            node [auto, swap, pos=0.4, inner sep=1pt] {}
            (V.center);
    \draw [thick] (V.center)
        to [out=\neangle, in=down] (2,1);
    \draw [thick] (V.center) to +(0,-1);
    \draw [thick] (W.center) to +(0,1);
    \draw [thick] (X.center) to +(0,-1);
\end{tikzpicture}
\end{aligned}
&
\begin{aligned}
\begin{tikzpicture}[scale=\tempscale]
\node (a) [Vertex, vertex colour=green] at (0,0) {};
\draw [thick] (0,2.5) to (a.center);
\draw [white] (0,-2.5) to (a.center);
\end{tikzpicture}
\end{aligned}
\hspace{\tempwidth}
:=
\hspace{\tempwidth}
\begin{aligned}
\begin{tikzpicture}[scale=\tempscale]
\node (a) [Vertex, vertex colour=green] at (0,0.5) {};
\draw [thick] (0,2.5) to (a.center);
\draw [white] (0,-2.5) to (a.center);
\draw [thick, fill=\fillcomp] (0,0.5)
    to [out=\seangle, in=up] (0.5,0)
    to [out=down, in=down, looseness=2] (-0.5,0)
    to [out=up, in=\swangle] (0,0.5);
\end{tikzpicture}
\end{aligned}
&
\begin{aligned}
\begin{tikzpicture}[scale=\tempscale, yscale=-1]
\node (a) [Vertex, vertex colour=green] at (0,0) {};
\draw [thick] (0,2.5) to (a.center);
\draw [white] (0,-2.5) to (a.center);
\end{tikzpicture}
\end{aligned}
\hspace{\tempwidth}
:=
\hspace{\tempwidth}
\begin{aligned}
\begin{tikzpicture}[scale=\tempscale, yscale=-1]
\node (a) [Vertex, vertex colour=green] at (0,0.5) {};
\draw [thick] (0,2.5) to (a.center);
\draw [white] (0,-2.5) to (a.center);
\draw [thick, fill=\fillcomp] (0,0.5)
    to [out=\seangle, in=up] (0.5,0)
    to [out=down, in=down, looseness=2] (-0.5,0)
    to [out=up, in=\swangle] (0,0.5);
\end{tikzpicture}
\end{aligned}
\end{calign}
We now use these to demonstrate equivalences between three different axiomatizations of the complementarity condition. Condition 2 below stands out as the simplest of these.

\begin{theorem}
\label{thm:complementary}
Given a pair of nondegenerate measurements indexed by classical information of dimension $n$, the following properties are equivalent:
\begin{enumerate}
\item There exists a controlled phase $\phi$ satisfying the physical complementarity condition:
\[
\begin{aligned}
\begin{tikzpicture}[scale=0.7, thick]
\node (a) [Vertex, vertex colour=red] at (0.5,0.25  ) {};
\node (b) [Vertex, vertex colour=red] at (1.5, 2.5) {};
\node (c) [Vertex, vertex colour=green] at (1.5, 3.00) {};
\draw (0.5,-0.25 -| a.center) to (a.center);
\draw [fill=\fillcomp] (-1,4)
    to (-1,2)
    to [out=down, in=\nwangle, out looseness=1.5] (a.center)
    to [out=\neangle, in=down, in looseness=1.5] (2,2)
    to [out=up, in=\seangle] (b.center)
    to [out=\swangle, in=up] +(-0.5,-0.5)
    to [out=down, in=down, looseness=2] +(-1,0)
    to (0.0,4);
\draw (b.center)
    to [out=up, in=down] (c.center);
\draw [fill=\fillcomp] (1,4) to (1,3.5)
    to [out=down, in=\nwangle] (c.center)
    to [out=\neangle, in=down] +(0.5,0.5)
    to (2,4);
\end{tikzpicture}
\end{aligned}
\quad=\quad
\frac{1}{\sqrt{n}}\,
\begin{aligned}
\begin{tikzpicture}[scale=\compscale, thick]
\node (a) [Vertex, vertex colour=red] at (-0.5,2) {};
\draw (-0.5,1.5) to (a.center);
\draw [fill=\fillcomp] (-1,5.75)
    to (-1,2.5)
    to [out=down, in=\nwangle] (a.center)
    to [out=\neangle, in=down] (0,2.5)
    to +(0,3.25);
\draw [fill=\fillcomp] (1,5.75)
    to (1,3)
    to [out=down, in=left] (1.5,2.5)
    to [out=right, in=down] +(0.5,0.5)
    to +(0,2.75);
\node [minimum width=30pt, draw, fill=white] at (0.5,4.25) {$\phi$};
\end{tikzpicture}
\end{aligned}
\]
\item The following composite is horizontally unitary:
\begin{equation}
\label{eq:compcomposite}
\sqrt{n} \, 
\begin{aligned}
\begin{tikzpicture}[scale=0.7]
\node (R) [Vertex, vertex colour=red] at (0,1.5) {};
\node (G) [Vertex, vertex colour=green] at (0,2.5) {};
\draw [fill=\fillcomp, thick] (0.5,0.5)
    to [out=up, in=down] (0.5,1)
    to [out=up, in=\seangle] (R.center)
    to [out=\swangle, in=up] (-0.5,1.0)
    to (-0.5,0.5);
\draw [fill=\fillcomp, thick] (-0.5,3.5)
    to (-0.5,3) 
    to [out=down, in=\nwangle] (G.center)
    to [out=\neangle, in=down] (0.5,3)
    to (0.5,3.5);
\draw [thick] (G.center) to (R.center);
\end{tikzpicture}
\end{aligned}
\end{equation}
\item The condition of Coecke and Duncan~\cite[Theorem~8.10]{cd07-gcnc} holds:
\begin{equation}
n
\begin{aligned}
\begin{tikzpicture}[thick]
\node (G) [Vertex, vertex colour=red] at (0,1.5) {};
\node (R) [Vertex, vertex colour=green] at (0.5,-0.5) {};
\node (g1) [Vertex, vertex colour=red] at (-0.5,-0.5) {};
\node (r1) [Vertex, vertex colour=green] at (0,0.5) {};
\node (g2) [Vertex, vertex colour=red] at (-0.5,-0.9) {};
\node (r2) [Vertex, vertex colour=green] at (0,0.9) {};
\draw (G.center) to +(0,0.75);
\draw (G.center) to [out=\swangle, in=\nwangle] (g1.center);
\draw (g1.center) to [out=\neangle, in=\swangle] (r1.center);
\draw (r1.center) to [out=\seangle, in=\nwangle] (R.center);
\draw (G.center) to [out=\seangle, in=\neangle] (R.center);
\draw (R.center) to +(0,-0.75);
\draw (r1.center) to (r2.center);
\draw (g1.center) to (g2.center);
\end{tikzpicture}
\end{aligned}
\quad=\quad\,\,\,\,
\begin{aligned}
\begin{tikzpicture}[thick]
\node (G) [Vertex, vertex colour=red] at (0,0.5) {};
\node (R) [Vertex, vertex colour=green] at (0,-1.5) {};
\draw (G.center) to +(0,0.75);
\draw (R.center) to +(0,-0.75);
\end{tikzpicture}
\end{aligned}
\end{equation}
\end{enumerate}
\end{theorem}
\begin{proof}
We will show $1 \Leftrightarrow 2$ and $2 \Leftrightarrow 3$.

The implication $1 \Leftrightarrow 2$ goes as follows:
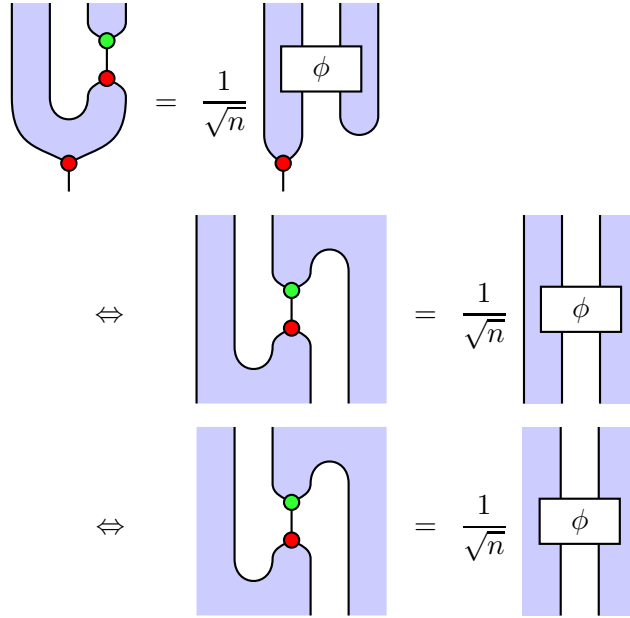
\def\compscale{0.5}
\begin{calign}
\nonumber
\begin{aligned}
\begin{tikzpicture}[scale=\compscale, thick]
\node (a) [Vertex, vertex colour=red] at (0.5,0.25  ) {};
\node (b) [Vertex, vertex colour=red] at (1.5, 2.5) {};
\node (c) [Vertex, vertex colour=green] at (1.5, 3.5) {};
\draw (0.5,-0.5 -| a.center) to (a.center);
\draw [fill=\fillcomp] (-1,4.5)
    to (-1,2)
    to [out=down, in=\nwangle, out looseness=1.5] (a.center)
    to [out=\neangle, in=down, in looseness=1.5] (2,2)
    to [out=up, in=\seangle] (b.center)
    to [out=\swangle, in=up] +(-0.5,-0.5)
    to [out=down, in=down, looseness=2] +(-1,0)
    to (0.0,4.5);
\draw (b.center)
    to [out=up, in=down] (c.center);
\draw [fill=\fillcomp] (1,4.5) to (1,4)
    to [out=down, in=\nwangle] (c.center)
    to [out=\neangle, in=down] +(0.5,0.5)
    to (2,4.5);
\end{tikzpicture}
\end{aligned}
\,\,\,\,=\,\,\, \frac{1}{\sqrt{n}} \,\,
\begin{aligned}
\begin{tikzpicture}[scale=\compscale, thick]
\node (a) [Vertex, vertex colour=red] at (-0.5,1.75) {};
\draw (-0.5,1.0) to (a.center);
\draw [fill=\fillcomp] (-1,6)
    to (-1,2.5)
    to [out=down, in=\nwangle] (a.center)
    to [out=\neangle, in=down] (0,2.5)
    to +(0,3.5);
\draw [fill=\fillcomp] (1,6)
    to (1,3)
    to [out=down, in=left] (1.5,2.5)
    to [out=right, in=down] +(0.5,0.5)
    to +(0,3);
\node [minimum width=30pt, draw, fill=white] at (0.5,4.25) {$\phi$};
\end{tikzpicture}
\end{aligned}
\hspace{130pt}
\\
\nonumber
\Leftrightarrow 
\qquad
\begin{aligned}
\begin{tikzpicture}[scale=\compscale]
\node (R) [Vertex, vertex colour=red] at (0,1.5) {};
\node (G) [Vertex, vertex colour=green] at (0,2.5) {};
\draw [fill=\fillcomp, draw=none] (0.5,-0.5)
    to [out=up, in=down] (0.5,1)
    to [out=up, in=\seangle] (R.center)
    to [out=\swangle, in=up] (-0.5,1)
    to [out=down, in=down, looseness=2] (-1.5,1)
    to (-1.5,4.5) to (-2.5,4.5) to (-2.5,-0.5);
\draw [thick] (0.5,-0.5)
    to [out=up, in=down] (0.5,1)
    to [out=up, in=\seangle] (R.center)
    to [out=\swangle, in=up] (-0.5,1)
    to [out=down, in=down, looseness=2] (-1.5,1)
    to (-1.5,4.5);
\draw [fill=\fillcomp, draw=none] (-0.5,4.5)
    to (-0.5,3) 
    to [out=down, in=\nwangle] (G.center)
    to [out=\neangle, in=down] (0.5,3)
    to [out=up, in=up, looseness=2] (1.5,3)
    to (1.5,-0.5)
    to (2.5,-0.5)
    to (2.5,4.5);
\draw [thick] (-0.5,4.5)
    to (-0.5,3) 
    to [out=down, in=\nwangle] (G.center)
    to [out=\neangle, in=down] (0.5,3)
    to [out=up, in=up, looseness=2] (1.5,3)
    to (1.5,-0.5);
\draw [thick] (G.center) to (R.center);
\draw [thick] (-2.5,-0.5) to (-2.5,4.5);
\end{tikzpicture}
\end{aligned}
\,\,\,\,=\,\,\,\frac{1}{\sqrt{n}} \,\,
\begin{aligned}
\begin{tikzpicture}[scale=\compscale, thick]
\path [fill=\fillcomp] (0,-2) to (0,3) to (1,3) to (1,-2);
\path [fill=\fillcomp] (2,-2) to (2,3) to (3,3) to (3,-2);
\draw (1,-2) to (1,3);
\draw (2,-2) to (2,3);
\node [minimum width=30pt, draw, fill=white] at (1.5,0.5) {$\phi$};
\draw [thick] (0,-2) to (0,3);
\end{tikzpicture}
\end{aligned}
\\
\nonumber
\Leftrightarrow 
\qquad
\begin{aligned}
\begin{tikzpicture}[scale=\compscale]
\node (R) [Vertex, vertex colour=red] at (0,1.5) {};
\node (G) [Vertex, vertex colour=green] at (0,2.5) {};
\draw [fill=\fillcomp, draw=none] (0.5,-0.5)
    to [out=up, in=down] (0.5,1)
    to [out=up, in=\seangle] (R.center)
    to [out=\swangle, in=up] (-0.5,1)
    to [out=down, in=down, looseness=2] (-1.5,1)
    to (-1.5,4.5) to (-2.5,4.5) to (-2.5,-0.5);
\draw [thick] (0.5,-0.5)
    to [out=up, in=down] (0.5,1)
    to [out=up, in=\seangle] (R.center)
    to [out=\swangle, in=up] (-0.5,1)
    to [out=down, in=down, looseness=2] (-1.5,1)
    to (-1.5,4.5);
\draw [fill=\fillcomp, draw=none] (-0.5,4.5)
    to (-0.5,3) 
    to [out=down, in=\nwangle] (G.center)
    to [out=\neangle, in=down] (0.5,3)
    to [out=up, in=up, looseness=2] (1.5,3)
    to (1.5,-0.5)
    to (2.5,-0.5)
    to (2.5,4.5);
\draw [thick] (-0.5,4.5)
    to (-0.5,3) 
    to [out=down, in=\nwangle] (G.center)
    to [out=\neangle, in=down] (0.5,3)
    to [out=up, in=up, looseness=2] (1.5,3)
    to (1.5,-0.5);
\draw [thick] (G.center) to (R.center);
\end{tikzpicture}
\end{aligned}
\,\,\,\,=\,\,\, \frac{1}{\sqrt{n}} \,\,
\begin{aligned}
\begin{tikzpicture}[scale=\compscale, thick]
\path [fill=\fillcomp] (0,-2) to (0,3) to (1,3) to (1,-2);
\path [fill=\fillcomp] (2,-2) to (2,3) to (3,3) to (3,-2);
\draw (1,-2) to (1,3);
\draw (2,-2) to (2,3);
\node [minimum width=30pt, draw, fill=white] at (1.5,0.5) {$\phi$};
\end{tikzpicture}
\end{aligned}
\end{calign}
For the first equivalence we apply the inverse to the red measurement vertex at the bottom of both sides, and `bend down' the upper-right line using the comparison 2\-cell~\eqref{eq:compare}. For the second equivalence we bend down the top-left line using the comparison 2\-cell and compose with the copying 2\-cell, using axiom~\eqref{eq:copycompare} to remove the resulting hole. Since $\phi$ is unitary, by Lemma~\ref{lem:horizontalfromvertical} we conclude that expression~\eqref{eq:compcomposite} is horizontally unitary. This establishes proposition 2 above. This argument is reversible, so we also have $2 \Rightarrow 1$.

To prove the equivalence $2 \Leftrightarrow 3$ we use the following graphical argument.
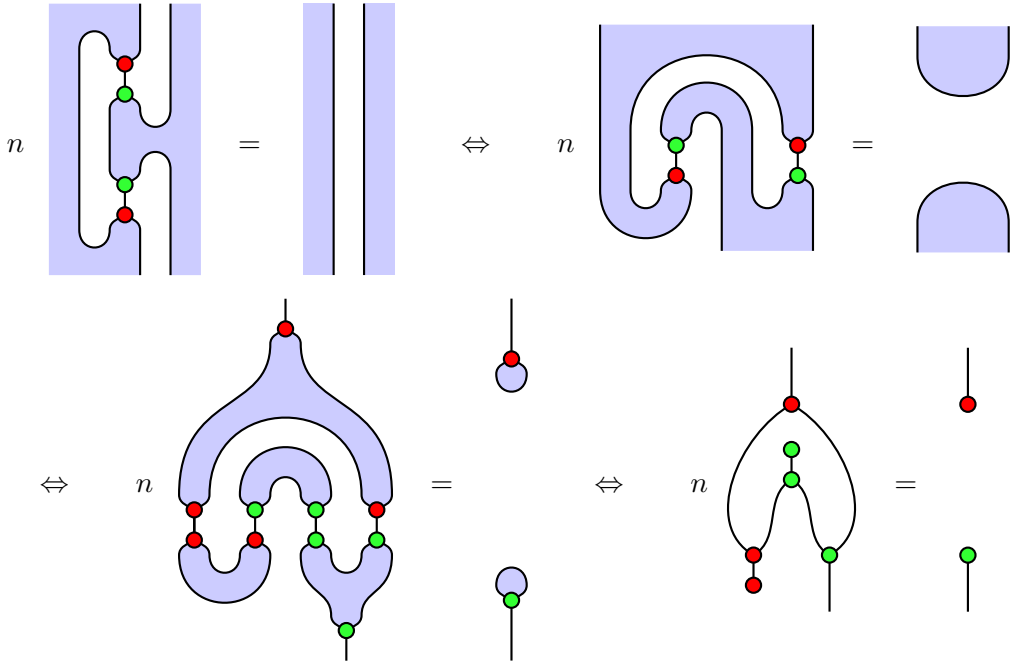
\def\compscale{0.4}
\begin{calign}
\nonumber
n\,\,\,\,\begin{aligned}
\begin{tikzpicture}[scale=\compscale, thick]
\node (R) [Vertex, vertex colour=red] at (0,1.5) {};
\node (G) [Vertex, vertex colour=green] at (0,2.5) {};
\node (R2) [Vertex, vertex colour=red] at (0,6.5) {};
\node (G2) [Vertex, vertex colour=green] at (0,5.5) {};
\draw [fill=\fillcomp, draw=none] (0.5,-0.5)
    to [out=up, in=down] (0.5,1)
    to [out=up, in=\seangle] (R.center)
    to [out=\swangle, in=up] (-0.5,1)
    to [out=down, in=down, looseness=2] (-1.5,1)
    to (-1.5,7)
    to [out=up, in=up, looseness=2] (-0.5,7)
    to [out=down, in=\nwangle] (R2.center)
    to [out=\neangle, in=down] (0.5,7)
    to (0.5,8.5) to (-2.5,8.5) to (-2.5,-0.5);
\draw (0.5,-0.5)
    to [out=up, in=down] (0.5,1)
    to [out=up, in=\seangle] (R.center)
    to [out=\swangle, in=up] (-0.5,1)
    to [out=down, in=down, looseness=2] (-1.5,1)
    to (-1.5,7.0)
    to [out=up, in=up, looseness=2] (-0.5,7.0)
    to [out=down, in=\nwangle] (R2.center)
    to [out=\neangle, in=down] (0.5, 7.0)
    to (0.5,8.5);
\draw [thick] (G.center) to (R.center);
\draw (R2.center) to (G2.center);
\draw [fill=\fillcomp, draw=none] (G2.center)
    to [out=\swangle, in=up] (-0.5,5)
    to (-0.5,3) 
    to [out=down, in=\nwangle] (G.center)
    to [out=\neangle, in=down] (0.5,3)
    to [out=up, in=up, looseness=2] (1.5,3)
    to (1.5,-0.5)
    to (2.5,-0.5)
    to (2.5,8.5)
    to (1.5,8.5)
    to (1.5,5)
    to [out=down, in=down, looseness=2] (0.5,5)
    to [out=up, in=\seangle] (G2.center);
\draw (1.5,8.5)
    to (1.5,5)
    to [out=down, in=down, looseness=2] (0.5,5)
    to [out=up, in=\seangle] (G2.center)
    to [out=\swangle, in=up] (-0.5,5) to (-0.5,3) 
    to [out=down, in=\nwangle] (G.center)
    to [out=\neangle, in=down] (0.5,3)
    to [out=up, in=up, looseness=2] (1.5,3)
    to (1.5,-0.5);
\end{tikzpicture}
\end{aligned}
\quad=\quad
\begin{aligned}
\begin{tikzpicture}[scale=\compscale, thick]
\path [fill=\fillcomp] (0,-2) to (0,7) to (1,7) to (1,-2);
\path [fill=\fillcomp] (2,-2) to (2,7) to (3,7) to (3,-2);
\draw (1,-2) to (1,7);
\draw (2,-2) to (2,7);
\end{tikzpicture}
\end{aligned}
\qquad\Leftrightarrow \qquad
n\,\,\,\,\begin{aligned}
\begin{tikzpicture}[thick, scale=\compscale]
    \node (G1) [Vertex, vertex colour=red] at (5,3) {};
    \node (G2) [Vertex, vertex colour=red] at (1,2) {};
    \node (R1) [Vertex, vertex colour=green] at (1,3) {};
    \node (R2) [Vertex, vertex colour=green] at (5,2) {};
    \draw (G1.center) to [out=down, in=up] (R2.center);
    \draw [fill=\fillcomp] (5.5,7)
        to [out=down, in=up] (5.5,3.5)
        to [out=down, in=\neangle] (G1.center)
        to [out=\nwangle, in=down] (4.5,3.5)
        to [out=up, in=up, looseness=1.7] (-0.5,3.5)
        to (-0.5,1.5)
        to [out=down, in=down, looseness=2] (0.5,1.5)
        to [out=up, in=\swangle] (G2.center)
        to [out=\seangle, in=up] (1.5,1.5)
        to [out=down, in=down, looseness=1.8] (-1.5,1.5)
        to (-1.5,4.5)
        to [out=up, in=down] (-1.5,7);
    \draw (R1.center) to (G2.center);
    \draw [fill=\fillcomp] (2.5,-0.5) to [out=up, in=down] (2.5,1) to (2.5,3.5) to [out=up, in=up, looseness=2] (1.5,3.5) to [out=down, in=\neangle] (R1.center) to [out=\nwangle, in=down] (0.5,3.5) to [out=up, in=up, looseness=1.8] (3.5,3.5) to (3.5,1.5) to [out=down, in=down, looseness=2] (4.5,1.5) to [out=up, in=\swangle] (R2.center) to [out=\seangle, in=up] (5.5,1.5) to [out=down, in=up] (5.5,-0.5);
\end{tikzpicture}
\end{aligned}
\quad=\quad
\begin{aligned}
\begin{tikzpicture}[thick, scale=\compscale]
    \draw [fill=\fillcomp] (5.5,7) to (5.5,6) to [out=down, in=down, looseness=1.5] (2.5,6) to (2.5,7);
    \draw [fill=\fillcomp] (2.5,-0.5)
        to (2.5,0.5)
        to [out=up, in=up, looseness=1.5] (5.5,0.5)
        to (5.5,-0.5);
\end{tikzpicture}
\end{aligned}
\\
\Leftrightarrow
\qquad
n\,\,\,\,\begin{aligned}
\begin{tikzpicture}[thick, scale=\compscale]
    \node (G1) [Vertex, vertex colour=red] at (5,3) {};
    \node (G2) [Vertex, vertex colour=red] at (1,2) {};
    \node (G3) [Vertex, vertex colour=red] at (2,9) {};
    \node (G4) [Vertex, vertex colour=red] at (-1,3) {};
    \node (G5) [Vertex, vertex colour=red] at (-1,2) {};
    \node (R1) [Vertex, vertex colour=green] at (1,3) {};
    \node (R2) [Vertex, vertex colour=green] at (5,2) {};
    \node (R3) [Vertex, vertex colour=green] at (3,3) {};
    \node (R4) [Vertex, vertex colour=green] at (3,2) {};
    \node (R5) [Vertex, vertex colour=green] at (4,-1) {};
    \draw (G1.center) to [out=down, in=up] (R2.center);
    \draw [fill=\fillcomp] (2,10)
        to (G3.center)
        to [out=\seangle, in=up] (2.5,8.5)
        to [out=down, in=up, in looseness=1.5] (5.5,3.5)
        to [out=down, in=\neangle] (G1.center)
        to [out=\nwangle, in=down] (4.5,3.5)
        to [out=up, in=up, looseness=1.75] (-0.5,3.5)
        to [out=down, in=\neangle] (G4.center)
        to (G5.center)
        to [out=\seangle, in=up] (-0.5,1.5)
        to [out=down, in=down, looseness=2] (0.5,1.5)
        to [out=up, in=\swangle] (G2.center)
        to [out=\seangle, in=up] (1.5,1.5)
        to [out=down, in=down, looseness=1.8] (-1.5,1.5)
        to [out=up, in=\swangle] (G5.center)
        to (G4.center)
        to [out=\nwangle, in=down] (-1.5,3.5)
        to [out=up, in=down, out looseness=1.5] (1.5,8.5)
        to [out=up, in=\swangle] (G3.center);
    \draw (R1.center) to (G2.center);
    \draw [fill=\fillcomp] (4,-2) to (R5.center) to [out=\nwangle, in=down] (3.5,-0.5) to [out=up, in=down, in looseness=1.5] (2.5,1.5) to [out=up, in=\swangle] (R4.center) to (R3.center) to [out=\nwangle, in=down] (2.5,3.5) to [out=up, in=up, looseness=2] (1.5,3.5) to [out=down, in=\neangle] (R1.center) to [out=\nwangle, in=down] (0.5,3.5) to [out=up, in=up, looseness=1.8] (3.5,3.5) to [out=down, in=\neangle] (R3.center) to (R4.center) to [out=\seangle, in=up] (3.5,1.5) to [out=down, in=down, looseness=2] (4.5,1.5) to [out=up, in=\swangle] (R2.center) to [out=\seangle, in=up] (5.5,1.5) to [out=down, in=up, out looseness=1.5] (4.5,-0.5) to [out=down, in=\neangle] (R5.center);
\end{tikzpicture}
\end{aligned}
\quad=\quad
\begin{aligned}
\begin{tikzpicture}[thick, scale=\compscale]
\node (a) [Vertex, vertex colour=red] at (0,0.5) {};
\draw [thick] (0,2.5) to (a.center);
\draw [thick, fill=\fillcomp] (0,0.5)
    to [out=\seangle, in=up] (0.5,0)
    to [out=down, in=down, looseness=2] (-0.5,0)
    to [out=up, in=\swangle] (0,0.5);
\begin{scope}[yscale=-1, yshift=7cm]
\node (b) [Vertex, vertex colour=green] at (0,0.5) {};
\draw [thick] (0,2.5) to (b.center);
\draw [thick, fill=\fillcomp] (0,0.5)
    to [out=\seangle, in=up] (0.5,0)
    to [out=down, in=down, looseness=2] (-0.5,0)
    to [out=up, in=\swangle] (0,0.5);
\end{scope}
\end{tikzpicture}
\end{aligned}
\qquad\Leftrightarrow
\qquad
n\hspace{-5pt}
\begin{aligned}
\begin{tikzpicture}[thick]
\node (G) [Vertex, vertex colour=red] at (0,1.5) {};
\node (R) [Vertex, vertex colour=green] at (0.5,-0.5) {};
\node (g1) [Vertex, vertex colour=red] at (-0.5,-0.5) {};
\node (r1) [Vertex, vertex colour=green] at (0,0.5) {};
\node (g2) [Vertex, vertex colour=red] at (-0.5,-0.9) {};
\node (r2) [Vertex, vertex colour=green] at (0,0.9) {};
\draw (G.center) to +(0,0.75);
\draw (G.center) to [out=\swangle, in=\nwangle] (g1.center);
\draw (g1.center) to [out=\neangle, in=\swangle] (r1.center);
\draw (r1.center) to [out=\seangle, in=\nwangle] (R.center);
\draw (G.center) to [out=\seangle, in=\neangle] (R.center);
\draw (R.center) to +(0,-0.75);
\draw (r1.center) to (r2.center);
\draw (g1.center) to (g2.center);
\end{tikzpicture}
\end{aligned}
\hspace{-10pt}\quad=\quad
\begin{aligned}
\begin{tikzpicture}[thick]
\node (G) [Vertex, vertex colour=red] at (0,0.5) {};
\node (R) [Vertex, vertex colour=green] at (0,-1.5) {};
\draw (G.center) to +(0,0.75);
\draw (R.center) to +(0,-0.75);
\end{tikzpicture}
\end{aligned}
\nonumber
\end{calign}
\end{proof}

\noindent
These results on complementarity provide a good context for discussing the differences between our 2\-categorical formalism and the monoidal category formalism used by Coecke, Duncan and others. Each of the 3 equivalent properties investigated here have their own distinct identity, and each has different advantages. Property~1 is a direct mathematical translation of the physical definition of complementarity, and makes the most direct sense. Property 2 is the most mathematically elegant and minimal. And property 3 requires the least machinery, as it is entirely 1\-categorical. Our theorem also serves to `explain' the  unusual form of this axiom.

Finally, thanks to this theorem, we note that we can define complementary observables in precisely the same terms as we described solutions to the teleportation and dense coding protocols in Theorems~\ref{thm:teleportationdoublyunitary} and \ref{thm:densecodingdoublyunitary}, giving a surprising unity to the foundations of all three phenomena.
\begin{theorem}
\label{thm:complementarydoublyunitary}
A vertically unitary preparation-measurement composite
\begin{equation}
\begin{aligned}
\begin{tikzpicture}[scale=0.7]
\node (R) [Vertex, vertex colour=red] at (0,1.5) {};
\node (G) [Vertex, vertex colour=green] at (0,2.5) {};
\draw [fill=\fillcomp, thick] (0.5,0.5)
    to [out=up, in=down] (0.5,1)
    to [out=up, in=\seangle] (R.center)
    to [out=\swangle, in=up] (-0.5,1.0)
    to (-0.5,0.5);
\draw [fill=\fillcomp, thick] (-0.5,3.5)
    to (-0.5,3) 
    to [out=down, in=\nwangle] (G.center)
    to [out=\neangle, in=down] (0.5,3)
    to (0.5,3.5);
\draw [thick] (G.center) to (R.center);
\end{tikzpicture}
\end{aligned}
\end{equation}
involves a pair of complementary observables if and only if it is also horizontally unitary, up to a scalar factor.
\end{theorem}
\begin{proof}
Straightforward from the proof of Theorem~\ref{thm:complementary}.
\end{proof}

\subsubsection*{Implementation}

We verify our physical complementarity condition~\eqref{eq:physicallycomplementary} by demonstrating that is it satisfied on a qubit, when the red basis is $\{\ket 0, \ket 1\}$ and the green basis is $\big\{\frac{1}{\sqrt{2}}(\ket{0} + \ket {1}), \frac{1}{\sqrt{2}} (\ket 0 - \ket 1) \big \}$. This gives the following values for the measurement vertices:
\begin{calign}
\begin{aligned}
\begin{tikzpicture}[thick]
    \node (V) [Vertex, vertex colour=red] at (1.5,1) {};
    \draw [fill=\fillA, draw=none]
        (2,1.6)
        to [out=down, in=\neangle] (V.center)
        to [out=\nwangle, in=down] (1,1.6);
    \draw
        (1.5,0.4)
        to (1.5,1.0);
    \draw (1,1.6)
        to [out=down, in=\nwangle] (V.center);
    \draw (V.center)
        to [out=\neangle, in=down] (2,1.6);
\end{tikzpicture}
\end{aligned}
\hspace{3pt}=\hspace{7pt}
\matrix{
\matrix{1 & 0
\\
0 & 1
}
}
&
\begin{aligned}
\begin{tikzpicture}[thick]
    \node (V) [Vertex, vertex colour=green] at (1.5,1) {};
    \draw [fill=\fillA, draw=none]
        (2,1.6)
        to [out=down, in=\neangle] (V.center)
        to [out=\nwangle, in=down] (1,1.6);
    \draw
        (1.5,0.4)
        to (1.5,1.0);
    \draw (1,1.6)
        to [out=down, in=\nwangle] (V.center);
    \draw (V.center)
        to [out=\neangle, in=down] (2,1.6);
\end{tikzpicture}
\end{aligned}
\hspace{3pt}=\hspace{7pt}
\matrix{
\frac 1 {\sqrt{2}}
\matrix{1 & 1
\\
1 & -1
}
}
\end{calign}
The controlled phase $\phi$ takes the following value:
\begin{equation}
\begin{aligned}
\begin{tikzpicture}[scale=\compscale, thick]
\path [fill=\fillcomp] (-0.25,-0.5) to (-0.25,1.5)
    to (1,1.5) to (1,-0.5);
\path [fill=\fillcomp] (2,-0.5) to (2,1.5)
    to (3.25,1.5) to (3.25,-0.5);
\draw (1,-0.5) to (1,1.5);
\draw (2,-0.5) to (2,1.5);
\node [minimum width=30pt, draw, fill=white] at (1.5,0.5) {$\phi$};
\end{tikzpicture}
\end{aligned}
\hspace{10pt}
=
\hspace{7pt}
\matrix{\matrix{1} & \matrix{1} \\ \matrix{1} & \matrix{-1} }
\end{equation}
It is a 2-by-2 matrix, all the entries of which are linear maps of type $\C \to \C$. It is unitary since each constituent linear map is unitary. The Mathematica notebook~\cite{v12-highernotebook} contains a verification that these components give a solution to equation~\eqref{eq:physicallycomplementary}.

\subsection{Quantum erasure}

\subsubsection*{Introduction}

Our study of complementarity gives a good framework for discussing \textit{quantum erasure}. This is the phenomenon whereby, after performing a quantum measurement as a part of a quantum procedure, erasing the resulting of the information can change the overall effect of the quantum procedure, even if the erasure is performed after the procedure has been completed. This is interesting to study as it fits badly with our classical intuition, which tells us that one should not be able to affect the result of an experiment by manipulating only systems disjoint from the measurement apparatus, especially when this manipulation is done only after the experiment is finished.

\subsubsection*{Analysis}

We analyze quantum erasure using the following protocol:
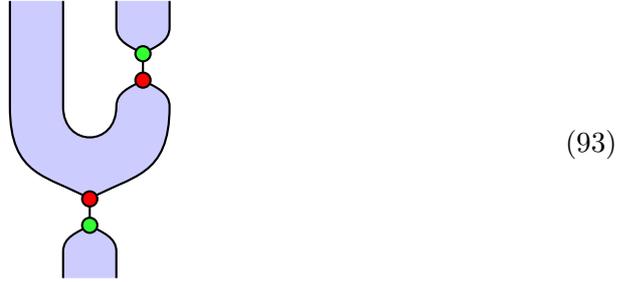
\begin{equation}
\label{eq:erasureprotocol}
\begin{aligned}
\begin{tikzpicture}[scale=\compscale, thick]
\node (a) [Vertex, vertex colour=red] at (0.5,0.25  ) {};
\node (b) [Vertex, vertex colour=red] at (1.5, 2.5) {};
\node (c) [Vertex, vertex colour=green] at (1.5, 3.00) {};
\node (d) [Vertex, vertex colour=green] at (0.5, -0.25) {};
\draw (0.5,-0.25 -| a.center) to (a.center);
\draw [fill=\fillcomp] (-1,4)
    to (-1,2)
    to [out=down, in=\nwangle, out looseness=1.5] (a.center)
    to [out=\neangle, in=down, in looseness=1.5] (2,2)
    to [out=up, in=\seangle] (b.center)
    to [out=\swangle, in=up] +(-0.5,-0.5)
    to [out=down, in=down, looseness=2] +(-1,0)
    to (0.0,4);
\draw (b.center)
    to [out=up, in=down] (c.center);
\draw [fill=\fillcomp] (1,4) to (1,3.5)
    to [out=down, in=\nwangle] (c.center)
    to [out=\neangle, in=down] +(0.5,0.5)
    to (2,4);
\draw [fill=\fillcomp] (0,-1.25)
    to (0,-0.75)
    to [out=up, in=\swangle] (d.center)
    to [out=\seangle, in=up] (1,-0.75)
    to (1,-1.25);
\end{tikzpicture}
\end{aligned}
\end{equation}
The procedure begins with a source of classical information. We then prepare a quantum system whose state witnesses this data, using the green measurement basis. This is then measured in another basis, drawn in red, which we take to be complementary to the green basis. The measurement result is copied, and one of these copies is used to prepare a new quantum system, again using the red basis. Finally, this new quantum system is re-measured in the green basis.

What is the result of this procedure? By complementarity, the result of the final measurement should be uncorrelated classically with the original classical information that existed at the beginning of the protocol. This can be shown by applying the physical complementarity condition~\eqref{eq:physicallycomplementary}, allowing us to rewrite our original expression~\eqref{eq:erasureprotocol} as follows:
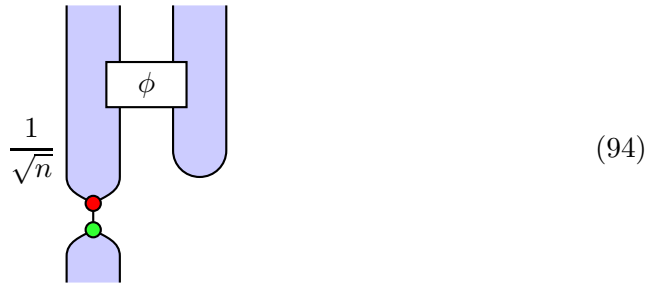
\begin{equation}
\frac{1}{\sqrt{n}}\,
\begin{aligned}
\begin{tikzpicture}[scale=\compscale, thick]
\node (a) [Vertex, vertex colour=red] at (-0.5,2) {};
\draw (-0.5,1.5) to (a.center);
\node (d) [Vertex, vertex colour=green] at (-0.5, 1.5) {};
\draw [fill=\fillcomp] (-1,5.75)
    to (-1,2.5)
    to [out=down, in=\nwangle] (a.center)
    to [out=\neangle, in=down] (0,2.5)
    to +(0,3.25);
\draw [fill=\fillcomp] (1,5.75)
    to (1,3)
    to [out=down, in=left] (1.5,2.5)
    to [out=right, in=down] +(0.5,0.5)
    to +(0,2.75);
\node [minimum width=30pt, draw, fill=white] at (0.5,4.25) {$\phi$};
\draw [fill=\fillcomp] (-1,0.5)
    to (-1,1)
    to [out=up, in=\swangle] (d.center)
    to [out=\seangle, in=up] (0,1)
    to (0,0.5);
\end{tikzpicture}
\end{aligned}
\end{equation}
The initial classical data is uncorrelated with the classical data in the top-right region since the diagrams are disconnected, apart from the application of a global phase.

However, we now modify our protocol by \textit{erasing} the result of the red measurement after everything else has been completed. We do this by adding an erasure vertex~\eqref{eq:forget} at the top-left of expression~\eqref{eq:erasureprotocol}. This gives the following result:
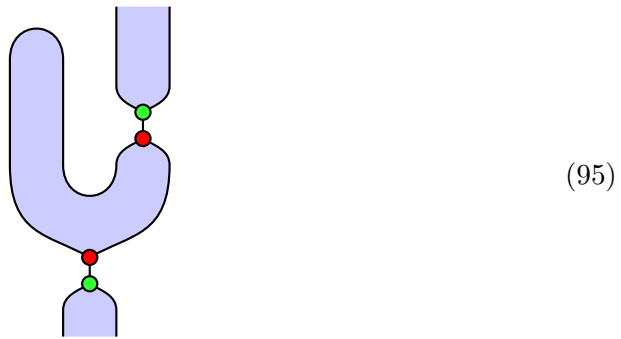
\begin{equation}
\label{eq:seconderasureprotocol}
\begin{aligned}
\begin{tikzpicture}[scale=\compscale, thick]
\node (a) [Vertex, vertex colour=red] at (0.5,0.25  ) {};
\node (b) [Vertex, vertex colour=red] at (1.5, 2.5) {};
\node (c) [Vertex, vertex colour=green] at (1.5, 3.00) {};
\node (d) [Vertex, vertex colour=green] at (0.5, -0.25) {};
\draw (0.5,-0.25 -| a.center) to (a.center);
\draw [fill=\fillcomp] (0,4) to [out=up, in=up, looseness=2] (-1,4)
    to (-1,2)
    to [out=down, in=\nwangle, out looseness=1.5] (a.center)
    to [out=\neangle, in=down, in looseness=1.5] (2,2)
    to [out=up, in=\seangle] (b.center)
    to [out=\swangle, in=up] +(-0.5,-0.5)
    to [out=down, in=down, looseness=2] +(-1,0)
    to (0.0,4);
\draw (b.center)
    to [out=up, in=down] (c.center);
\draw [fill=\fillcomp] (1,5) to (1,3.5)
    to [out=down, in=\nwangle] (c.center)
    to [out=\neangle, in=down] +(0.5,0.5)
    to (2,5);
\draw [fill=\fillcomp] (0,-1.25)
    to (0,-0.75)
    to [out=up, in=\swangle] (d.center)
    to [out=\seangle, in=up] (1,-0.75)
    to (1,-1.25);
\end{tikzpicture}
\end{aligned}
\end{equation}
By the topological behaviour of classical information in our formalism, and the inverseness of measurement and preparation, we can simplify this as follows:
\begin{equation}
\label{eq:erasureproof}
\begin{aligned}
\begin{tikzpicture}[scale=\compscale, thick]
\node (a) [Vertex, vertex colour=red] at (0.5,0.25  ) {};
\node (b) [Vertex, vertex colour=red] at (1.5, 2.5) {};
\node (c) [Vertex, vertex colour=green] at (1.5, 3.00) {};
\node (d) [Vertex, vertex colour=green] at (0.5, -0.25) {};
\draw (0.5,-0.25 -| a.center) to (a.center);
\draw [fill=\fillcomp] (0,4) to [out=up, in=up, looseness=2] (-1,4)
    to (-1,2)
    to [out=down, in=\nwangle, out looseness=1.5] (a.center)
    to [out=\neangle, in=down, in looseness=1.5] (2,2)
    to [out=up, in=\seangle] (b.center)
    to [out=\swangle, in=up] +(-0.5,-0.5)
    to [out=down, in=down, looseness=2] +(-1,0)
    to (0.0,4);
\draw (b.center)
    to [out=up, in=down] (c.center);
\draw [fill=\fillcomp] (1,5) to (1,3.5)
    to [out=down, in=\nwangle] (c.center)
    to [out=\neangle, in=down] +(0.5,0.5)
    to (2,5);
\draw [fill=\fillcomp] (0,-1.25)
    to (0,-0.75)
    to [out=up, in=\swangle] (d.center)
    to [out=\seangle, in=up] (1,-0.75)
    to (1,-1.25);
\end{tikzpicture}
\end{aligned}
\quad=\quad
\begin{aligned}
\begin{tikzpicture}[scale=\compscale, thick]
\node (a) [Vertex, vertex colour=red] at (1.5,0.25  ) {};
\node (b) [Vertex, vertex colour=red] at (1.5, 2.5) {};
\node (c) [Vertex, vertex colour=green] at (1.5, 3.00) {};
\node (d) [Vertex, vertex colour=green] at (1.5, -0.25) {};
\draw (0.5,-0.25 -| a.center) to (a.center);
\draw [fill=\fillcomp] (a.center)
    to [out=\neangle, in=down, in looseness=1.5] (2,0.75)
    to (2,2)
    to [out=up, in=\seangle] (b.center)
    to [out=\swangle, in=up] +(-0.5,-0.5)
    to (1.0,0.75)
    to [out=down, in=\nwangle] (a.center);
\draw (b.center)
    to [out=up, in=down] (c.center);
\draw [fill=\fillcomp] (1,5) to (1,3.5)
    to [out=down, in=\nwangle] (c.center)
    to [out=\neangle, in=down] +(0.5,0.5)
    to (2,5);
\draw [fill=\fillcomp] (1,-1.25)
    to (1,-0.75)
    to [out=up, in=\swangle] (d.center)
    to [out=\seangle, in=up] (2,-0.75)
    to (2,-1.25);
\end{tikzpicture}
\end{aligned}
\quad=\quad
\begin{aligned}
\begin{tikzpicture}[scale=\compscale, thick]
\node (c) [Vertex, vertex colour=green] at (1.5, 3.00) {};
\node (d) [Vertex, vertex colour=green] at (1.5, -0.25) {};
\draw (0.5,-0.25 -| a.center) to (a.center);
\draw (d.center) to (c.center);
\draw [fill=\fillcomp] (1,5) to (1,3.5)
    to [out=down, in=\nwangle] (c.center)
    to [out=\neangle, in=down] +(0.5,0.5)
    to (2,5);
\draw [fill=\fillcomp] (1,-1.25)
    to (1,-0.75)
    to [out=up, in=\swangle] (d.center)
    to [out=\seangle, in=up] (2,-0.75)
    to (2,-1.25);
\end{tikzpicture}
\end{aligned}
\quad=\quad
\begin{aligned}
\begin{tikzpicture}[scale=\compscale, thick]
\draw [fill=\fillcomp, draw=none] (0,0) rectangle (1,6.25);
\draw (0,0) to (0,6.25);
\draw (1,0) to (1,6.25);
\end{tikzpicture}
\end{aligned}
\end{equation}
The result of the experiment has changed completely. Having performed the erasure, the final measurement outcome is identical to the original classical information, instead of uncorrelated to it as before. This demonstrates the counter-intuitive effects of quantum erasure in a fully graphical way.

\subsubsection*{Implementation}

We verify this computationally by demonstrating that, with the choice of the red vertex as the computational basis and the green vertex as the plus/minus basis, the first and last expressions of equation~\eqref{eq:erasureproof}
are equal. The calculation can be found in the \textit{Mathematica} notebook~\cite{v12-highernotebook}.

\subsection{More general protocols}
\label{sec:generalized}

\subsubsection*{Introduction}

An important feature of our graphical approach is that it allows us to write down general quantum informatic procedures in a precise way. Quite easily, specifications can be given which do not correspond to anything described in the literature, and it is in general a difficult problem to decide whether these specifications can actually be implemented.

The problem of finding new protocols becomes a question of pure mathematics in our new setting. Since every diagram corresponds to a physically-implementable quantum procedure, every valid graphical identity in \cat{2Hilb} corresponds to a specification of some quantum task, in the most general sense  of two different physical procedures which have the same effect when implemented on quantum systems.

\subsubsection*{Interlaced  teleportation}

Here we describe a generalization of quantum teleportation. It can be thought of as the simultaneous execution of two separate teleportation procedures, but acting on the same quantum system, with the measurement and unitary correction steps of the two protocols interlaced in a particular way.  This is its graphical specification:
\begin{equation}
\begin{aligned}
\begin{tikzpicture}
    \node [Vertex, vertex colour=black] (B) at (1.4,0.2) {};
    \node [Vertex, vertex colour=red] (D) at (0.5,1) {};
\begin{pgfonlayer}{foreground}
    \node (g) [minimum width=0.8cm, minimum height=0.45cm, fill=white, draw=black, thick, inner sep=0pt]
        at (0.5,1) {$g$};
\end{pgfonlayer}
\begin{scope}[xshift=0.2cm]
\begin{scope}[yshift=-1cm]
    \node [Vertex, vertex colour=white] (A) at (-1.4,0.2) {};
    \node [Vertex] (C) at (-0.5,1) {};
\begin{pgfonlayer}{foreground}
    \node (f) [minimum width=0.8cm, minimum height=0.45cm, fill=white, draw=black, thick, inner sep=0pt]
        at (-0.5,1) {$f$};
\end{pgfonlayer}
\node (E2) at (f.-140 -| -2,1) {};
\end{scope}
\node (E) at (-2,1.8) {};
\node (F) at (f.140 |- E.center) {};
\end{scope}
\node (G) at (g.40 |- E.center) {};
\node (H) at (2,2 |- E.center) {};
\node (H2) at (g.-40 -| 2,1) {};
\node (I) at (g.140 |- E.center) {};
\node (J) at (-0.5,-2.3) {};
\draw [thick, fill=\fillA, draw=black] (E.center)
    to (E2.center)
    to [out=down, in=\nwangle] (A.center)
    to [out=\neangle, in=down] (f.-140) to (f.140)
    to [out=up, in=down] (F.center);
\draw [thick, fill=\fillA, draw=black] (H.center)
    to (H2.center)
    to [out=down, in=\neangle] (B.center)
    to [out=\nwangle, in=down] (g.-40) to (g.40) 
    to [out=up, in=down] (G.center);
\begin{scope}[thick]
    \draw (J.center)
        to [out=up, in=down] (A.s1);
    \draw (A.s2)
        to [out=down, in=100] (0.22,-1.6)
        to [out=-80, in=left] (0.7,-2)
        to [out=right, in=down] (B.s2);
    \draw (B.s1) to [out=down, in=up] (1.1,-0.5)
        to [out=down, in=right] (0.75,-1.75)
        to [out=left, in=down] ([yshift=-0.5cm, xshift=0.4cm] f.-40)
        to [out=up, in=down] (f.-40);
    \draw (f.40)
        to [out=up, in=down] (g.-140);
    \draw (g.140)
        to [out=up, in=down] (I.center);
\end{scope}
\end{tikzpicture}
\end{aligned}
\quad=\quad
\begin{aligned}
\begin{tikzpicture}
\begin{scope}[yshift=-0.5cm]
\node (C) at (-0.8,1) {};
\node (E2) at (-2,1) {};
\node (D) at (0.8,1) {};
\node (H2) at (2,1) {};
\end{scope}
\node (E) at (-2,1.8) {};
\node (F) at (-0.8,2 |- E.center) {};
\node (G) at (0.8,2 |- E.center) {};
\node (H) at (2,2 |- E.center) {};
\node (I) at (0,2 |- E.center) {};
\node (J) at (0,-2.3) {};
\draw [thick, fill=\fillA, draw=black] (E.center)
    to (E2.center)
    to [out=down, in=down, looseness=2] (C.center)
    to (F.center);
\draw [thick, fill=\fillA, draw=black] (H.center)
    to (H2.center)
    to [out=down, in=down, looseness=2] (D.center)
    to (G.center);
\begin{scope}[thick]
    \draw (J.center) to (I.center);
\end{scope}
\end{tikzpicture}
\end{aligned}
\end{equation}
Step by step, the left-hand side of this equation describes the following procedures.
\begin{enumerate}
\item
We begin with a single quantum system $S$, which we refer to as system 1.
\item
Four new instances of $S$ are then prepared, referred to as copies 2, 3, 4 and 5, such the pairs $(2,5)$ and $(3,4)$ are in a Bell state.
\item
Copies 1 and 2 of the system are then measured in some multi-partite basis~A.
\item
Simultaneously, copies 4 and 5 are also measured in some multi-partite basis~B.
\item
Copy 3 then has a correction performed on it, depending on the result of the A-basis measurement.
\item
Finally, another correction is applied to the same qubit, this time depending on the result of the B-basis measurement.
\end{enumerate}
The protocol is successful if this procedure has the same result as that described on the right-hand side of the equation: the initial system $S$ being unaffected, and its state uncorrelated to the results of the two measurements that were performed.

This seems a perfectly reasonable quantum-informatic task. However, by a result of Fong~\cite{f12-it}, it cannot be achieved. Since our specifications remain unchanged by topological deformation of the classical information, this failure of implementability is a \emph{topological invariant} of our specification. We are left with a compelling open question: can implementability of graphical specifications be deduced by solely topological means?

\appendix
\section{Modelling information transfer to the\\environment}
\label{sec:environmental}

\subsection{Introduction}

In this Appendix, we describe an alternative presentation of the 2\-category \cat{2Hilb} from an explicitly physical perspective. Our model is built around quantum systems, their dynamics, their environment, and the interactions between these. When talking in general terms, we will use $S$ and $E$ to refer to our system and its environment respectively, and will also let these symbols stand for their Hilbert spaces of quantum states where appropriate.

The model does not describe arbitrary interactions between a system and its environment, but only those which cause information to be `transferred' in a particularly straightforward way. Using the terminology of the \emph{decoherence} programme, this is comparable to the situation when environmental interactions select out \emph{robust states} of a quantum systems, with classical properties (see~\cite{s05-decoherence} for a survey.) This is a scenario in which the notion of `classical data' extracted by the environment can most clearly be formalized, and so it is a good foundation on which to build our model.

We omit proofs of lemmas in this section, since they are all straightforward, and identical to proofs in the standard theory of modules, albeit couched in unusual physical terminology.

\subsection{Classical data types}
\label{sec:classicaldatatypes}

\noindent
A \emph{classical data type} is a Hilbert space $V$ equipped with \textit{copying} and \textit{deleting} maps
\begin{align}
    \delta &\colon V \to V \otimes V
    \\
    \epsilon &\colon V \to \mathbb{C}
\end{align}
satisfying associativity, unit, and commutativity laws:
\begin{gather}
\label{eq:assoc}
(\delta \otimes \id _V) \circ \delta = (\id _V \otimes \delta) \circ \delta
\\
\label{eq:unit}
(\epsilon \otimes \id _V) \circ \delta = \id _V = (\id _V \otimes \epsilon) \circ \delta
\\
\label{eq:comm}
\text{swap} _{V,V} \circ \delta = \delta
\end{gather}
Here we make use of the map
\[
\text{swap} _{V,V} \colon V \otimes V \to V \otimes V,
\]
which  we define by its action $\ket\phi \otimes \ket\psi \mapsto \ket\psi \otimes \ket\phi$
on product states. A common name for this structure  is a \emph{commutative coalgebra} (or \textit{commutative comonoid}).

In our model, we assume that the environment $E$ is given the structure of a  classical data type. Physically, this is motivated  by the idea that the environment stores certain kinds of information about our quantum system in a highly redundant manner, yielding these effective copying and deleting structures.

We will use a graphical notation to describe our algebraic operations~\cite{js91-gtc, s11-sgl}. In this notation vertical lines represent Hilbert spaces, and vertices represent linear maps. Horizontal juxtaposition represents the tensor product operation, and vertical juxtaposition represents composition. Using this notation our copying and deleting maps have the following form, which should be read from bottom to top:
\def\myscale{0.2}
\def\myYscale{0.4}
\begin{equation}
\begin{aligned}
\begin{tikzpicture}[xscale=\myscale,yscale=\myYscale]
\node (m) at (0,0) {};
\node (tl) at (-1,1) {};
\node (tr) at (1,1) {};
\node (b) at (0,-1) {};
\draw (b.center) to (m.center) to [out=left, in=down] (tl.center);
\draw (m.center) to [out=right, in=down] (tr.center);
\begin{scope}[xshift=10cm]
    \node (b2) at (0,-1) {};
    \node [draw, circle, fill, inner sep=2pt] (m2) at (0,0) {};
    \draw (b2.center) to (m2.center) {};
\end{scope}
\end{tikzpicture}
\end{aligned}
\end{equation}
Our axioms~(\ref{eq:assoc}--\ref{eq:comm}) are then depicted in the following way:
\begin{equation}
\label{eq:graphicalcomonoid}
\begin{aligned}
\begin{tikzpicture}[xscale=\myscale,yscale=\myYscale]
\node (m) at (0.25,-1) {};
\node (tl) at (-2,1) {};
\node (tr) at (1.5,1) {};
\node (b) at (0.25,-2) {};
\node (tm) at (0,1) {};
\node (l) at (-1,0) {};
\node (r) at (1.5,0) {};
\draw (b.center) to (m.center) to [out=left, in=down] (l.center) to [out=left, in=down] (tl.center);
\draw (m.center) to [out=right, in=down] (r.center) to (tr.center);
\draw (l.center) to [out=right, in=down] (tm.center);
\end{tikzpicture}
\end{aligned}
=
\begin{aligned}
\begin{tikzpicture}[xscale=-\myscale,yscale=\myYscale]
\node (m) at (0.25,-1) {};
\node (tl) at (-2,1) {};
\node (tr) at (1.5,1) {};
\node (b) at (0.25,-2) {};
\node (tm) at (0,1) {};
\node (l) at (-1,0) {};
\node (r) at (1.5,0) {};
\draw (b.center) to (m.center) to [out=left, in=down] (l.center) to [out=left, in=down] (tl.center);
\draw (m.center) to [out=right, in=down] (r.center) to (tr.center);
\draw (l.center) to [out=right, in=down] (tm.center);
\end{tikzpicture}
\end{aligned}
\hspace{15pt}
\begin{aligned}
\begin{tikzpicture}[xscale=\myscale,yscale=\myYscale]
\node (m) at (0,-1) {};
\node (tr) at (1,1) {};
\node (b) at (0,-2) {};
\node [draw, circle, fill, inner sep=2pt] (l) at (-1,0) {};
\node (r) at (1,0) {};
\draw (b.center) to (m.center) to [out=left, in=down] (l.center);
\draw (m.center) to [out=right, in=down] (r.center) to (tr.center);
\end{tikzpicture}
\end{aligned}
=
\begin{aligned}
\begin{tikzpicture}[xscale=\myscale,yscale=\myYscale]
\node (t) at (0,1) {};
\node (b) at (0,-2) {};
\draw (b.center) to (t.center);
\end{tikzpicture}
\end{aligned}
=
\begin{aligned}
\begin{tikzpicture}[xscale=-\myscale,yscale=\myYscale]
\node (m) at (0,-1) {};
\node (tr) at (1,1) {};
\node (b) at (0,-2) {};
\node [draw, circle, fill, inner sep=2pt] (l) at (-1,0) {};
\node (r) at (1,0) {};
\draw (b.center) to (m.center) to [out=left, in=down] (l.center);
\draw (m.center) to [out=right, in=down] (r.center) to (tr.center);
\end{tikzpicture}
\end{aligned}
\hspace{15pt}
\begin{aligned}
\begin{tikzpicture}[xscale=\myscale,yscale=\myYscale]
\node (m) at (0,0) {};
\node (tl) at (-1,1) {};
\node (tr) at (1,1) {};
\node (b) at (0,-1) {};
\node (ttl) at (-1,2) {};
\node (ttr) at (1,2) {};
\draw (b.center) to (m.center) to [out=left, in=down] (tl.center) to [out=up, in=down, looseness=0.5] (ttr.center);
\draw (m.center) to [out=right, in=down] (tr.center) to [out=up, in=down, looseness=0.5] (ttl.center);
\end{tikzpicture}
\end{aligned}
=
\begin{aligned}
\begin{tikzpicture}[xscale=\myscale,yscale=\myYscale]
\node (m) at (0,0) {};
\node (tl) at (-1,1) {};
\node (tr) at (1,1) {};
\node (b) at (0,-1) {};
\node (ttl) at (1,2) {};
\node (ttr) at (-1,2) {};
\draw (b.center) to (m.center) to [out=left, in=down] (tl.center) to [out=up, in=down, looseness=0.5] (ttr.center);
\draw (m.center) to [out=right, in=down] (tr.center) to [out=up, in=down, looseness=0.5] (ttl.center);
\end{tikzpicture}
\end{aligned}
\end{equation}
This graphical form of the axioms has the advantage of being more immediately comprehensible than the traditional algebraic presentation.

We will further require our classical data type to be a  \textit{special \dag\-Frobenius algebra}, also called a \emph{classical structure} by Coecke and Pavlovic~\cite{cp06-qmws}. The axioms for this involve the adjoints $\delta^\sdag: V \otimes V \to V$ and $\epsilon^\sdag : \C \to V$ of our copying and deleting maps, which are drawn in the following way:
\begin{equation}
\begin{aligned}
\begin{tikzpicture}[xscale=\myscale,yscale=-\myYscale]
\node (m) at (0,0) {};
\node (tl) at (-1,1) {};
\node (tr) at (1,1) {};
\node (b) at (0,-1) {};
\draw (b.center) to (m.center) to [out=left, in=down] (tl.center);
\draw (m.center) to [out=right, in=down] (tr.center);
\begin{scope}[xshift=10cm]
    \node (b2) at (0,-1) {};
    \node [draw, circle, fill, inner sep=2pt] (m2) at (0,0) {};
    \draw (b2.center) to (m2.center) {};
\end{scope}
\end{tikzpicture}
\end{aligned}
\end{equation}
The classical structure axioms then take the following form:
\begin{equation}
\label{eq:graphicalfrobenius}
\begin{aligned}
\begin{tikzpicture}[scale=0.6]
\node (a) at (0.1,0) {};
\node (b) at (1.0,0)  {};
\node (c) at (0.5,1) {};
\node (d) at (1,0.5) {};
\node (e) at (0.5,1.5) {};
\node (f) at (1.4,1.5) {};
\draw (a.center)
    to [out=up, in=left] (c.center)
    to [out=right, in=left] (d.center)
    to [out=right, in=down] (f.center);
\draw (c.center) to (e.center);
\draw (b.center) to (d.center);
\end{tikzpicture}
\end{aligned}
=
\begin{aligned}
\begin{tikzpicture}[xscale=0.6, yscale=-0.6]
\node (a) at (0.1,0) {};
\node (b) at (1.0,0)  {};
\node (c) at (0.5,1) {};
\node (d) at (1,0.5) {};
\node (e) at (0.5,1.5) {};
\node (f) at (1.4,1.5) {};
\draw (a.center)
    to [out=up, in=left] (c.center)
    to [out=right, in=left] (d.center)
    to [out=right, in=down] (f.center);
\draw (c.center) to (e.center);
\draw (b.center) to (d.center);
\end{tikzpicture}
\end{aligned}
\hspace{20pt}
\begin{aligned}
\begin{tikzpicture}[scale=0.6]
\node (a) at (0,0) {};
\node (b) at (0,0.4)  {};
\node (c) at (0,1.1) {};
\node (d) at (0,1.5) {};
\draw (a.center)
    to (b.center)
    to [out=left, in=left, looseness=1.4] (c.center)
    to (d.center);
\draw (b.center)
    to [out=right, in=right, looseness=1.4] (c.center);
\end{tikzpicture}
\end{aligned}
=
\begin{aligned}
\begin{tikzpicture}[scale=0.6]
\node (a) at (0,0) {};
\node (d) at (0,1.5) {};
\draw (a.center)
    to (d.center);
\end{tikzpicture}
\end{aligned}
\end{equation}
\noindent
These axioms ensure that $V$ is finite-dimensional, that $\delta$ acts by  copying the elements of some orthonormal basis of $V$, and that $\epsilon$ acts by deleting them~\cite{cpv08-dfb}:
\begin{align}
\delta &\colon \ket{i} \mapsto \ket{i} \otimes \ket{i}
\\
\epsilon &\colon \ket{i} \mapsto 1
\end{align}
This is a significant restriction, but one which is appropriate for describing quantum informatic phenomena. It would be interesting to consider more general classical data types,  especially for modelling \emph{continuous} data; this would require dropping the classical structure axioms~\eqref{eq:graphicalfrobenius} above.

\subsection{Environmental interactions}
\label{sec:environmentalinteractions}
\def\diagscale{1}

\noindent
We now suppose that our system $S$ and environment $E$ interact, in such a way that information about the state of the system is transferred to the environment. To formalize this notion, we consider  linear maps of the form
\begin{equation}
\tau \colon S \to E \otimes S,
\end{equation}
having the following graphical representation:
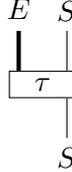
\tikzset{env/.style={ultra thick}}
\[
\begin{aligned}
\begin{tikzpicture}[yscale=\diagscale]
\node [comod] (f) at (0,0) {$\tau$};
\draw [env] (f.150) to +(0,15pt) node (A) [above] {$E$};
\draw (f.330) to +(0,-15pt) node (B) [below] {$S$};
\draw (f.30) to +(0,15pt) node (C) [above] {$S$};
\end{tikzpicture}
\end{aligned}
\]
We use a thicker pen to draw the curve representing the environmental system, so it be easily distinguished.

We now equip our environment $E$ with the structure of a classical data type $(E,\delta,\epsilon)$. We say that  $\tau \colon S \to E \otimes S$ is an \emph{interaction map} if the following axioms hold:
\extralabel{eq:comod1}
\extralabel{eq:comod2}
\[
\tag{\ref{eq:comod1},\ref{eq:comod2}}
\begin{array}{c}
\begin{aligned}
\begin{tikzpicture}[yscale=\diagscale]
\node [comod] (f) at (0,0) {$\tau$};
\draw (f.30) to +(0,15pt) node (C) [above] {};
\node (f2) [comod, anchor=330] at (C.south) {$\tau$};
\draw (f2.30) to +(0,15pt) node (C2) [above=0pt] {$S$};
\draw [env] (f2.150) to +(0pt,15pt) node (A2) [above] {$E$};
\draw [env] (f.150)
    to [out=up, in=down] +(-20pt,15pt) node (A) [above] {}
    to ([xshift=-20pt] A2.south) node [above] {$E$};
\draw (f.330) to +(0,-15pt) node (B) [below] {$S$};
\end{tikzpicture}
\end{aligned}
=
\hspace{-10pt}
\begin{aligned}
\begin{tikzpicture}[yscale=\diagscale]
\node [comod] (f) at (0,0) {$\tau$};
\draw (f.30)
    to +(0,42.9pt)
        node (C) [above] {$S$};
\draw [env] (f.150)
    to [out=up, in=down] +(-10pt,18pt)
        node (A) [above] {}
    to [out=left, in=down] ([xshift=-20pt] A2.south) node [above] {$E$};
\draw (f.330)
    to +(0,-15pt)
        node (B) [below] {$S$};
\draw [env] (A.south)
    to [out=right, in=down] (A2.south)
        node [above] {$E$};
\end{tikzpicture}
\end{aligned}
\end{array}
\hspace{20pt}
\begin{array}{c}
\begin{aligned}
\begin{tikzpicture}[yscale=\diagscale]
\node [comod] (f) at (0,0) {$\tau$};
\draw [env] (f.150) to +(0,15pt) node (A) [draw, circle, fill, inner sep=2pt] {};
\draw (f.330) to +(0,-20pt) node (B) [below] {$S$};
\draw (f.30) to +(0,35pt) node (C) [above] {$S$};
\end{tikzpicture}
\end{aligned}
=
\begin{aligned}
\begin{tikzpicture}[yscale=\diagscale]
\draw (0,0) node [below] {$S$} to (0,68pt) node (C) [above] {$S$};
\end{tikzpicture}
\end{aligned}
\end{array}
\]
If $\tau$ has these properties, then it can be thought of as transferring data of type $(E,\delta,\epsilon)$ to the environment. This interpretation arises from the first axiom above: in words, extracting information twice from $S$ is the same as extracting it once and copying the result. We also have a non-disturbance property, thanks to the second axiom: the action of $\tau$ can be undone by a linear map which acts solely on the environmental degrees of freedom. The axioms we have given here exactly match the conventional mathematical notion of a \emph{comodule} for a comonoid.

We will be interested in the more general situation where data is transferred to more than one type of environmental system. Suppose that we have two such interaction maps, $\tau\colon S \to E \otimes S$ and $\tau'\colon S \to E' \otimes S$, where both $E$ and $E'$ are equipped separately with the structure of a classical data type. Physically, we imagine that these environmental interactions are occurring in some `uncontrollable' fashion, transferring data about the system $S$ to both environments arbitrarily often. The only case in which the resulting information flow is well-defined will be when $\tau$ and $\tau'$ commute, in the following sense:
\begin{equation}
\label{eq:commutinginteraction}
\begin{aligned}
\begin{tikzpicture}[yscale=\diagscale]
\node [comod,minimum height=12pt] (f) at (0,0) {$\tau$};
\draw (f.35) to +(0,15pt) node (C) [above] {};
\node (f2) [comod, anchor=325, minimum height=12pt] at (C.south) {\raisebox{-7.4pt}{\smash{$\tau'$}}};
\draw (f2.35) to +(0,25pt) node (C2) [above] {$S$};
\node (P) at ([yshift=5pt] f2.145) {};
\draw [env] (f2.145)
    to (P.center)
    to [out=up, in=down] ([xshift=-20pt,yshift=20pt] P.center)
        node (A2) [above] {$E'$};
\draw [env] (f.145)
    to [out=up, in=down] +(-20pt,15pt)
        node (A) [above] {}
    to ([xshift=-20pt] P.center)
    to [out=up, in=down] ([yshift=20pt] P.center)
        node [above] {$E$};
\draw (f.325) to +(0,-15pt) node (B) [below] {$S$};
\end{tikzpicture}
\end{aligned}
\hspace{7pt}
=
\begin{aligned}
\begin{tikzpicture}[yscale=\diagscale]
\node [comod,minimum height=12pt] (f) at (0,0) {\raisebox{-7.4pt}{\smash{$\tau'$}}};
\draw (f.35) to +(0,15pt) node (C) [above] {};
\node (f2) [comod, anchor=325, minimum height=12pt] at (C.south) {$\tau$};
\draw (f2.35) to +(0,25pt) node (C2) [above] {$S$};
\node (P) at ([yshift=5pt] f2.145) {};
\draw [env] (f2.145)
    to [out=up, in=down] ([yshift=20pt] P.center)
        node (A2) [above] {$E$};
\draw [env] (f.145)
    to [out=up, in=down] +(-20pt,15pt)
        node (A) [above] {}
    to ([xshift=-20pt,yshift=20pt] P.center)
        node [above] {$E'$};
\draw (f.325) to +(0,-15pt) node (B) [below] {$S$};
\end{tikzpicture}
\end{aligned}
\end{equation}
This commutativity condition can be readily extended to the situation of more than two interaction maps, in which case all pairs are required to commute.

An interesting example of an interaction map arises in the case that the quantum system is the \textit{same} as the environment. In this case, we can choose the copying map $\delta \colon E \to E \otimes E$ to itself be our interaction map.
\begin{lemma}
\label{lem:copyisinteraction}
For a classical data type $(E,\delta,\epsilon)$, the map $\delta:E \to E \otimes E$ satisfies the axiom of an interaction map.
\end{lemma}

\noindent
It is reassuring that this is possible, as it emphasizes that despite the conceptual split between system and environment that lies at the heart of this framework, the environment can itself be treated as a system equipped with a canonical interaction.

\subsection{Simultaneous interactions}
\label{sec:simultaneousmeasurements}

\noindent
We now consider the possibility that two quantum systems $S$ and $S'$ interact simultaneously with the \textit{same} environmental system $E$. In this situation, correlations are induced between states of the two system, since if  one measures the value of the environmental classical data type, one obtains a restriction on the joint state of $S \otimes S'$ to those states compatible with this value.

Suppose that our two environmental interaction maps for systems $S$ and $S'$ are
\begin{align*}
\tau &\colon S\phantom{'} \hspace{-2pt} \to E \otimes S
\\
\tau' &\colon S' \hspace{-2pt} \to E \otimes S'
\end{align*}
We want to calculate those joint states of $S \otimes S'$ which are \textit{consistent}, in the sense that they transfer the same classical data to the environment under the application of $\tau$ or $\tau'$. We call this the tensor product of $S$ and $S'$ \emph{with respect to the environment E}, written $S \otimes _E S'$. Suppose that $\psi \in S \otimes S'$ is such a joint state. Then we can express this consistency property by the following equation:
\begin{equation}
\label{eq:consistencyproperty}
\begin{aligned}
\begin{tikzpicture}[yscale=\diagscale]
    \draw (0,0) -- (27pt,0pt) -- (13.5pt,-16pt) -- cycle;
    \node [comod,minimum height=12pt, anchor=-35] (f)
        at (4pt,15pt) {$\tau$};
    \draw (f.-35) to (4pt,0pt);
    \draw (f.35) to +(0pt,20pt) node [above] {$S$};
    \draw [env] (f.145) to +(0pt,20pt) node [above] {$E$};
    \draw (23pt,0pt) to (23pt,47.3pt) node [above] {$S'$};
    \node at (13.5pt,-7pt) {$\psi$};
\end{tikzpicture}
\end{aligned}
=
\begin{aligned}
\begin{tikzpicture}[yscale=\diagscale]
    \draw (0,0) -- (27pt,0pt) -- (13.5pt,-16pt) -- cycle;
    \node [comod,minimum height=12pt, anchor=-35] (f)
        at (23pt,15pt) {\raisebox{-7.4pt}{\smash{$\tau'$}}};
    \draw (f.-35) to (23pt,0pt);
    \draw (f.35) to +(0pt,20pt) node [above] {$S'$};
    \draw [env] (f.145)
        to [out=up, in=down]
            ([xshift=-20pt, yshift=20pt] f.145)
            node [above] {$E$};
    \node at (13.5pt,-7pt) {$\psi$};
    \draw (4pt,0pt)
        to [out=up, in=down] ([xshift=-20pt] f.-145)
        to [out=up, in=down] ([xshift=-20pt] f.145)
        to [out=up, in=down] ([yshift=20pt] f.145)
            node [above] {$S$};
\end{tikzpicture}
\end{aligned}
\end{equation}
We wish to find the subspace of $S \otimes S'$ spanned by such states. Writing this subspace as $S \otimes _E S'$, we can construct it as the following equalizer:
\vspace{-8pt}
\begin{equation}
\label{eq:equalizer}
\begin{aligned}
\begin{tikzpicture}[xscale=1]
\node (T) at (0,0) {$S \otimes_E S'$};
\node (A) at ([xshift=0.7cm] T.east) [anchor=west] {$S \otimes S'$};
\node (B) at ([xshift=2.4cm] A.east) [anchor=west] {$E \otimes S \otimes S'$};
\draw [right hook->] (T) to node [above] {$e$} (A);
\draw [->] (A.10)
    to [out=25, in=155, looseness=0.5]
    node [above] {$\tau \otimes \id _{S'}$}
    (B.170);
\draw [->] (A.-10)
    to [out=-25, in=-155, looseness=0.5]
    node [below] {$(\text{swap}_{E,S} \otimes \id _{S'}) \circ (\id _E \otimes \tau')$}
    (B.-170);
\end{tikzpicture}
\end{aligned}
\end{equation}
\vspace{-8pt}

\noindent
We choose the embedding $e$ to be an isometry. The state space $S \otimes _E S'$ can itself be given a canonical environmental interaction map, as the following composite:
\begin{equation}
\label{eq:compositeinteractionmap}
S \otimes _E S' \xto{\textstyle e}
S \otimes S'
\xto{\tau \otimes \id _{S'}}
E \otimes S \otimes S'
\xto{\id_E \otimes e ^\sdag}
E \otimes (S \otimes _E S')
\end{equation}
\begin{lemma}
The composite~\eqref{eq:compositeinteractionmap} satisfies the axioms of an interaction map.
\end{lemma}

\ignore{
Above, we described how the environment can itself be treated as a system in our framework. We can now show that copies of this environmental system can be freely produced. Suppose that we have a system $S$ equipped with environmental interaction $\tau \colon S \to E \otimes S$, and the environment $E$ treated as a system with environmental interaction $\delta \colon E \to E \otimes E$. Then we have two systems interacting with the same environment, and so the tensor product $S \otimes _E E$ of these systems with respect to the environment can be constructed. We now consider the following lemma:
\begin{lemma}
Given interaction maps $\tau:S \to E \otimes S$ and $\delta : E \to E \otimes E$, the tensor product with respect to the environment $S \otimes_E E$ is isomorphic as a Hilbert space to $S$.
\end{lemma}

\noindent
Therefore the presence of this extra copy of the environment results in an equivalent physical system to $S$ itself. Conversely, any number of copies of the environment can be considered as being simultaneously present. This fits very well with the physical intuition driving our model, and is crucial for our 2\-categorical framework to be introduced in Section~\ref{sec:2cat}.
}

\subsection{Protected dynamics}
\label{sec:protecteddynamics}

\noindent
The final layer of structure to add to our model is a notion of dynamics for our quantum systems. Since our systems are interacting with their environment in an uncontrollable fashion, the dynamics of our system will only be well-defined if it commutes with the interaction maps in an appropriate fashion.

Suppose $f:S \to S'$ is the linear map representing dynamical evolution. We assume that our dynamics is \textit{local}, and so $S$ and $S'$ interact with the same family of environments. We say that $f$ is \emph{protected from decoherence}, or simply \emph{protected}, when the following equation holds:
\begin{equation}
\label{eq:protected}
\begin{aligned}
\begin{tikzpicture}[yscale=\diagscale]
\node [comod,minimum height=12pt] (c) at (0,0) {\raisebox{-7.4pt}{\smash{$\tau'$}}};
\draw [env] (c.145) to +(0,15pt) node (A) [above] {$E$};
\draw (c.325) to +(0,-15pt) node [circle, draw, inner sep=1pt, fill=white] {\small $f$} to +(0,-30pt) node (B) [below] {$S$};
\draw (c.35) to +(0,15pt) node [above] {$S'$};
\end{tikzpicture}
\end{aligned}
\hspace{8pt}
=
\hspace{5pt}
\begin{aligned}
\begin{tikzpicture}[yscale=\diagscale]
\node [comod, minimum height=12pt] (c) at (0,0) {$\tau$};
\draw [env]  (c.145) to +(0,30pt) node (A) [above] {$E$};
\draw (c.325) to +(0,-15pt) node (B) [below] {$S$};
\draw (c.35) to +(0,15pt) node [circle, draw, inner sep=1pt, fill=white] {\small $f$} to +(0,30pt) node [above] {$S'$};
\end{tikzpicture}
\end{aligned}
\end{equation}
We require this to hold for each environment $E$ with which the systems $S$ and $S'$ interact.

For a given system, precisely which dynamics are protected from decoherence will depend on the environmental interaction maps themselves. However, it is possible to set these up in such a way that \textit{any} evolution of the system will be protected.

One way is to set the environment to be the trivial classical data type $(\C,1,1)$, and the interaction map as the canonical isomorphism of vector spaces $\tau\colon S \to \C \otimes S$. A more realistic way to ensure arbitrary dynamics are `protected' is to maintain a nontrivial environmental classical data type $(E,\delta,\epsilon)$, but to choose the `system' Hilbert space to be $B \otimes S$, where $B$ is a `buffer' system which interacts directly with the environment and `protects' our true system $S$ of interest. In this scenario, we require $B$ to be equipped with its own interaction map $\tau _{B}\colon B \to E \otimes B$. The interaction map $\tau_{B\otimes S}\colon (B \otimes S) \to E \otimes (B \otimes S)$ is then defined as the tensor product of $\tau _B$ with the identity on $S$:
\begin{equation}
\label{eq:tensorinteraction}
\begin{aligned}
\begin{tikzpicture}[yscale=\diagscale]
\node [comod] (f) at (0,0) {$\tau _B$};
\draw  [env] (f.150) to +(0,15pt) node (A) [above] {$E$};
\draw (f.330) to +(0,-15pt) node (B) [below] {$B$};
\draw (f.30) to +(0,15pt) node (C) [above] {$B$};
\draw ([xshift=20pt] B.north)
    node [below] {$S$}
    to ([xshift=20pt] C.south)
    node [above] {$S$};
\end{tikzpicture}
\end{aligned}
\end{equation}
The following lemma establishes that this construction is valid.
\begin{lemma}
\label{lem:compoundinteraction}
For an interaction map \mbox{$\tau_B: B \to E \otimes B$}, then $\tau_B \otimes \id_S : B \otimes S \to E \otimes (B \otimes S)$ is also an interaction map.
\end{lemma}

\noindent
We now see that this achieves our goal.

\begin{lemma}
\label{lem:protectedlemma}
For any linear map $f:S \to S'$, the composite $\id_B \otimes f:B \otimes S \to B \otimes S'$ is a protected linear map, with respect to the interaction maps described in Lemma~\ref{lem:compoundinteraction}.
\end{lemma}

\noindent
This construction is a good model for a quantum system $S$ placed inside a box $B$, which isolates it completely: while $B$ interacts with the environment, $S$ does not, and quantum evolution can take place without any information about $S$ being transmitted to the environment.

\subsection{Controlled operations}
\label{sec:2controlledoperations}

\noindent
We saw in Lemma~\ref{lem:copyisinteraction} that the environment can itself be viewed as a system equipped with an interaction map. Using construction~\eqref{eq:tensorinteraction}, we may therefore treat $E \otimes S$ itself as a system,  for an arbitrary Hilbert space $S$, with an interaction map that copies the data stored in $E$. As we saw in Lemma~\ref{lem:protectedlemma}, any linear map $\id_E \otimes f$ for $f\colon S \to S$ is protected with respect to this interaction.

However, our quantum system as a whole is now $E \otimes S$ in this new setting, and it is interesting to ask what the protected evolutions are of this composite system. Applying definition~\eqref{eq:protected}, they are maps $f\colon E \otimes S \to E \otimes S$ satisfying the following equation:
\begin{equation}
\begin{aligned}
\begin{tikzpicture}[yscale=\diagscale]
\node [comod] (f) at (0,0) {$f$};
\node (P) at ([xshift=-9pt, yshift=15pt] f.145) {};
\node (A) at ([xshift=-18pt, yshift=30pt] f.145) [above] {$E$};
\node (B) at ([xshift= 0pt , yshift=30pt] f.145) [above] {$E$};
\node (C) at ([yshift=30pt] f.35) [above] {$S$};
\draw (f.35)
    to (C.south);
\draw (f.145)
    to [out=up, in=down] (P.center)
    to [out=right, in=down] (B.south);
\draw (f.-35)
    to +(0,-15pt)
        node [below] {$S$};
\draw (f.-145)
    to +(0,-15pt)
        node [below] {$E$};
\draw [env] (P.center)
    to [out=left, in=down] (A.south);
\end{tikzpicture}
\end{aligned}
\hspace{5pt}
=
\begin{aligned}
\begin{tikzpicture}[yscale=\diagscale]
\node [comod] (f) at (0,0) {$f$};
\node (P) at ([xshift=-9pt, yshift=-15pt] f.-145) {};
\node (A) at ([xshift=-18pt, yshift=15pt] f.145) [above] {$E$};
\node (B) at ([xshift= 0pt , yshift=15pt] f.145) [above] {$E$};
\node (C) at ([yshift=15pt] f.35) [above] {$S$};
\draw (f.35)
    to (C.south);
\draw (f.-145)
    to [out=down, in=right] (P.center);
\draw [env] (P.center)
    to [out=left, in=down] ([xshift=-18pt] f.-145)
    to (A.south);
\draw (f.-35)
    to +(0,-30pt)
        node [below] {$S$};
\draw (f.145)
    to +(0,+15pt);
\draw (P.center)
    to [out=down, in=up]
        ([yshift=-30pt,xshift=-9pt] f.-145) node [below] {$E$};
\end{tikzpicture}
\end{aligned}
\end{equation}
For concreteness, we now suppose that the copying map associated to  $E$ acts by copying a basis of vectors $\ket{i} \in E$, where $i \in \{1, \ldots, \dim(E)\}$ acts as an index. Then
\begin{align*}
\delta &\colon \ket{i} \mapsto \ket{i} \otimes \ket {i}
\\
\epsilon &\colon \ket{i} \mapsto 1,
\end{align*}
and we can demonstrate that $f$ is precisely a series of \textit{controlled operations}: operations on $S$ which depend on the value stored in $E$.
\begin{lemma}
For a system with Hilbert space $E \otimes S$, classical data type $(E,\delta,\epsilon)$ as above, and interaction map \mbox{$\delta \otimes \id_S: E \otimes S \to E \otimes E \otimes S$}, the protected maps \mbox{$f: E \otimes  S \to E \otimes S$} have the following form:
\begin{equation}
f \colon \ket{i} \otimes \ket{\sigma} \mapsto \ket{i} \otimes f_i(\ket{\sigma})
\end{equation}
where $\ket{\sigma}$ is an arbitrary state of $S$, and  $f_i \colon S \to S$ is an indexed family of linear maps on $S$.
\end{lemma}

\bibliography{../../../jov}

\begin{thebibliography}{10}

\bibitem{ac04-csqp}
Samson Abramsky and Bob Coecke.
\newblock A categorical semantics of quantum protocols.
\newblock {\em Proceedings of the 19th Annual IEEE Symposium on Logic in
  Computer Science}, pages 415--425, 2004.
\newblock IEEE Computer Science Press.

\bibitem{ac08-cqm}
Samson Abramsky and Bob Coecke.
\newblock {\em Handbook of Quantum Logic and Quantum Structures}, volume~2,
  chapter Categorical Quantum Mechanics.
\newblock Elsevier, 2008.

\bibitem{b-twf}
John~C Baez.
\newblock This week's finds in mathematical physics.
\newblock http://math.ucr.edu/home/baez/TWF.html.

\bibitem{b97-hda2}
John~C Baez.
\newblock Higher-dimensional algebra {II}: 2-{H}ilbert spaces.
\newblock {\em Advances in Mathematics}, 127:125--189, 1997.

\bibitem{b97-inc}
John~C Baez.
\newblock An introduction to $n$-categories.
\newblock {\em Lecture Notes in Computer Science}, 1290:1--33, 1997.

\bibitem{b99-hda}
John~C Baez.
\newblock Higher-dimensional algebra and {P}lanck-scale physics.
\newblock 1999.

\bibitem{bbfw08-idr}
John~C Baez, Aristide Baratin, Laurent Freidel, and Derek~K Wise.
\newblock Infinite-dimensional representations of 2-groups.

\bibitem{bd95-hda}
John~C Baez and James Dolan.
\newblock Higher-dimensional algebra and topological quantum field theory.
\newblock {\em Journal of Mathematical Physics}, 36:6073--6105, 1995.

\bibitem{bd98-cat}
John~C Baez and James Dolan.
\newblock Categorification.
\newblock In Ezra Getzler and Mikhail Kapranov, editors, {\em Higher Category
  Theory}, volume 230 of {\em Contemporary Mathematics}, pages 1--36. American
  Mathematical Society, 1998.

\bibitem{bd01-fsfd}
John~C Baez and James Dolan.
\newblock {\em Mathematics Unlimited --- 2001 and Beyond}, chapter ``From
  finite sets to {F}eynman diagrams'', pages 29--50.
\newblock Springer, Berlin, 2001.

\bibitem{bhw09-hda7}
John~C Baez, Alexander~E Hoffnung, and Christopher~D Walker.
\newblock Higher-dimensional algebra {VII}: {G}roupoidification.
\newblock 2009.

\bibitem{bl11-pnc}
John~C Baez and Aaron Lauda.
\newblock A prehistory of $n$-categorical physics.
\newblock In Hans Halvorson, editor, {\em Deep Beauty: Understanding the
  Quantum World through Mathematical Innovation}. Cambridge University Press,
  2011.

\bibitem{b93-teleport}
Charles~H. Bennett, Gilles Brassard, Claude Cr\'epeau, Richard Jozsa, Asher
  Peres, and William~K. Wootters.
\newblock Teleporting an unknown quantum state via dual classical and
  {{E}instein-{P}odolsky-{R}}osen channels.
\newblock {\em Physical Review Letters}, 70(13):1895--1899, 1993.

\bibitem{c03-loe}
Bob Coecke.
\newblock The logic of entanglement: An invitation.
\newblock Technical report, University of Oxford, 2003.
\newblock Computing Laboratory Research Report PRG-RR-03-12.

\bibitem{c08-cpwp}
Bob Coecke.
\newblock Complete positivity without positivity and without compactness.
\newblock Technical Report PRG-RR-07-05, Oxford University Computing
  Laboratory, 2007.

\bibitem{cd07-gcnc}
Bob Coecke and Ross Duncan.
\newblock A graphical calculus for non-commuting quantum observables.
\newblock 2007.
\newblock In preparation.

\bibitem{cdkw12-scnl}
Bob Coecke, Ross Duncan, Aleks Kissinger, and Quanlong Wang.
\newblock Strong complementarity and non-locality in categorical quantum
  mechanics.
\newblock {\em Proceedings of the 27th Annual {IEEE} Symposium on Logic in
  Computer Science}, 2012.
\newblock IEEE Computer Science Press, to appear.

\bibitem{ch11-pcp}
Bob Coecke and Chris Heunen.
\newblock Pictures of complete positivity in arbitrary dimension.
\newblock In {\em Proceedings of QPL VIII, ENTCS}, 2011.
\newblock To appear.

\bibitem{cpp09-cqs}
Bob Coecke, Eric~Oliver Paquette, and Dusko Pavlovic.
\newblock Classical and quantum structuralism.
\newblock In Simon Gay and Ian Mackie, editors, {\em Semantic Techniques in
  Quantum Computation}, pages 29--69. Cambridge University Press, 2010.

\bibitem{cp06-qmws}
Bob Coecke and Dusko Pavlovic.
\newblock {\em The Mathematics of Quantum Computation and Technology}, chapter
  Quantum Measurements Without Sums.
\newblock Taylor and Francis, 2006.

\bibitem{cpv08-dfb}
Bob Coecke, Dusko Pavlovic, and Jamie Vicary.
\newblock Commutative dagger-{F}robenius algebras in {FdHilb} are orthogonal
  bases.
\newblock ({RR}-08-03), 2008.
\newblock Technical Report.

\bibitem{cp08-ecc}
Bob Coecke and Simon Perdrix.
\newblock Environment and classical channels in categorical quantum mechanics.
\newblock In {\em Proceedings of the 19th {EACSL} Annual Conference on Computer
  Science Logic}, number 6247 in Lecture Notes in Computer Science, 2010.

\bibitem{cf94-4d}
Louis Crane and Igor Frenkel.
\newblock Four dimensional topological quantum field theory, {H}opf categories,
  and the canonical bases.
\newblock {\em Journal of Mathematical Physics}, 35:5136--5154, 1994.

\bibitem{cy98-eoc}
Louis Crane and David Yetter.
\newblock Examples of categorification.
\newblock {\em Cahiers de Topologie et G\'eom\'etrie Diff\'erentielle
  Cat\'egoriques}, 39(1), 1998.

\bibitem{d10-afu}
David Deutsch.
\newblock {\em Many Worlds}, chapter Apart from Universes, pages 542--552.
\newblock OUP, 2010.

\bibitem{e06-2vect}
Josep Elgueta.
\newblock A strict totally coordinatized version of {K}apranov and
  {V}oevodsky's 2-category {2Vect}.
\newblock {\em Math. Proc. Camb. Phil. Soc.}, 142(407):407--428, 2007.

\bibitem{eno02-ofc}
Pavel Etingof, Dmitri Nikshych, and Viktor Ostrik.
\newblock On fusion categories.
\newblock {\em Annals of Mathematics}, 162(2):581--642, 2005.

\bibitem{f12-it}
Brendan Fong.
\newblock Unpublished note, 2012.

\bibitem{bh10-ihgt}
John~Huerta John C~Baez.
\newblock An invitation to higher gauge theory.
\newblock Based on lectures given at 2nd School and Workshop on Quantum Gravity
  and Quantum Geometry at the 2009 Corfu Summer Institute, 2010.

\bibitem{js91-gtc}
Andr\'e Joyal and Ross Street.
\newblock The geometry of tensor calculus {I}.
\newblock {\em Advances in Mathematics}, 88:55--112, 1991.

\bibitem{l06-faaa}
Aaron~D. Lauda.
\newblock Frobenius algebras and ambidextrous adjunctions.
\newblock {\em Theory and Applications of Categories}, 16(4):84--122, 2006.

\bibitem{l98-bb}
Tom Leinster.
\newblock Basic bicategories.
\newblock Unpublished notes, 1998.

\bibitem{l09-otc}
Jacob Lurie.
\newblock On the classification of topological quantum field theories.
\newblock {\em Current Developments in Mathematics}, 2008:129--280, 2009.

\bibitem{m06-caqm}
Jeffrey Morton.
\newblock Categorified algebra and quantum mechanics.
\newblock {\em Theory and Applications of Categories}, 16(29):785--854, 2006.

\bibitem{o03-mc}
Victor Ostrik.
\newblock Module categories, weak {H}opf algebras and modular invariants.
\newblock {\em Transformation Groups}, 3(2):177--206, 2003.

\bibitem{r11-2vect}
Dan Roberts.
\newblock Representing modular tensor categories: A computer algebra system for
  topological quantum computing.
\newblock Master's thesis, Department of Computer Science, University of
  Oxford, 2011.

\bibitem{s05-decoherence}
Maximilian Schlosshauer.
\newblock Decoherence, the measurement problem, and interpretations of quantum
  mechanics.
\newblock {\em Rev. Mod. Phys.}, 76(4):1267--1305, 2005.

\bibitem{sp11-thesis}
Christopher~John Schommer-Pries.
\newblock {\em The classification of two-dimensional extended topological field
  theories}.
\newblock 2009.
\newblock Thesis (Ph.D.)--University of California, Berkeley.

\bibitem{s07-dccc}
Peter Selinger.
\newblock Dagger compact closed categories and completely positive maps.
\newblock In {\em Proceedings of the 3rd International Workshop on Quantum
  Programming Languages (QPL 2005)}, 2007.

\bibitem{s11-sgl}
Peter Selinger.
\newblock {\em New Structures for Physics}, chapter A Survey of Graphical
  Languages for Monoidal Categories, pages 289--355.
\newblock Number 813 in Lecture Notes in Physics. Springer, 2011.

\bibitem{v12-highernotebook}
Jamie Vicary.
\newblock Mathematica notebook to accompany \emph{Higher Quantum Theory}.
  Available at \texttt{http://www.cs.ox.ac.uk/people/jamie.vicary/}, 2012.

\bibitem{w01-tdc}
Reinhard Werner.
\newblock All teleportation and dense coding schemes.
\newblock {\em Journal of Physics A}, 34(35):7081--7094, 2001.

\end{thebibliography}

\end{document}